\documentclass[reqno,10pt, centertags,draft]{amsart}
\usepackage{amsmath,amsthm,amscd,amssymb,latexsym,upref}
\usepackage{mathrsfs}

\makeatletter
\def\theequation{\@arabic\c@equation}

\newcommand{\e}{\hbox{\rm e}}

\newcommand{\gaD}{\gamma_{{}_D}}
\newcommand{\gaN}{\gamma_{{}_N}}
\newcommand{\aT}{\mathfrak{a}_\Theta}

\newcommand{\Mf}{\mathfrak{m}}
\newcommand{\Mas}{\operatorname{Mas}}

\newcommand{\dist}{\operatorname{dist}}

\newcommand{\bbR}{{\mathbb{R}}}

\newcommand{\bbC}{{\mathbb{C}}}

\newcommand{\cB}{{\mathcal B}}

\newcommand{\cD}{{\mathcal D}}

\newcommand{\cF}{{\mathcal F}}
\newcommand{\cG}{{\mathcal G}}
\newcommand{\cH}{{\mathcal H}}

\newcommand{\cK}{{\mathcal K}}
\newcommand{\cL}{{\mathcal L}}
\newcommand{\cM}{{\mathcal M}}
\newcommand{\cN}{{\mathcal N}}

\newcommand{\cP}{{\mathcal P}}

\newcommand{\cS}{{\mathcal S}}

\newcommand{\cU}{{\mathcal U}}

\newcommand{\cX}{{\mathcal X}}

\newcommand{\no}{\nonumber}
\newcommand{\lb}{\label}

\newcommand{\ol}{\overline}

\newcommand{\Om}{\Omega}
\newcommand{\dOm}{{\partial\Omega}}

\newcommand{\tr}{\operatorname{Tr}}
\newcommand{\gr}{\operatorname{Gr}}

\newcommand{\ran}{\text{\rm{ran}}}

\newcommand{\codim}{\text{\rm{codim}}}
\newcommand{\dom}{\text{\rm{dom}}}

\newcommand{\bi}{\bibitem}
\newcommand{\sgn}{\operatorname{sign}}

\renewcommand{\Re}{\operatorname{Re}}
\renewcommand{\Im}{\operatorname{Im}}

\numberwithin{equation}{section}

\newtheorem{theorem}{Theorem}[section]
\newtheorem{lemma}[theorem]{Lemma}
\newtheorem{proposition}[theorem]{Proposition}
\newtheorem{corollary}[theorem]{Corollary}
\theoremstyle{definition}
\newtheorem{definition}[theorem]{Definition}
\newtheorem{claim}[theorem]{Claim}
\newtheorem{hypothesis}[theorem]{Hypothesis}
\newtheorem{remark}[theorem]{Remark}
\newtheorem{example}[theorem]{Example}

\newcommand{\beq}{\begin{equation}}
\newcommand{\enq}{\end{equation}}
\newcommand{\bbox}[1]{\vspace{20pt}\fbox{\parbox{300pt}{{\bf #1}}}\vspace{20pt}}
\newcommand{\vertiii}[1]{{\left\vert\kern-0.25ex\left\vert\kern-0.25ex\left\vert #1
    \right\vert\kern-0.25ex\right\vert\kern-0.25ex\right\vert}}
    
    \newcommand{\Fr}{\operatorname{F_{\rm{red}}}}
     \newcommand{\FLr}{\operatorname{F\Lambda_{\rm{red}}}}
       \newcommand{\mor}{\operatorname{Mor}}
        \newcommand{\mas}{\operatorname{Mas}}
        \newcommand{\Sp}{\operatorname{Sp}}

\begin{document}

\title[The Morse and Maslov indices]{The Morse and Maslov indices for\\  multidimensional Schr\"odinger operators with matrix-valued potentials}

\allowdisplaybreaks

\author[G. Cox]{Graham Cox}
\address{Mathematics Department,
The University of North Carolina at Chapel Hill Chapel Hill, NC 27599, USA}
\email{ghcox@email.unc.edu}
\author[C.\ Jones]{Christopher K.\ R.\ T. Jones}
\address{Mathematics Department,
The University of North Carolina at Chapel Hill Chapel Hill, NC 27599, USA}
\email{ckrtj@email.unc.edu}
\author[Y. Latushkin]{Yuri Latushkin}
\address{Department of Mathematics,
The University of Missouri, Columbia, MO 65211, USA}
\email{latushkiny@missouri.edu}
\author[A. Sukhtayev]{Alim Sukhtayev}
\address{Department of Mathematics, Texas A\&M University, \\
College Station, TX 77843-3368, USA}
\email{alim@math.tamu.edu}
\thanks{Partially supported by the grants NSF DMS-0754705, DMS-1067929, DMS-0410267, DMS-1312906 and ONR N00014-05-1-0791, and by the Research Council and the Research
Board of the University of Missouri. Yuri Latushkin sincerely thanks Lai-Sang Young for the opportunity to spend his sabbatical at the Courant Institute where this paper was completed}

\begin{abstract}
We study the Schr\"odinger operator $L=-\Delta+V$ on a star-shaped domain $\Om$ in $\bbR^d$ with Lipschitz boundary $\dOm$. The operator is equipped with quite general Dirichlet- or Robin-type boundary conditions induced by operators between $H^{1/2}(\dOm)$ and $H^{-1/2}(\dOm)$, and the potential takes values in the set of symmetric $N\times N$ matrices. By shrinking the domain and rescaling the operator we obtain a path in the Fredholm--Lagrangian Grassmannian of the subspace of $H^{1/2}(\dOm)\times H^{-1/2}(\dOm)$ corresponding to the given boundary condition. The path is formed by computing the Dirichlet and Neumann traces of weak solutions to the rescaled eigenvalue equation. We prove a formula relating  the number of negative eigenvalues of $L$ (the Morse index), the signed crossings of the path (the Maslov index), the number of negative eigenvalues of the potential matrix evaluated at the center of the domain, and the number of negative eigenvalues of a bilinear form related to the boundary operator.
\end{abstract}

\date{\today}

\maketitle

{\scriptsize{\tableofcontents}}
\normalsize

\section{Introduction}

The foundational Morse Index Theorem, a generalization of a classical result of Sturm, relates the Morse and Maslov  indices of selfadjoint ordinary differential operators. The Morse index counts the negative eigenvalues of the operator together with their multiplicities while the Maslov index counts the conjugate points of the respective differential equation. This theorem and its various generalizations play a fundamental role in many questions from the calculus of variations \cite{M63} to geometry 
\cite{SW08} and the stability of traveling waves \cite{CDB06,CDB09,CDB11,J88,SS08}, and have attracted much attention over the years \cite{B56,D76,CLM94,FJN03,rs93,RoSa95,U73}. 
Another remarkable generalization of the Sturm Theorem was given by Arnold \cite{arnold67,Arn85} who considered the case of {\em systems}, that is, matrix-valued ordinary differential equations.
Although the first version of the Morse Index Theorem for partial differential equations was given by Smale \cite{S65} in 1965, surprisingly little is currently known, even for scalar equations \cite{DJ11,J,Sw}. In particular, we are not aware of any papers on the Morse Index Theorem for systems of partial differential equations beyond the Dirichlet case treated by Smale. 

In the current work we fill this gap. In particular, we continue the work initiated in \cite{DJ11} for  scalar-valued potentials and smooth domains, and describe a relation between the Morse index and the Maslov index for matrix-valued Schr\"odinger operators on multidimensional, star-shaped, Lipschitz domains. We pay special attention to the role of the Dirichlet-to-Neumann and Neumann-to-Dirichlet maps and the weak Neumann trace operators, and fully develop the functional analytic aspects of the theory, refining and adapting to the problem at hand some methods from perturbation theory \cite{Kato}.


The original purpose of this work 
was to complete some arguments in \cite{DJ11} and the companion papers \cite{DN06,DN08}; see the list provided at the end of the Introduction for details. Along the way, we observed new matrix-valued contributions to the Morse Index Theorem, such as the 
Morse index of the potential matrix which appears in Theorem \ref{tim:Nbased}.


For scalar-valued Schr\"odinger operators on a smooth,
star-shaped domain $\Om$, the main idea in \cite{DJ11} was to shrink $\Om$ to its subdomains $\Om_t = \{tx : x \in \Omega\}$, and
thus obtain a path $t\mapsto\Upsilon(t)$ in the set of Lagrangian planes of the symplectic Hilbert space 
$H^{1/2}(\dOm)\times H^{-1/2}(\dOm)$, formed by the boundary traces of weak $H^1(\Om_t)$-solutions to the respective homogenous partial differential equations. The Malsov index of this path was related in \cite{DJ11} to the Morse index of the original Schr\"odinger operator on $L^2(\Om)$. This result has generated a recent flurry of activity on the subject of Morse and Maslov indices for differential operators.

In \cite{JLM13}  a similar approach was used to relate the Morse and Maslov indices of a one-dimensional Schr\"odinger operator with periodic potential and $\theta$-periodic boundary conditions. In \cite{DP} the results of \cite{DJ11} were extended to general second-order, scalar-valued elliptic operators on star-shaped domains.  The paper \cite{CJM} 
dealt with 
the case when $\Om$ is not necessarily star-shaped, and the domains $\Om_t$ are obtained from $\Om$ using a general family of diffeomorphisms instead of the linear scaling $x \mapsto tx$. 
In \cite{PW} and \cite{LSS} a different symplectic construction was given, utilizing the ``space of abstract boundary values" for a given symmetric operator, in which Lagrangian subspaces correspond to selfadjoint extensions---see \cite{BF98} and the literature cited therein.  This formulation allowed \cite{PW} to study the more involved case (on a star-shaped domain) when the underlying differential operator  is not bounded from below, in which case the Morse index must be replaced by the spectral flow.  For the semi-bounded case \cite{LSS} proved a general theorem relating the Morse and Maslov indices for a family of abstract self-adjoint operators, which contains as a special case Schr\"odinger operators on $\bbR^d$ with periodic potentials and 
quasi-periodic boundary conditions.


The current paper, while closely related to the recent work cited above, is distinguished by its emphasis on the functional analytic (as opposed to geometric or variational) aspects of the problem, and hence provides a view of the developments in \cite{DJ11} complementary to those found in \cite{CJM} and \cite{LSS}.


Let $\Omega\subset\bbR^d$ be a star-shaped domain with Lipschitz boundary $\partial\Omega$, and consider the following eigenvalue problem
\begin{align}\label{BVP1}
-\Delta u&+V(x)u=\lambda u,\,x\in\Omega,\lambda\in\bbR,\\
&\tr u\in\cG.\label{BVP2}
\end{align}
Here $u:\Omega\to\bbR^N$ is a vector-valued function\footnote{Throughout the paper we suppress vector notations as much as possible, and consistently write $L^2(\Omega)$ instead of $L^2(\Omega;\bbR^N)=\big(L^2(\Omega)\big)^N$, $C^0(\Omega)$ instead of $C^0(\Omega;\bbR^{N\times N})$ etc.} in the real Sobolev space $H^1(\Omega)=H^1(\Omega;\bbR^N)$, the potential $V=V(\cdot)\in C^0(\overline{\Omega};\bbR^{N\times N})$ is a continuous, matrix-valued function having symmetric values, $V(x)^\top=V(x)$, and $\Delta=\partial_{x_1}^2+\dots+\partial_{x_d}^2$ is the Laplacian. We denote by $\tr$ the trace operator acting from $H^1(\Om)$ into the boundary space $\cH$,
\begin{equation}\label{dfTr}
\tr\colon H^1(\Omega)\to \cH, \,
u\mapsto (\gaD u,\gaN u),\quad \cH=H^{1/2}(\partial\Omega)\times H^{-1/2}(\partial\Omega),
\end{equation}
where $\gaD u=u\big|_{\partial\Omega}$ is the Dirichlet trace and $\gaN u=\nu\cdot\nabla u\big|_{\partial\Omega}$ is the (weak) Neumann trace, which will be properly defined in Section
\ref{sec:EllPr}.
Throughout, the Laplacian is understood in the weak sense, i.e. as a map from $H^1(\Om)$ into $H^{-1}(\Om)=\big(H^1_0(\Om)\big)^\ast$.

The boundary condition in \eqref{BVP2} is determined by a given closed linear subspace $\cG$ of the boundary space $\cH=H^{1/2}(\partial\Omega)\times H^{-1/2}(\partial\Omega)$. 
We will assume that $\cG$ is Lagrangian with respect to the symplectic form $\omega$ defined on $\cH$ by
\begin{equation}\label{dfnomega}
\omega\big((f_1,g_1),(f_2,g_2)\big)=\langle g_2,f_1\rangle_{1/2}-\langle g_1,f_2\rangle_{1/2},
\end{equation} where $\langle g,f\rangle_{1/2}$ denotes the action of the functional $g\in H^{-1/2}(\partial\Omega)$ on the function $f\in H^{1/2}(\partial\Omega)$. 
As in  \cite{DJ11}, we will use the following terminology.
\begin{definition}\label{defDiNeB}
We say that the subspace $\cG$ or the boundary condition in \eqref{BVP2} is {\em Dirichlet-based} if $\cG$ is the inverse  graph \[\gr'(\Theta')=\big\{
(\Theta' g,g)\in\cH: g\in H^{-1/2}(\partial\Omega)\big\}\] of a compact, selfadjoint  operator $\Theta'\colon H^{-1/2}(\partial\Omega)\to H^{1/2}(\partial\Omega)$. In particular, $\Theta'=0$ yields the Dirichlet boundary condition with $\cG=\cH_D$, where we denote
\[\cH_D=\big\{(0,g)\in\cH: g\in H^{-1/2}(\dOm)\big\}.\]
 We say that the subspace $\cG$ or the boundary condition in \eqref{BVP2} is {\em Neumann-based} if $\cG$ is
the graph \[\gr(\Theta)=\big\{
(f,\Theta f)\in\cH: f\in H^{1/2}(\partial\Omega)\big\}\] of a compact, selfadjoint operator $\Theta\colon H^{1/2}(\partial\Omega)\to H^{-1/2}(\partial\Omega)$. In particular, $\Theta=0$ yields the Neumann boundary condition with $\cG=\cH_N$, where we denote
\[\cH_N=\big\{(f,0)\in\cH: f\in H^{1/2}(\dOm)\big\}.\]\hfill$\Diamond$
\end{definition}
The Neumann-based boundary conditions are usually called generalized Robin boundary conditions since \eqref{BVP2} for $\cG=\gr(\Theta)$ can be written as $\gaN u-\Theta\gaD u=0$ in $H^{-1/2}(\dOm)$. If $\Theta$ is the operator of multiplication by a real-valued function $\theta\in L^\infty(\dOm)$ composed with the embedding $H^{1/2}(\dOm)\hookrightarrow H^{-1/2}(\dOm)$, then one obtains the usual Robin boundary conditions, cf., e.g., \cite[Corollary 2.8]{GM08} and the literature reviewed in \cite[Remark 2.9]{GM08}. We refer to Section \ref{sub4.1} for a discussion of further assumptions on the operators $\Theta$ and $\Theta'$ and their consequences.   

The main results of this paper are for generalized Robin (i.e. Neumann-based) boundary conditions. We also prove some results for Dirichlet-based boundary conditions assuming in addition that the operator $\Theta'$ is nonpositive. The precise relation between the Morse and the Maslov indices in the general Dirichlet-based case is still unknown.

We will view the boundary value problem \eqref{BVP1} and \eqref{BVP2} as the eigenvalue problem for the 
Schr\"odinger operator defined on $L^2(\Om)$ as
\beq\lb{dfLG}
L_\cG=-\Delta+V,\quad \dom(L_\cG)=\big\{u\in H^1(\Om): \Delta u\in L^2(\Om), \, \tr u\in\cG\big\}.\enq In particular, $L_{\cH_D}$ and $L_{\cH_N}$ denote the operators equipped with Dirichlet and Neumann boundary conditions, respectively.

We will assume that $L_{\cG}$ has compact resolvent, and thus has no essential spectrum. In this case one defines the Morse index $\mor(L_\cG)$ to be the number of negative eigenvalues of $L_\cG$, counted with multiplicity. 
In Section \ref{sec4} we will give a Lagrangian characterization of the eigenvalues of $L_\cG$ in terms of conjugate times, or crossings. These are points where a nontrivial intersection occurs between $\cG$ and a certain path of subspaces in the Fredholm--Lagrangian Grassmannian of $\cG$ (see Definition \ref{dfnFLG} below). The signed count of these crossings is called the Maslov index; we will recall the precise definition in Section \ref{secMindFrL}. The main objective of the current paper in the Neumann-based case is to relate the Morse index of $L_\cG$, the Morse index of the potential matrix $V(0)$ at the center of the domain, the Morse index of a certain matrix associated with the boundary operator $\Theta$, and the  Maslov index of the above-described path with respect to $\cG$. We also treat the Dirichlet case, which is easier as the Morse index is related merely to the Maslov index, and does not depend on $V(0)$.

Following the strategy of \cite{DJ11}, the path of Lagrangian subspaces will be formed by shrinking the domain $\Om$ and rescaling the boundary value problem \eqref{BVP1} and \eqref{BVP2} accordingly.
Since $\Omega$ is star-shaped, without loss of generality we assume that $0\in\Omega$, and for each $x\in\Omega$ there exist a unique $t\in(0,1]$ and $y\in\partial\Omega$ such that $x=ty$. For each $ t\in(0,1]$ we define a subdomain $\Om_t$ and trace operator $\tr_t:H^1(\Omega)\to\cH$ by
\begin{equation}\label{defOmt}
\Omega_t=\{x\in\Omega: x=t' y\,\text{ for }\, t'\in[0,t),\,y\in\dOm\},\,
\tr_tu:= (\gaD  u,t^{-1}\gaN  u).
\end{equation}
Using the rescaled operator $\tr_t$ and the subspace $\cG$ from \eqref{BVP2}, we will consider the following family of boundary value problems on $\Om$,
\begin{align}\label{tBVP1}
-\Delta u&+t^2V(tx)u=t^2\lambda u,\,x\in\Omega,\,\lambda\in\bbR,\, t\in(0,1],\\
&\tr_tu\in\cG.\label{tBVP2}
\end{align}
The boundary value problem \eqref{tBVP1} and \eqref{tBVP2} appears quite naturally, since passing from functions $u(\cdot)$ on $\Om_t$ to functions $x\mapsto u(tx)$ on $\Om$ allows one to rescale boundary value problems on $\Om_t$ to boundary value problems on $\Om$. Indeed, given a subspace $\cG$ as in \eqref{BVP2}, consider the subspace $\cG_t$ of the boundary space on $\dOm_t$ obtained from $\cG$ by the same rescaling. Then the solutions of \eqref{BVP1} on $\Om_t$ whose traces belong to $\cG_t$ are exactly the rescaled solutions of \eqref{tBVP1} whose traces belong to $\cG$; see Lemma \ref{lemResc} below for more details.

Let $\Sigma=[a,b]$ be a parameter set and let
\begin{equation}\label{defGam}
\Gamma=\big\{(\lambda(s),t(s)): s\in\Sigma\big\}\subset (-\infty,0]\times(0,1]
\end{equation}
 denote the boundary of the square $[-\Lambda,0]\times[\tau , 1]$ for some  (small) $\tau >0$ and (large) $\Lambda=\Lambda(\tau)>0$,
 so that $\Gamma=\cup_{j=1}^4\Gamma_j$ is a continuous, piecewise $C^1$ curve, see \eqref{dfnSigmaj}, \eqref{dfnlambdat} and Figure \ref{fig1}. 
\begin{figure}
\begin{picture}(100,100)(-20,0)
\put(2,0){$-\Lambda$}
\put(10,8){\line(0,1){4}}
\put(80,5){\vector(0,1){95}}
\put(10,20){\line(0,1){60}}
\put(10,80){\vector(0,-1){40}}
\put(5,10){\vector(1,0){95}}
\put(71,40){\text{\tiny $\Gamma_2$}}
\put(12,60){\text{\tiny $\Gamma_4$}}
\put(-32,54){\text{\tiny no conjugate}}
\put(-22,50){\text{\tiny points}}
\put(84,54){\text{\tiny conjugate}}
\put(84,50){\text{\tiny points}}
\put(45,73){\text{\tiny $\Gamma_3$}}
\put(43,14){\text{\tiny $\Gamma_1$}}
\put(100,12){$\lambda$}
\put(83,98){$t$}
\put(80,20){\vector(0,1){30}}
\put(78,0){$0$}
\put(10,20){\line(1,0){70}}
\put(10,20){\vector(1,0){50}}
\put(10,80){\line(1,0){70}}
\put(80,80){\vector(-1,0){55}}
\put(82,78){$1$}
\put(82,18){$\tau $}
\put(70,20){\circle*{4}}
\put(80,60){\circle*{4}}
\put(80,70){\circle*{4}}
\put(20,80){\circle*{4}}
\put(40,80){\circle*{4}}
\put(60,80){\circle*{4}}
\put(18,87){{\tiny \text{$L_\cG$-eigenvalues}}}
\put(15,24){{\tiny \text{$V(0),B$-eigenvalues}}}
\end{picture}
\caption{A possible position of the conjugate times (crossings) in Theorem \ref{tim:Nbased} when $\tau >0$ is small enough for the Neumann-based case. The crossings on a horizontal line $t=\text{const}$ are the eigenvalues of the operator $L_{s,\cG}(\tau)$ for some $s\in\Sigma$ (and in particular of the operator $L_\cG$). For $\tau$ small enough, the number of the negative eigenvalues of the operator $L_{0,\cG}(\tau)$ corresponding to the line $t=\tau$ is computed in Theorem \ref{tim:Nbased} via the number of negative squares of the form $\mathfrak{b}$ induced by the boundary operator $\Theta$ on the space of constant-valued vector functions and via the number of the negative eigenvalues of the value of the potential $V(0)$ restricted to the null space of the form $\mathfrak{b}$.}\label{fig1}
\end{figure}
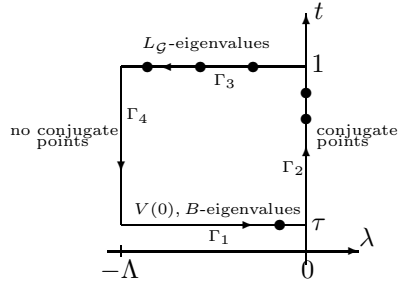 
Throughout, we will use the following notations:
 \begin{align}\label{defVsLs}\begin{split}
 V_s(x)&=t^2(s)V(t(s)x)-t^2(s)\lambda(s),\, x\in\Omega,\\
  L_su&=-\Delta u+V_su,\,u\in H^1(\Omega),\end{split}\\
  T_su&=(\gaD u,(t(s))^{-1}\gaN u),\,s\in\Sigma,\label{defTs}
 \end{align}
where the operator $-\Delta$ is understood in the weak sense, that is, as a map from $H^1(\Omega)$ into $H^{-1}(\Omega)=\big(H^1_0(\Omega)\big)^*$. For $\tau\in(0,1]$ and $s\in\Sigma$ we let $L_{s,\cG}(\tau)$ denote the operator $L_{s,\cG}(\tau)u=L_su$ on $L^2(\Om)$ with the boundary condition $T_su\in\cG$; see \eqref{dfnLst01} and the discussion surrounding this formula for more details. We parametrize $\Gamma$ so that $L_{0,\cG}(1)$ is equal to the operator $L_\cG$ from \eqref{dfLG}.

Our main standing assumptions are summarized as follows; see also Section \ref{sec4}, where we articulate them in more details.
\begin{hypothesis} \lb{h1}
We assume throughout the paper that:
\begin{itemize}\item[(i)]  $\Omega\subset{\bbR}^d$, $d\ge2$, is a nonempty, open, bounded, star-shaped, Lipschitz domain;
\item[(ii)] the subspace $\cG$ in \eqref{BVP2} is Lagrangian with respect to the symplectic form \eqref{dfnomega}, and is either Dirichlet- or Neumann-based;
\item[(iii)] the operators $\Theta$ and $\Theta'$ in Definition \ref{defDiNeB} are compact and selfadjoint, and 
the resulting operators $L_{s,\cG}(\tau)$ (and in particular $L_\cG$) on $L^2(\Om)$ defined in \eqref{dfnLst01}  are selfadjoint, have compact resolvents and, for each $\tau > 0$, are semibounded from below uniformly for $s\in\Sigma$;
\item[(iv)]  the potential $V$ is continuous, $V\in C^0(\overline{\Omega}; \bbR^N\times\bbR^N)$, and $V(x)$ is a symmetric $(N\times N)$ matrix for each $x\in\Omega$.
\end{itemize}
Sometimes assumption (i) will be replaced by a stronger assumption:
\begin{itemize}\item[(i')]  $\Omega\subset{\bbR}^d$, $d\ge2$, is a nonempty, open, bounded, star-shaped domain with $C^{1,r}$ boundary for some $1/2<r<1$.
\end{itemize}
Sometimes we will impose additional assumptions on $\Theta$ and $\Theta'$:
\begin{itemize}\item[(ii')] assumption (ii) holds and the operators $\Theta'$ and $\Theta$ are nonpositive.
\end{itemize}
When needed, in addition to (i) or (i') and (ii) or (ii'), we will also assume that
\begin{itemize}\item[(iii')] assumption (iii) holds and $\dom(L_\cG)$ and $\dom(L_{s,\cG}(\tau))$ are subsets of  $H^2(\Om)$.
\end{itemize}
When needed we will assume that
\begin{itemize}\item[(iv')] assumption (iv) holds and $V\in C^1(\overline{\Omega}; \bbR^N\times\bbR^N)$ is continuously differentiable.
\hfill$\Diamond$
\end{itemize}
\end{hypothesis}

In particular, if Hypothesis \ref{h1}(iii') holds, then each weak solution $u$ to \eqref{BVP1} satisfying the boundary condition \eqref{BVP2} is automatically a strong solution. In Section \ref{sec4} we will give some sufficient conditions on $\Theta$ and $\Theta'$ such that Hypothesis \ref{h1}(iii) or (iii') holds (see  Remarks \ref{sbb} and \ref{DC4.2} and Theorem \ref{GMthmH2}).
 
Let $\tr_tu=(\gaD u,t^{-1}\gaN u)$ be the rescaled trace map and $\cK_t$ be the subspace in $H^1(\Om)$ of weak solutions to the equation $(-\Delta +t^2V(tx))u=0$ for $t\in[\tau ,1]$, with $\tau \in(0,1]$ fixed.
Observe that no boundary conditions are imposed on the functions $u\in\cK_t$. 


As we will see below, our hypotheses imply that the subspaces $\Upsilon(t)=\tr_t(\cK_t)$ form a smooth path in the Fredholm--Lagrangian Grassmanian of the boundary space $\cG$. Therefore, one can define the Maslov index of $\Upsilon(t)$ with respect to $\cG$. Intuitively this is a signed count of the times at which $\Upsilon(t)$ and $\cG$ intersect nontrivially, where the sign depends on the manner in which $\Upsilon(t)$ passes through $\cG$ as $t$ increases.


We are now ready to formulate our main result for the Dirichlet-based case.
\begin{theorem}\label{tim:Dbased} Let $\cG=\gr'(\Theta')$ be a Dirichlet-based subspace of the boundary space $\cH=H^{1/2}(\dOm; \bbR^N)\times H^{-1/2}(\dOm; \bbR^N)$. Assume Hypothesis \ref{h1} (i), (ii'), (iii), (iv), in particular,  that the operator $\Theta'\in\cB(H^{-1/2}(\dOm; \bbR^N), H^{1/2}(\dOm; \bbR^N))$ is nonpositive. 
Then the Morse index of the operator $L_{\cG}$ on $L^2(\Om; \bbR^N)$ defined by $L_{\cG} u=(-\Delta +V(x))u$, $(\gaD u,\gaN u)\in\cG$ and the Maslov index of the path $\Upsilon\colon [\tau ,1]\to F\Lambda(\cG)$ defined by $ t\mapsto \tr_t(\cK_t)$ for $t\in[\tau ,1]$, with $\tau \in(0,1]$ sufficiently small, are related as follows:
\beq\lb{MasvsMorDir}
\mor(L_\cG)=-\mas(\Upsilon).
\enq
Moreover, if Hypothesis \ref{h1} (i'), (iv') hold and $\Theta'=0$ (i.e. we are considering the Dirichlet problem, $\cG=\cH_D$), then
\beq\label{Dfor}
\mor(L_{\cH_D})=\sum_{t\in[\tau,1)}\dim_\bbR\ker\big(-\Delta_{\cH_D}+t^2V(tx)\big),
\enq
where $-\Delta_{\cH_D}$ is the Dirichlet Laplacian.
\end{theorem}
This is a generalization of the celebrated Morse Index Theorem of Smale \cite{S65} (see also \cite{U73}) to Lipschitz domains. It also generalizes \cite[Theorem 2.4]{DJ11} to the matrix-valued case, while doing away with the assumption in \cite{DJ11} that $V(0)$ is sign definite. The last assertion in Theorem \ref{tim:Dbased} is a consequence of general formulas given in Subsection \ref{ss:g2} for the Maslov crossing form at the crossings on $\Gamma_2$, which may be of independent interest.

We now formulate our main result for the Neumann-based case.
 Let $\cS$ denote the $N$-dimensional subspace of $H^{1/2}(\dOm; \bbR^N)$ consisting of boundary values of constant vector-valued  functions in $H^1(\Om; \bbR^N)$. Given a bounded, selfadjoint operator $\Theta$ as in Definition \ref{defDiNeB}, we
let $B$ denote the $(N\times N)$ matrix associated with the quadratic form $\mathfrak{b}(p,q)=\langle \Theta p,q\rangle_{1/2}$ defined for $p,q\in\cS$. Let $Q_0$ denote the orthogonal projection  in $\cS$ on $\ker(B)$.
\begin{theorem}\lb{tim:Nbased}  Let $\cG=\gr(\Theta)$ be a  Neumann-based subspace in the boundary space $\cH=H^{1/2}(\dOm; \bbR^N)\times H^{-1/2}(\dOm; \bbR^N)$.
Assume Hypothesis \ref{h1} (i), (ii), (iii), (iv) and also assume that the quadratic form $\mathfrak{v}(p,q)=\langle V(0)p,q\rangle_{\bbR^N}$ for $p,q\in\ker(B)$ is nondegenerate on $\ran(Q_0)=\ker(B)$ where we use notations introduced in the paragraph preceding the theorem.
 Then the Morse index of the operator $L_{\cG}$ on $L^2(\Om; \bbR^N)$ defined by $L_{\cG} u=(-\Delta +V(x))u$, $(\gaD u,\gaN u)\in\cG$
is given by
\beq\lb{MasvsMor}
\mor(L_\cG)=-\mas(\Upsilon)+\mor(-B)+\mor\big(Q_0V(0)Q_0\big),
\enq
where $\Upsilon\colon [\tau ,1]\to F\Lambda(\cG)$ denotes the path $t\mapsto \tr_t(\cK_t)$ for $t\in[\tau ,1]$ and $V(0)$ is the potential matrix evaluated at $x=0$.
\end{theorem}
The nondegeneracy assumption on the form $\mathfrak{v}(p,q)$ can potentially be removed. However, this requires the use of higher power asymptotic expansions for the eigenvalues of a finite-dimensional part the operator $L_{0,\cG}(\tau)$ as $\tau\to0$, cf.\ formula \eqref{eq1.84} below. Although these expansions are available in the abstract perturbation theory \cite[Section II.5]{Kato}, the formulas are cumbersome and thus we do not pursue this topic here.
In the Neumann case $\Theta=0$, \eqref{MasvsMor} simplifies to
\[\mor(L_\cG)=-\mas(\Upsilon)+\mor\big(V(0)\big)\,\text{ provided $0\notin\Sp(V(0))$.}\]
In the scalar ($N=1$) Neumann-based case  one has $B=\langle \Theta {\mathbf 1}_{\dOm}, {\mathbf 1}_{\dOm}\rangle_{1/2}$, where ${\mathbf 1}_{\dOm}(x)=1$ for $x\in\dOm$, and the nondegeneracy condition in the theorem becomes $V(0)\neq0$ provided $B=0$. One derives from \eqref{MasvsMor} with $N=1$ a  version of \cite[Theorem 2.5]{DJ11} for Lipschitz domains: if $\tau\in(0,1]$ is small enough, then
\begin{equation*}\begin{split}\mor(L_\cG)=\begin{cases}-\mas(\Upsilon)\,&\text{ if $B<0$, or $B=0$ and $V(0)>0$, }\\
-\mas(\Upsilon)+1\,&\text{ if $B>0$, or $B=0$ and $V(0)<0$.}
\end{cases}
\end{split}\end{equation*}

An essential part of the proof of both Theorems \ref{tim:Dbased} and
\ref{tim:Nbased} is the $C^k$ regularity of the path $\Upsilon\colon t\mapsto \tr_t(\cK_t)\in F\Lambda(\cG)$. Our proof of this fact in Proposition \ref{smoothinFLG} is quite elementary and does not use the Calderon projection (see \cite[Chapter 11]{G09} or  \cite[Chapter 12]{BW93}), relying instead on the abstract perturbation Lemmas \ref{lemL6} and \ref{lemAF}. Additionally, in Subsection \ref{sec:DNNDO} we prove that the subspaces $\tr_t(\cK_t)$ can be described as graphs of the Dirichlet-to-Neumann and Neumann-to-Dirichlet maps whenever these are defined.  Another essential part of the proof of Theorem \ref{tim:Nbased} is the asymptotic perturbation theory for the operators $L_{0,\cG}(\tau)$ as $\tau\to0$, corresponding to the lower right corner of the rectangle $\Gamma$ in Figure \ref{fig1}. This theory is fully developed in Section \ref{sec:mainres} and is based on ideas from \cite[Chapters II, VII]{Kato}. Finally, in Section \ref{sec5} we prove several quite explicit formulas for the Maslov crossing form on the right vertical boundary of $\Gamma$.



We conclude the Introduction by listing the results in 
\cite{DJ11} and \cite{DN06,DN08} that required additional analytical arguments, and indicate where these arguments are given. (Note that the major conclusions of \cite{DJ11} are all
correct, as confirmed by the current paper.) 
We note that \cite{DP} also depends on the arguments in \cite{DJ11} (and consequently \cite{DN06,DN08}).

{\bf (i)}\, We reprove Lemmas 3.3 and 3.4 of \cite{DJ11} in Lemma \ref{lem3433}.

{\bf (ii)}\, The full proof of Proposition 4.1 of \cite{DJ11}
(including Lemma 6 of \cite{DN06} and Lemmas 2 and 6 of \cite{DN08}) and Proposition 4.3 of \cite{DJ11} (including Remark 7 in \cite{DN06}) is given in Proposition \ref{smoothinFLG}. This requires several
preliminary steps, which appear in Lemmas \ref{lem:propFG}, \ref{lemL6} and \ref{lemAF}. In fact, the proof of the regularity of the path $\Upsilon\colon [\tau,1]\to F\Lambda(\cG)$ is one of the main technical accomplishments of the current paper.

{\bf (iii)}\, In Section \ref{sec5} we give a detailed proof of Lemma 4.9 (including Claim 4.10) and Proposition 5.3 of \cite{DJ11}.

{\bf (iv)}\, The perturbation results in Lemmas 5.5 and 5.6 and Claim 5.8 of \cite{DJ11} are recovered from the detailed asymptotic analysis in Section \ref{sec:mainres}.

{\bf (v)}\, We observe that the relation between the Morse and Maslov indices in the Neumann-based case requires more careful assumptions than are given in Definition 2.2 of \cite{DJ11}, as demonstrated by Example \ref{ex:nostr} below. In particular, it is necessary to either interpret the Morse index in a weak (that is, $H^1(\Om)$) sense, or impose an additional assumption, such as Hypothesis \ref{h3}, on the boundary operator $\Theta$ in Definition \ref{defDiNeB} to ensure that the domain of the Schr\"odinger operator is contained in $H^2(\Om)$.



{\bf Notations.} Throughout the paper we use the following basic notations.  We denote by $\cB(\cX_1,\cX_2)$ and $\cB_\infty(\cX_1,\cX_2)$ the set of bounded linear operators and the set of compact operators from a Hilbert space $\cX_1$ into a Hilbert space $\cX_2$ (real or complex), and abbreviate these as $\cB(\cX)$ and $\cB_\infty(\cX)$ when $\cX=\cX_1=\cX_2$. We denote by $I_\cX$ the identity operator on $\cX$. Calligraphic letters are used to denote various subspaces. Given two closed linear subspaces $\cL,\cM\subset\cX$, we denote by $\cL+\cM$ their (not necessarily direct) sum, by $\cL\dot{+}\cM$ their direct sum (which need not be orthogonal), and by $\cL\oplus\cM$ their orthogonal sum. For a linear operator $T$ on a Banach space $\cX$ we denote by $T^{-1}$ its (bounded) inverse, by $\ker (T)$ its null space, by $\ran (T)$ its range, by $\Sp(T)$ its spectrum, by $T^*$ its adjoint (or transpose, when the space is real), by $T|_\cL$ its restriction to a subspace $\cL\subset\cX$, and by $T(\cL)=\{Tx:x\in\cL\}$ the range of the restriction. If $\cX$ is a Banach space and $\cX^\ast$ is its adjoint then ${}_{\cX^*}\langle v,u\rangle_{\cX}$ denotes the action of a functional $v\in\cX^*$ on $u\in \cX$. We abbreviate 
\beq\lb{Abb12}\begin{split}
\langle g,f\rangle_{1/2}&={}_{H^{-1/2}(\dOm;\bbR^N)}\langle (g_n)_{n=1}^N,(f_n)_{n=1}^N\rangle_{H^{1/2}(\dOm;\bbR^N)}\\
&\qquad=\sum_{n=1}^N{}_{H^{-1/2}(\dOm)}\langle g_n, f_n\rangle_{H^{1/2}(\dOm)},
\end{split} \enq
and also write 
\begin{equation*}
\langle u,v\rangle_{L^2(\Omega)}= \sum_{n=1}^N\int_\Om u_n(x)v_n(x)\,dx,\, u=(u_n)_{n=1}^N, v=(v_n)_{n=1}^N\in L^2(\Om;\bbR^N),
\end{equation*}
for the scalar product in $L^2(\Omega)$. We let $\top$ denote the transposition.  A generic constant possibly different from one estimate to another is denoted by $c$.

Throughout, we suppress vector notations by writing $L^2(\Om)$ instead of $L^2(\Om;\bbR^N)$, etc. Similarly, for a vector-valued function $u=(u_n)_{n=1}^N:\Om\to\bbR^N$ we write $\nabla u$ applying the gradient to each component $u_n$ of $u$, and analogously write $\Delta u=(\Delta u_n)_{n=1}^N$, etc. Given two vector-valued functions $u=(u_n)_{n=1}^N$ and $v=(v_n)_{n=1}^N$, we write $uv=(u_nv_n)_{n=1}^N$ for their componentwise product. We often use ``$\cdot$'' to denote the scalar product in $\bbR^d$ and write 
\begin{equation}\label{Abb1}
\nabla u\cdot\nabla v=\sum_{n=1}^N\langle\nabla u_n,\nabla v_n\rangle_{\bbR^d},\, \langle \nabla u,\nabla v\rangle_{L^2}=\sum_{n=1}^N\int_{\Om}\langle\nabla u_n(x),\nabla v_n(x)\rangle_{\bbR^d}\,dx.
\end{equation} 
More general, if $U=(U_n)_{n=1}^N$ and $V=(V_n)_{n=1}^N$ are vector-valued functions with components $U_n,V_n\in\bbR^d$ then we
write
\beq\lb{Abb2}
U\cdot V=\sum_{n=1}^N\langle U_n, V_n\rangle_{\bbR^d},\, \langle U,V\rangle_{L^2(\Om)}=\sum_{n=1}^N\int_{\Om}\langle U_n(x), V_n(x)\rangle_{\bbR^d}\,dx.
\enq
 We denote by $\cH=H^{1/2}(\dOm)\times H^{-1/2}(\dOm)$ the boundary space, by $H^1_\Delta(\Om)$ the subspace of weakly harmonic functions in $H^1(\Om)$, by $\gaD$ the Dirichlet and by $\gaN$ the weak Neumann trace, and write $\tr u=(\gaD u,\gaN u)$.
 
 We use the notation $L=-\Delta+V$, $L_s=-\Delta+V_s$ etc. for the operators acting from $H^1(\Om)$ into $H^{-1}(\Om)$, and $L_\cG$, $L_{s,\cG}$ etc. for the respective operators on $L^2(\Om)$ equipped with  boundary conditions associated with a given subspace $\cG$ in $\cH$.

\section{Preliminaries on elliptic problems and associated operators}\label{sec:EllPr}

For the reader's convenience, in this section we collect several well-known facts from \cite{E10,GM08,Gr,G71,G09,R96,T11,T96} concerning trace maps, Neumann operators, Dirichlet-to-Neumann and Neumann-to-Dirichlet operators and elliptic problems on Lipschitz domains that will be needed in the sequel (Sections \ref{ss:DaNt}--\ref{ss:EE}). Our main topic, however, is a description of the set of weak solutions to the homogenous equation associated with a Schr\"odinger operator via certain Birman--Schwinger type operators (Section \ref{ss:wsbso}). 

Throughout, we assume that Hypothesis \ref{h1} (i) holds.
 Unless otherwise specified, the operator $-\Delta$ is understood in the weak sense, that is, as a map between the spaces $H^1_0(\Omega)$ and $H^{-1}(\Omega)=\big(H^1_0(\Omega)\big)^*$ defined by means of the Riesz Lemma:
\[{}_{H^{-1}(\Omega)}\langle -\Delta u,\Phi\rangle_{H^1_0(\Omega)}=
\langle \nabla u,\nabla\Phi\rangle_{L^2(\Omega)},\quad u, \Phi\in H^1_0(\Omega).\]
We stress that $-\Delta\in\cB\big(H^1_0(\Om),H^{-1}(\Om)\big)$ is an isomorphism and denote its inverse by $(-\Delta)^{-1}\in\cB\big(H^{-1}(\Omega),H^1_0(\Omega)\big)$.
We will denote by \beq\label{dfnpi12}
\pi_1\colon H^1(\Omega)\to H^1_0(\Omega)\,\text{ and }\, \pi_2\colon H^1(\Om)\to H^1_\Delta(\Omega)\enq the complementary bounded   projections associated with the direct sum decomposition $H^1(\Omega)=H^1_0(\Omega)\dot{+} H^1_\Delta(\Om)$, where \begin{equation}\label{dfnHar}
H^1_\Delta(\Om)=\big\{u\in H^1(\Om): \langle \nabla u,\nabla\Phi\rangle_{L^2(\Omega)}=0\,\text{ for all }\,\Phi\in H^1_0(\Om)\big\}\end{equation} denotes the set of weakly harmonic functions. We will also use the symbol $-\Delta$ to denote the extension of the Laplacian from $H^1_0(\Om)$ to $H^1(\Om)$ defined by $-\Delta u=(-\Delta)\pi_1u$ for $u\in H^1(\Om)$. Given a bounded potential $V\in L^\infty(\Om)$, we will consider the Sch\"odinger operator $L=-\Delta+V$ as an operator from $H^1(\Om)$ into $H^{-1}(\Om)$. 
We denote by $\cK_L$ the set of weak solutions to the equation $Lu=0$, that is, we let
\begin{equation}\label{dfncK}\begin{split}
\cK_L&=\big\{u\in H^1(\Om): Lu=0\,\text{ in } H^{-1}(\Om)\big\}\\&=
\big\{u\in H^1(\Om): \langle \nabla u,\nabla\Phi\rangle_{L^2(\Omega)}+\langle V u,\Phi\rangle_{L^2(\Omega)}=0\text{ for all }\Phi\in H^1_0(\Om)\big\}.\end{split}
\end{equation}

\subsection{The Dirichlet and Neumann traces}\label{ss:DaNt} Our main objective in this subsection is to introduce the (weak) trace map $\tr u=(\gaD u,\gaN u)$, defined on a dense subset $\cD$ in $H^1(\Om)$ and mapping into the boundary space $\cH=H^{1/2}(\dOm)\times H^{-1/2}(\dOm)$.

First, we define the strong trace operators.
 Let us introduce the boundary trace operator $\gaD^0$ (the Dirichlet
trace) by
\begin{equation}\lb{2.4}
\gaD^0\colon C^0(\ol{\Om})\to C^0(\dOm), \quad
\gaD  ^0 u = u|_\dOm .
\end{equation}
By the standard trace theorem, see, e.g., \cite[Proposition 4.4.5]{T11}, or \cite[Theorem 3.38]{McL}, there exists a bounded, surjective Dirichlet
trace operator 
\begin{equation}
\gaD  \colon H^{r}(\Om)\to H^{r-1/2}(\dOm) \hookrightarrow
L^2(\dOm), \quad 1/2<r<3/2.
\lb{6.1}
\end{equation}
Furthermore, the map has a bounded right inverse, i.e., given any $f\in H^{r-1/2}(\dOm)$ there exists $u\in H^{r}(\Om)$ such that $\gaD   u=f$ and $\|u\|_{H^{r}(\Om)}\leq C\|f\|_{H^{r-1/2}(\dOm)}$.

The following existence and uniqueness result for the Dirichlet boundary value problems is well known, see, e.g., \cite[Theorem 2.1]{G71} and also \cite{Gr} and the literature therein.

\begin{theorem}\lb{thm1} Let $1\leq r\leq3/2$ and temporarily assume that $\dOm$ is $C^\infty$. Then the map $$D\colon H^{r}(\Om)\rightarrow H^{r-2}(\Om)\times H^{r-1/2}(\dOm)$$ defined by $Du=(-\Delta u, \gaD   u)$ is an isomorphism. That is, for any $v\in H^{r-2}(\Om)$ and $f\in H^{r-1/2}(\dOm)$ there exists a unique $u\in H^{r}(\Om)$ so that $Du=(v,f)$, and $u$ satisfies the estimate $\|u\|_{H^{r}(\Om)}\leq c(\|v\|_{H^{r-2}(\Om)}+\|f\|_{H^{r-1/2}(\dOm)})$. If Hypothesis \ref{h1}{\rm (i)} is satisfied then the statement in the theorem holds for $r=1$.
\end{theorem}
\noindent In other words, for any given $v\in H^{r-2}(\Om)$ and $f\in H^{r-1/2}(\dOm)$ there exists a unique solution $u\in H^{r}(\Om)$ to the Dirichlet problem
$ -\Delta u=v,\, \gaD u=f$.

We will now define the strong Neumann trace operator $\gaN   ^{\rm s}$ by
  \beq
\gaN   ^{\rm s} = \nu\cdot\gaD  \nabla \colon H^{r+1}(\Om)\to 
L^2(\dOm),\, 1/2<r<3/2, \lb{6.2}
\enq
where $\nu$ denotes the outward-pointing unit normal vector to $\partial\Om$.
With the notation just introduced, the following Green's formula holds:
\beq\label{GrF}\begin{split}
\int_{\Omega}\nabla &u\cdot \nabla \Phi\,dx \\&=-\int_{\Omega} \langle \Delta u, \Phi \rangle_{\bbR^N} dx+\int_{\dOm} \langle \gaN^{\rm s}u, \gaD  \Phi \rangle_{\bbR^N} dy,\, u\in H^{2}(\Om), \Phi\in H^{1}(\Om)\end{split}\enq
where $u=(u_n)_{n=1}^N$ and $\Phi=(\Phi_n)_{n=1}^N$.

Our next task is to define the weak Neumann trace operator $\gaN $,
an unbounded operator with dense domain $\cD \subset H^1(\Om)$ and range in $H^{-1/2}(\dOm)$, the dual space to $H^{1/2}(\dOm)$.
First, we  define the weak operator of multiplication by the normal vector. We say that $u\in H^{1}(\Om)$ has $\Delta u\in L^2(\Om)$ if there exists $v\in L^2(\Om)$ such that $\Delta u=v$ in $H^{-1}(\Om)$.
We define the unbounded operator $\nu\cdot*: w\mapsto \nu\cdot w$ as follows:
\beq\label{dfnwnu}\begin{split}
&\nu\cdot* \colon \dom(\nu\cdot*)\subset L^2(\Omega)\to H^{-1/2}(\partial\Omega),\\
&\dom(\nu\cdot*)=\{w\in
L^2(\Omega)\,|\,{\rm div } (w) \in L^2(\Omega)\},\\
&\langle\nu\cdot
w,\phi\rangle_{1/2}=\langle w, \nabla\Phi\rangle_{L^2(\Om)} +
\langle{\rm div } (w), \Phi\rangle_{L^2(\Om)},
\end{split}\enq
whenever $\phi\in H^{1/2}(\partial\Omega)$ and $\Phi\in
H^{1}(\Omega)$ are such that $\gaD  \Phi=\phi$. Here $w=(w_n)_{n=1}^N$, with $w_n(x)\in\bbR^d$, ${\rm div } (w)=({\rm div  } (w_n))_{n=1}^N$ and we recall the notation in \eqref{Abb12} and \eqref{Abb2}. The last pairing is compatible
with the distributional pairing on $\partial\Om$ and the above definition is independent of the particular extension
$\Phi\in H^{1}(\Omega)$ of $\phi$, see \cite[(A12) -- (A14)]{GLMZ05}. 

We introduce the (weak) Neumann trace operator
$\gaN $ as follows. Let
\beq\label{dfcD}
\cD=\{u\in
H^1(\Omega)\,\big|\,\Delta u\in L^2(\Omega)\,\text{ in $H^{-1}(\Om)$}\},
\enq
and define $\dom(\gaN)=\cD$ and
\begin{align}\label{dfngaN}
&\gaN \colon \dom(\gaN) \subset H^1(\Omega)\to H^{-1/2}(\partial\Omega),\quad \gaN u=\nu\cdot\nabla u,\end{align}
where $\nu\cdot*$ is the weak operator of multiplication by the normal vector defined in \eqref{dfnwnu}.
We recall \eqref{Abb12} and \eqref{Abb1} and the following Green's formula:
\begin{equation}
\langle\gaN u,\gaD  \Phi\rangle_{1/2} = \langle{\nabla u}, \nabla
\Phi\rangle_{L^2(\Om)} + \langle \Delta u,\Phi\rangle_{L^2(\Om)},\, u\in\cD, \Phi\in H^1(\Om).
\lb{wGreen}\end{equation}
The set $\cD$ can be equipped with the following ``graph norm" for the Laplacian:
\beq\label{dfncDn}
\|u\|_\cD=\|u\|_{H^1(\Om)}+\|\Delta u\|_{L^2(\Om)},\quad u\in\cD.
\enq
We note that $\gaN $ is not bounded from $\dom(\gaN )\subset H^1(\Omega)$ to $H^{-1/2}(\partial\Omega)$ when $\dom(\gaN)$ is equipped with the $H^1(\Omega)$ norm, and in fact is not even closed, though it is densely defined. However, as the following elementary lemma shows, $\gaN $ is bounded when $\dom(\gaN) =\cD$ is equipped with the norm \eqref{dfncDn}.

\begin{lemma}\lb{claim1}
Assume Hypothesis \ref{h1} (i). Then

{\rm (i)}\,
$\|\gaN u\|_{H^{-1/2}(\dOm)}\leq c(\|u\|_{H^{1}(\Om)}+\|\Delta u\|_{L^2(\Om)})$ for all $u\in\cD=\dom(\gaN)$, and

{\rm (ii)}\, the space $\cD=\dom(\gaN)$ is complete in the norm \eqref{dfncDn}.
\end{lemma}
\begin{proof}
To begin the proof of (i), we fix $\phi\in H^{1/2}(\dOm)$ and let $\Phi$ denote the solution to the boundary value problem $-\Delta\Phi=0$, $\gaD\Phi=\phi$; by Theorem \ref{thm1} with $r=1$ such a function necessarily exists, and satisfies $\|\Phi\|_{H^1(\Om)}\le c\|\phi\|_{H^{1/2}(\dOm)}$.
Then Green's formula \eqref{wGreen} yields
\begin{align*}
\big|\langle\gaN u,\phi\rangle_{1/2}\big|
&\leq\|u\|_{H^{1}(\Om)}\|\Phi\|_{H^{1}(\Om)}+
\|\Delta u\|_{L^2(\Om)}\|\Phi\|_{H^{1}(\Om)}\\
&\leq c\|\phi\|_{H^{1/2}(\dOm)}\big(\|u\|_{H^{1}(\Om)}+\|\Delta u\|_{L^2(\Om)}\big)\end{align*}
as required.
To show (ii), let $\{u_n\}$ be a Cauchy sequence in $\big(\cD,\|\cdot\|_\cD\big)$. Then $\{u_n\}$ is a Cauchy sequence in $\big(H^1(\Omega),\|\cdot\|_{H^{1}(\Om)}\big)$ and $\{\Delta u_n\}$ is a Cauchy sequence in $\big(L^2(\Omega),\|\cdot\|_{L^2(\Omega)}\big)$. Therefore,
there exist $u\in H^{1}(\Om)$ and $v\in L^2(\Omega)$ so that $u_n\to u$ in $H^{1}(\Om)$ and $\Delta u_n\to v$ in  $L^2(\Omega)$ as $n\to\infty$. Then, using twice  the definition of the (weak) Laplacian, for any $\phi\in H^1_0(\Om)$ we infer:
\begin{align*}
{}_{H^{-1}(\Om)}\langle-\Delta u,\phi\rangle_{H^{1}_0(\Om)}&=\langle\nabla u,\nabla\phi\rangle_{L^2(\Omega)}=\lim_{n\to\infty}\langle\nabla u_n,\nabla\phi\rangle_{L^2(\Omega)}\\&=\lim_{n\to\infty}\langle-\Delta u_n,\phi\rangle_{L^2(\Omega)}=\langle-v,\phi\rangle_{L^2(\Omega)}.
\end{align*}
This means that $\Delta u=v\in L^2(\Omega)$ and thus $u\in\cD$ and $u_n\to u$ in $\|\cdot\|_\cD$.
\end{proof}

We are now ready to define the (weak) trace map, $\tr$, as an unbounded operator having domain $\dom(\tr)=\cD\subset H^1(\Om)$, with $\cD$ as in \eqref{dfcD}, mapping into the boundary space $\cH=H^{1/2}(\dOm)\times H^{-1/2}(\dOm)$ as follows:
\beq\label{dfnTr}
\tr\colon\dom(\tr)\subset H^1(\Omega)\to H^{1/2}(\dOm)\times H^{-1/2}(\dOm),\,
\tr u=(\gaD  u,\gaN u).\enq
By Lemma \ref{claim1}, $\tr$ is a bounded operator when the space $\cD$ is equipped with the graph norm \eqref{dfncDn}. Also, since the space of weakly harmonic functions $H^1_\Delta(\Om)$ is a subset of $\cD$ on which the norms $\|\cdot\|_\cD$ and $\|\cdot\|_{H^1(\Om)}$ coincide, the restriction $\tr|_{H^1_\Delta(\Om)} \colon H^1_\Delta(\Om)\to H^{1/2}(\dOm)\times H^{-1/2}(\dOm)$ is a bounded operator.

We will also need the following assertion, known as the boundary unique continuation property (see, e.g., \cite[Proposition 2.5]{BB12} or \cite[Theorem 3.2.2]{I06}).
\begin{lemma}\label{UCP}
Assume Hypothesis \ref{h1}(i) and let $L=-\Delta+V$ for a potential $V\in L^\infty(\Om)$. If $u\in H^1(\Om)$ is a weak solution to the equation $Lu=0$ that satisfies both boundary conditions $\gaD u=0$ and $\gaN u=0$, then $u=0$.
\end{lemma}

\subsection{The Neumann, Dirichlet-to-Neumann and Neumann-to-Dirichlet operators}\label{ss:NOaDtNNtD}
In this subsection we define the Neumann operator $N_{-\Delta} $ mapping the Dirichlet boundary
values of harmonic functions into their Neumann boundary values, and the generalizations of this operator, $N_{L-\lambda}$ and $M_{L-\lambda}=-N_{L-\lambda}^{-1}$, when the Laplacian is replaced by  the Schr\"odinger operator $L-\lambda I_{H^1(\Om)}=-\Delta+V-\lambda I_{H^1(\Om)}$. 

As a motivation, we begin with the strong Neumann operator temporarily assuming that $\dOm$ is $C^\infty$. Let $f\in H^{3/2}(\dOm)$. By Theorem \ref{thm1} with $r=2$ there exists a unique solution $u_D\in H^{2}(\Om)$ to the boundary value problem
$-\Delta u_D=0$ in  $\Om$, 
    $\gaD   u_D=f$ on $\dOm$ so that $\|u_D\|_{H^2(\Om)}\le c\|f\|_{H^{1/2}(\dOm)}$.
We define the strong Neumann operator $N^{\rm s}_{-\Delta}$ by
 \begin{equation}\label{dfDNs}
 N^{\rm s}_{-\Delta} \colon H^{3/2}(\dOm)\to H^{1/2}(\dOm),\, N_{-\Delta} ^{\rm s} f=-\gaN u_D,
 \end{equation}
 see, e.g., \cite[Section 7.11]{T96}.
 Sometimes $N^{\rm s}_{-\Delta} $ is called the Dirichlet-to-Neumann operator, see, e.g. \cite{GM08,GM11} and the literature therein, but we will reserve the longer name for the situation when the Laplacian is replaced by the Schr\"odinger operator. Invoking the strong trace map $\tr^{\rm s}$, cf.\ \eqref{6.2},
 \beq\label{dfntrs}
 \tr^{\rm s}\colon H^{2}(\Om)\to H^{3/2}(\dOm)\times H^{1/2}(\dOm),\quad
\tr^{\rm s}u=(\gaD  u,\gaN^{\rm s}  u),\enq
we remark that the restriction of $\tr^{\rm s}$ to the space $\{u\in H^2(\Om): -\Delta u=0 \text{ in } L^2(\Om)\}$ of strongly harmonic functions satisfies the identity 
\beq\label{nop}
\tr^{\rm s} u=(f,- N^{\rm s}_{-\Delta}  f) \text{ with } f=\gaD u.
\enq
By Green's formula \eqref{GrF} we then have, for $\Phi\in H^{1}(\Om)$,
\beq\label{grfcon}
\langle N^{\rm s}_{-\Delta} f,  \gaD  \Phi\rangle_{L^2(\dOm)}
=-\langle\gaN^{\rm s}   u_D, \gaD\Phi\rangle_{L^2(\dOm)}
=-\langle\nabla u_D, \nabla\Phi\rangle_{L^2(\Om)}.\enq 
We are ready to define the weak Neumann operator assuming Hypothesis \ref{h1} (i).
\begin{definition}\label{DNO} Let $f\in H^{1/2}(\dOm)$, with  $u_D\in H^1(\Om)$ the weak solution to the boundary value problem $-\Delta u_D=0$ in $\Omega$ and $\gaD u_D=f$ on $\dOm$ whose existence is guaranteed by Theorem \ref{thm1} with $r=1$.
The weak Neumann operator $N_{-\Delta}$ is defined by
\beq\label{dfnwNo}
N_{-\Delta} \colon H^{1/2}(\dOm)\to H^{-1/2}(\dOm),\, N_{-\Delta}  f=-\gaN u_D,
\enq
 where $\gaN$ is the weak Neumann trace operator defined in \eqref{dfngaN}.
\hfill$\Diamond$
\end{definition}

We summarize the properties of the weak Neumann and trace operators.
\begin{lemma}\label{lem:propNO}
{\bf (i)}\, The weak Neumann
operator $N_{-\Delta} \in\cB\big(H^{1/2}(\dOm), H^{-1/2}(\dOm)\big)$ defined in \eqref{dfnwNo} is a bounded extension of the strong Neumann
operator $N^{\rm s}_{-\Delta} \in\cB\big(H^{3/2}(\dOm), H^{1/2}(\dOm)\big)$ defined in \eqref{dfDNs}. 

{\bf (ii)}\, The restriction of the weak trace map $\tr$ defined in \eqref{dfnTr} to the space $H^1_\Delta(\Om)$ of weakly harmonic functions satisfies the identity 
\beq\label{wnop}
\tr u=(\gaD  u,\gaN  u)=(f,- N_{-\Delta}  f) \text{ for } u \in H^1_\Delta(\Om), \, f=\gaD u;\enq
in other words, $\tr(H^1_\Delta(\Om))=\gr(-N_{-\Delta})$ in $H^{1/2}(\dOm)\times H^{-1/2}(\dOm)$.
\end{lemma}
\begin{proof}
Proving (i), we take any $f\in H^{3/2}(\dOm)$ and $g\in H^{1/2}(\dOm)$.  By Theorem \ref{thm1} for $r=1$ and $v=0$, there exists a harmonic function $\Phi\in H^1_\Delta(\Om)$ so that  $\gaD\Phi=g$ and $\|\Phi\|_{H^{1}(\Om)}\leq c\|g\|_{H^{1/2}(\dOm)}$.
Then formula \eqref{grfcon} yields
\begin{align*}
\langle N^{\rm s}_{-\Delta}  f,g\rangle_{1/2}&=\langle N^{\rm s}_{-\Delta}  f, \gaD\Phi)\rangle_{L^2(\dOm)}\\
&=-\langle\gaN^{\rm s} u_D, \gaD\Phi\rangle_{L^2(\dOm)}=-\langle\nabla u_D,\nabla\Phi\rangle_{L^2(\Om)}
\end{align*}
and therefore 
\begin{align*}
\big|\langle N^{\rm s}_{-\Delta}  f,g\rangle_{1/2}\big|\leq\|u_D\|_{H^{1}(\Om)}\|\Phi\|_{H^{1}(\Om)}\leq c\|f\|_{H^{1/2}(\dOm)}\|g\|_{H^{1/2}(\dOm)}.\end{align*}
This shows that $N^{\rm s}_{-\Delta} \colon H^{3/2}(\dOm)\to H^{1/2}(\dOm)$ admits a bounded extension $N_{-\Delta} \in\cB\big(H^{1/2}(\dOm),H^{-1/2}(\dOm)\big)$, as claimed.

The identity in (ii) follows immediately from the definition of $N_{-\Delta}$.
\end{proof}

We will now discuss the Dirichlet-to-Neumann and Neumann-to-Dirichlet operators 
associated with the Schr\"odinger operator $L-\lambda I_{H^1(\Om)}=-\Delta+V-\lambda I_{H^1(\Om)}$, assuming that $V\in L^\infty(\Om)$ and $\lambda\in\bbR$. We recall that Hypothesis \ref{h1}(i) is assumed throughout the paper. Passing from the real space $H^1(\Om; \bbR^N)$ to its compexification $H^1(\Om; \bbC^N)$ one can also treat the case $\lambda\in\bbC$, but this will not be needed.

We fix $f\in H^{1/2}(\dOm)$ and consider the following Dirichlet problem for $u\in H^1(\Om)$:
\begin{align}\lb{mainsys}
    -\Delta u+V(x)u-\lambda u=0\,\text{ in }\, H^{-1}(\Om),\,
    \gaD   u=f & \,\text{ in }\, H^{1/2}(\dOm).
\end{align}
Using the direct sum decomposition $H^1(\Om)=H^1_0(\Om)\dot+H^1_\Delta(\Om)$,
see \eqref{dfnpi12} and \eqref{dfnHar}, we split $u=u_1+u_2$, where $u_1\in H^1_0(\Om)$ and $u_2\in H^1_\Delta(\Om)$. Then \eqref{mainsys} becomes
\begin{align}
      (-\Delta+V(x)-\lambda)u_1&=(\lambda-V(x))u_2 \,\text{ in }\, H^{-1}(\Om),\,
    \gaD   u_1=0 \,\text{ in }\, H^{1/2}(\dOm),\lb{sys1}\\
 -\Delta u_2&=0 \,\text{ in }\, H^{-1}(\Om),\,
    \gaD   u_2=f  \,\text{ in }\, H^{1/2}(\dOm).\lb{sys2}
  \end{align}
By Theorem \ref{thm1}, the boundary value problem \eqref{sys2} has a unique solution $u_2$ satisfying $\|u_2\|_{H^1(\Om)}\le c\|f\|_{H^{1/2}(\dOm)}$. For the boundary value problem  \eqref{sys1} the following Fredholm alternative holds (see, e.g., \cite[Section 6.3.2]{E10} or \cite[Equation (2.105)]{GM08}):
either there exists a unique weak solution to the boundary value problem 
\[ (-\Delta+V(x)-\lambda)u_1=(\lambda-V(x))u_2 \,\text{ in }\, H^{-1}(\Om),
    \gaD   u_1=0 \,\text{ in }\, H^{1/2}(\dOm),\]
or else there exists a nonzero weak solution to the boundary value problem
\beq\lb{evp}
    (-\Delta+V(x)-\lambda)u_1=0 \,\text{ in }\, H^{-1}(\Om),\,
    \gaD   u_1=0 \,\text{ in }\, H^{1/2}(\dOm).
 \enq
Let $L_{\cH_D}$ denote the Schr\"odinger operator in $L^2(\Om)$ equipped with the standard Dirichlet boundary condition, that is,
\[L_{\cH_D}u=Lu,\,\dom(L_{\cH_D})=\{u\in H^1(\Om): \Delta u\in L^2(\Om), \gaD u=0\}.\]
We note that $H^1_0(\Om)$ is defined to be the closure of $C^\infty_0(\Om)$ in $H^1(\Om)$, but the equality $H^1_0(\Om)=\{u\in H^1(\Om): \gaD u=0\}$ for Lipschitz domains can be found, e.g., in \cite[Corollary 1.5.1.6]{Gr}. If Hypothesis \ref{h1}(i') holds, then 
$\dom(L_{\cH_D})=H^2(\Om)\cap H^1_0(\Om)$, where the inclusion $\dom(L_{\cH_D})\subset H^2(\Om)$ is a consequence of elliptic regularity; see, e.g., \cite[Section 6.3.2]{E10} for a discussion for $C^2$ domains, and for $C^{1,r}$ domains, \cite[Lemma A.1]{GLMZ05} and also \cite[Theorem 2.10]{GM08} and \cite[Lemma 2.14]{GM08}. 

We now view $L_{\cH_D}$ as an unbounded operator in $L^2(\Om)$.
Then $\lambda\in\Sp(L_{\cH_D})$, that is, $\lambda$ is a Dirichlet eigenvalue of $L$, if and only if \eqref{evp} indeed has a nontrivial solution. Therefore, the following existence and uniqueness result holds:
{\em if $\lambda$ is not in $\Sp(L_{\cH_D})$ then the system \eqref{mainsys} has a unique weak solution $u=u_D\in H^1(\Om)$ and, moreover,
$\|u_D\|_{H^1(\Om)}\leq c\|f\|_{H^{1/2}(\dOm)}$}. The last estimate,
cf.\ also \eqref{2.35nn}, follows by applying to $u=u_1+u_2$ the standard $H^2$ estimate, see, e.g. \cite[Section 6.3.2]{E10}, and 
using $V\in L^\infty(\Om)$ and the  inequality $\|u_2\|_{H^1(\Om)}\le c\|f\|_{H^{1/2}(\dOm)}$:
\[\|u_1\|_{H^1(\Om)}\le\|u_1\|_{H^2(\Om)}\le c\|(\lambda-V)u_2\|_{L^2(\Om)}\le c\|u_2\|_{H^1(\Om)}\le c\|f\|_{H^{1/2}(\Om)}.
\]
Recalling \eqref{dfcD}, we note that $u_D\in\cD=\dom(\gaN)$ since $-\Delta u_D=(\lambda-V)u_D$ and $(\lambda-V)u_D\in L^2(\Om)$.
We will also make use of the following ``strong'' version of the above mentioned existence and uniqueness result: if $\lambda$ is not in $\Sp(L_{\cH_D})$ then for any $f\in H^{1}(\dOm)$ system \eqref{mainsys} has a unique solution $u=u_D\in H^{3/2}(\Om)$
and, moreover,
$\|u_D\|_{H^{3/2}(\Om)}\leq c\|f\|_{H^{1}(\dOm)}$, see, e.g., \cite[Theorem 3.6]{GM08}.

Analogously, let us fix any $g\in H^{-1/2}(\dOm)$ and consider the following Neumann boundary value problem for a function $u\in H^1(\Om)$:
\begin{align}\lb{mainsysN}
    -\Delta u+V(x)u-\lambda u=0\,\text{ in }\, H^{-1}(\Om),\,
    \gaN   u=g & \,\text{ in }\, H^{-1/2}(\dOm).
\end{align}
Let $L_{\cH_N}$ denote the Schr\"odinger operator equipped with the standard Neumann boundary condition, that is,
\[L_{\cH_N}u=Lu,\,\dom(L_{\cH_N})=\{u\in H^1(\Om): \Delta u\in L^2(\Om), \gaN u=0\}.\]
Again, we refer to \cite[Lemma A.1]{GLMZ05} for the inclusion $\dom(L_{\cH_N})\subset H^2(\Om)$. We can now view $L_{\cH_N}$ as an unbounded operator on $L^2(\Om)$.
Then $\lambda\in\Sp(L_{\cH_N})$, that is, $\lambda$ is a Neumann eigenvalue of $L$, if and only if the homogeneous problem \eqref{mainsysN}  has a nontrivial solution.
If $\lambda$ is not in $\Sp(L_{\cH_N})$ then system \eqref{mainsysN} has a unique weak solution $u=u_N\in H^1(\Om)$ and, moreover,
$\|u_N\|_{H^1(\Om)}\leq c\|g\|_{H^{-1/2}(\dOm)}$, see, e.g., \cite[Corollary 4.4]{GM08}. The ``strong'' version of the existence and uniqueness result reads as follows:
if $\lambda$ is not in $\Sp(L_{\cH_N})$, then for each $g\in L^2(\dOm)$ the system \eqref{mainsysN} has a unique solution $u=u_N\in H^{3/2}(\Om)$ and, moreover,
$\|u_N\|_{H^{3/2}(\Om)}\leq c\|g\|_{L^2(\dOm)}$, see, e.g., \cite[Corollary 3.3]{GM08}.

We are ready to define the weak and strong  Dirichlet-to-Neumann and Neumann-to-Dirichlet maps associated with the Schr\"odinger operator $L-\lambda$.

\begin{definition}\label{dfnDNNDmaps}
{\bf (i)} \, Assume that $\lambda$ is not in $\Sp(L_{\cH_D})$. 
Let $f\in H^{1/2}(\dOm)$ and  let $u_D\in H^1(\Om)$ be the unique weak solution to the boundary value problem \eqref{mainsys}.
The weak Dirichlet-to-Neumann operator $N_{L-\lambda}$
associated with the Schr\"odinger operator $L=-\Delta +V$ is defined by
\beq\label{dfnwDNo}
N_{L-\lambda}\colon H^{1/2}(\dOm)\to H^{-1/2}(\dOm),\, N_{L-\lambda} f=-\gaN u_D,
\enq
 where $\gaN$ is the weak Neumann trace operator defined in \eqref{dfngaN}. Similarly, let $f\in H^{1}(\dOm)$ and  let $u_D\in H^{3/2}(\Om)$ be the unique solution to the boundary value problem \eqref{mainsys}.
The strong Dirichlet-to-Neumann operator $N_{L-\lambda}^{\rm s}$
associated with the Schr\"odinger operator $L=-\Delta +V$ is defined by
\beq\label{dfnwDNos}
N_{L-\lambda}^{\rm s}\colon H^{1}(\dOm)\to L^2(\dOm),\, N_{L-\lambda}^{\rm s} f=-\gaN^{\rm s} u_D,
\enq
 where $\gaN^{\rm s}$ is the strong Neumann trace operator defined in \eqref{6.2}. 
 
 {\bf (ii)} \, Assume that $\lambda$ is not in $\Sp(L_{\cH_N})$. 
Let $g\in H^{-1/2}(\dOm)$ and  let $u_N\in H^1(\Om)$ be the unique weak solution to the boundary value problem \eqref{mainsysN}.
The weak Neumann-to-Dirichlet operator $M_{L-\lambda}$
associated with the Schr\"odinger operator $L=-\Delta +V$ is defined by
\beq\label{dfnwNDo}
M_{L-\lambda}\colon H^{-1/2}(\dOm)\to H^{1/2}(\dOm),\, M_{L-\lambda} g=\gaD u_N.
\enq
Similarly, let $g\in L^2(\dOm)$ and  let $u_N\in H^{3/2}(\Om)$ be the unique solution to the boundary value problem \eqref{mainsysN}.
The strong Neumann-to-Dirichlet operator $M_{L-\lambda}^{\rm s}$
associated with the Schr\"odinger operator $L=-\Delta +V$ is defined by
\beq\label{dfnwNDos}
M_{L-\lambda}^{\rm s}\colon L^2(\dOm)\to H^{1}(\dOm),\, 
M_{L-\lambda}^{\rm s}g=\gaD u_N. 
\enq
\hfill$\Diamond$
\end{definition}
We now collect several well known results regarding the Dirichlet-to-Neumann and Neumann-to-Dirichlet operators.
\begin{lemma}\label{propDNND} Assume Hypothesis \ref{h1} (i). Then the following assertions hold:

 {\bf (i)}\, If $\lambda\in\bbR\setminus\Sp(L_{\cH_D})$ then 
 
\noindent ($i_1$)\, 
 $N_{L-\lambda}\in\cB\big(H^{1/2}(\dOm),H^{-1/2}(\dOm)\big)$;
 
\noindent ($i_2$)\, $\big(N_{L-\lambda}\big)^*=N_{L-{\lambda}}$ as operators in $\cB\big(H^{1/2}(\dOm),H^{-1/2}(\dOm)\big)$; 

\noindent ($i_3$)\, $N_{L-\lambda}$ is an extension of $N_{L-\lambda}^{\rm s}\in\cB(H^1(\Om),L^2(\dOm))$;

\noindent ($i_4$)\, the following formula holds:
 \beq\label{frmarDN}N_{L-\lambda}^{\rm s}=\gaN\big(\gaN(L_{\cH_D}-{\lambda}I_{L^2(\Om)})^{-1}\big)^\ast;\enq
 
\noindent ($i_5$)\,
 the restriction of the weak trace map $\tr$ to the set $\cK_{L-\lambda}$ of weak solutions to the equation $(L-\lambda I_{H^1(\Om)})u=0$ (defined in \eqref{dfncK}) satisfies the identity 
\beq\label{wndop}
\tr u=(f,- N_{L-\lambda} f) \text{ for } u \in \cK_{L-\lambda}, \,  f=\gaD u;\enq
in other words,  $\tr(\cK_{L-\lambda})=\gr(-N_{L-\lambda})$ in $H^{1/2}(\dOm)\times H^{-1/2}(\dOm)$.

{\bf (ii)}\, If $\lambda\in\bbR\setminus\Sp(L_{\cH_N})$ then

\noindent ($ii_1$)\, $M_{L-\lambda}\in\cB\big(H^{-1/2}(\dOm),H^{1/2}(\dOm)\big)$;

\noindent ($ii_2$)\, $\big(M_{L-\lambda}\big)^*=M_{L-{\lambda}}$ as operators in $\cB\big(H^{-1/2}(\dOm),H^{1/2}(\dOm)\big)$;

\noindent ($ii_3$)\,
 $M_{L-\lambda}$ is an extension of $M_{L-\lambda}^{\rm s}\in\cB(L^2(\dOm),H^1(\dOm))$;
 
\noindent ($ii_4$)\, the following formula holds:
 \beq\label{frmarND}M_{L-\lambda}^{\rm s}=\gaD\big(\gaD(L_{\cH_N}-{\lambda}I_{L^2(\Om)})^{-1}\big)^\ast;\enq
 
\noindent ($i_5$)\,
the restriction of the weak trace map $\tr$ to the set $\cK_{L-\lambda}$ of the weak solutions to the equation $(L-\lambda I_{H^1(\Om)})u=0$ satisfies the identity
\beq\label{wdnop}
\tr u=(-M_{L-\lambda}g,g) \text{ for } u \in \cK_{L-\lambda}, \, g=\gaN u;\enq
in other words, $\tr(\cK_{L-\lambda})=\gr'(-M_{L-\lambda})$ 
 in  $H^{1/2}(\dOm)\times H^{-1/2}(\dOm)$.

{\bf (iii)}\, If $\lambda\in\bbR\setminus\big(\Sp(L_{\cH_D})\cup\Sp(L_{\cH_N})\big)$ then
$\big(N_{L-\lambda}\big)^{-1}=-M_{L-\lambda}$.
\end{lemma}
\begin{proof}
We discuss (i) since (ii) is similar.
We use
 Lemma \ref{claim1}, the fact that $V\in L^\infty(\Om)$ and  the inequality $\|u_D\|_{H^1(\Om)}\leq c\|f\|_{H^{1/2}(\dOm)}$ to infer:
\begin{align*}
\|N_{L-\lambda}&f\|_{H^{-1/2}(\partial\Omega)}=\|-\gaN u_D\|_{H^{-1/2}(\partial\Omega)}\leq c(\|u_D\|_{H^{1}(\Om)}+\|\Delta u_D\|_{L^2(\Om)})\\
&= c(\|u_D\|_{H^{1}(\Om)}+\|(\lambda-V)u_D\|_{L^2(\Om)})\leq
c\|u_D\|_{H^{1}(\Om)}\leq c\|f\|_{H^{1/2}(\dOm)},
\end{align*}
and thus $(i_1)$ holds. The formula for the adjoint operator in $(i_2)$ is proved in \cite[Lemma 4.13]{GM08}. The assertion $(i_3)$ regarding the extension is given in \cite[Corollary 4.7]{GM08}. Formula \eqref{frmarDN} is given in \cite[Theorem 3.7, Remark 3.9]{GM08} or in \cite[(3.57),(3.59)]{GMZ07}. Finally, if $u\in\cK_{L-\lambda}$ then $-\Delta u=(\lambda-V)u\in L^2(\Om)$ and thus $u\in\cD=\dom(\gaN)=\dom(\tr)$, cf.\ \eqref{dfcD}, \eqref{dfngaN}. Now $(i_5)$ follows by Definition \ref{dfnDNNDmaps}(i). Also, we remark that the norm $\|\cdot\|_\cD$ from \eqref{dfncDn} is equivalent to the norm $\|\cdot\|_{H^1(\Om)}$
on $\cK_{L-\lambda}$; this is in accord with the boundedness of the Dirichlet-to-Neumann operator in $(i_1)$ and the boundedness of the trace map from $\cD$ with $\|\cdot\|_\cD$ into $\cH$, cf.\ Lemma \ref{claim1}(i). The proof of (iii) is in \cite[Theorem 3.7]{GM08}.
\end{proof}
\begin{remark}\lb{adjNM} The adjoint operators in assertions $(i_2)$ and $(ii_2)$ of Lemma \ref{propDNND} have been computed in the sense of duality between the spaces $H^{1/2}(\dOm)$ and $H^{-1/2}(\dOm)=\big(H^{1/2}(\dOm)\big)^*$ but could be rewritten using the adjoints with respect to the scalar products
in $H^{1/2}(\dOm)$ and $H^{-1/2}(\dOm)$.  Indeed, let $J_{\rm R}\colon H^{-1/2}(\dOm)\to H^{1/2}(\dOm)$ be the duality map given by the Riesz Lemma so that $\langle J_{\rm R} g,f\rangle_{H^{1/2}(\dOm)}=\langle g,f\rangle_{1/2}$ for any $g\in H^{-1/2}(\dOm)$ and $f\in H^{1/2}(\dOm)$. The operator $J_{\rm R}$ is an isometric isomorphism and thus $J_{\rm R}^*=J_{\rm R}^{-1}$. Using this operator and denoting by $^*$ the adjoint operator with respect to the scalar products in the Hilbert spaces $H^{1/2}(\dOm)$ and $H^{-1/2}(\dOm)$, one can equivalently rewrite $(i_2)$ and $(ii_2)$ as 
$\big(J_{\rm R} N_{L-\lambda})^*=J_{\rm R} N_{L-{\lambda}}$ in $H^{1/2}(\dOm)$ and 
$\big(J_{\rm R}^*M_{L-\lambda})^*=J_{\rm R}^*M_{L-{\lambda}}$ in $H^{-1/2}(\dOm)$.
\hfill$\Diamond$\end{remark}
\begin{remark}\label{JJR} Using the Riesz duality map $J_{\rm R}$ just introduced in Remark \ref{adjNM}, there is a standard way of defining on $\cH=H^{1/2}(\dOm)\times H^{-1/2}(\dOm)$ a complex structure $J$ so that $J^*=-J$, $J^2=-I_\cH$ and the symplectic form $\omega$ on $\cH$ defined in \eqref{dfnomega} is given by $\omega\big((f_1,g_1),(f_2,g_2)\big)=\langle J(f_1,g_1),(f_2,g_2)\rangle_{\cH}$ where the scalar product in $\cH$ is given by $\langle(f_1,g_1), (f_2,g_2)\rangle_\cH=\langle f_1, f_2\rangle_{H^{1/2}(\dOm)}+\langle g_1, g_2\rangle_{H^{-1/2}(\dOm)}$.
Following \cite[Example 2.1]{F}, we let $J(f,g)=(J_{\rm R}g,-J_{\rm R}^{-1} f)$ for $(f,g)\in\cH$.
\hfill$\Diamond$\end{remark}

Setting the stage for the  Lagrangian analysis, we conclude this section by proving that the form $\omega$ defined in \eqref{dfnomega} vanishes on the set $\tr(\cK_{L-\lambda})$ of traces of weak solutions to the equation $(L-\lambda I_{H^1(\Om)})u=0$, where $\cK_{L-\lambda}$ is defined as in \eqref{dfncK}.
\begin{lemma}\label{lemSF}
 Let $w,v\in\cK_{L-\lambda}$. Then 
 $\omega((\gaD  w,\gaN w),(\gaD  v,\gaN v))=0$.
\end{lemma}
\begin{proof}
 Using \eqref{wGreen}, we infer:
\begin{align*}
 \omega((&\gaD  w,\gaN w),(\gaD  v,\gaN v))
 =\langle \gaN v,\gaD  w\rangle_{1/2}
 - \langle\gaN w,\gaD  v\rangle_{1/2}\\
 &=\langle\nabla v, \nabla w\rangle_{L^2(\Om)}+\langle\Delta v,w\rangle_{L^2(\Om)}
 -\Big(\langle\nabla w, \nabla v\rangle_{L^2(\Om)}+\langle\Delta w,v\rangle_{L^2(\Om)}\Big)\\
 &=\langle\Delta v,w\rangle_{L^2(\Om)}-
 \langle\Delta w,v\rangle_{L^2(\Om)}\\
 &=\langle(V-\lambda)v,w\rangle_{L^2(\Om)}-\langle(V-\lambda)w,v\rangle_{L^2(\Om)}=0.
\end{align*}
\end{proof}

\subsection{Elliptic estimates}\label{ss:EE}
In this subsection we provide a rather standard elliptic fact regarding the (weak) trace map $\tr$ defined in \eqref{dfnTr}.

We recall from Subsection \ref{ss:DaNt} that the operator $-\Delta$ is understood in  weak sense, that is, as a map between the spaces $H^1_0(\Omega)$ or $H^1(\Om)$ and $H^{-1}(\Omega)=\big(H^1_0(\Omega)\big)^*$ and  that $-\Delta\in\cB\big(H^1_0(\Om),H^{-1}(\Om)\big)$ is an isomorphism, with the inverse map denoted by $(-\Delta)^{-1}\in\cB\big(H^{-1}(\Omega),H^1_0(\Omega)\big)$.  Often, we use the same symbol $-\Delta$ to denote the extension of the Laplacian from $H^1_0(\Om)$ to $H^1(\Om)$ defined by $-\Delta u=(-\Delta)\pi_1u$ for $u\in H^1(\Om)$. Thus, the operator $L=-\Delta+V$ with a potential $V\in L^\infty(\Omega)$ is defined on $H^1(\Om)$ while its range is in $H^{-1}(\Om)$. We recall the notation $\cK_L$ for the set of the weak solutions to the equation $Lu=0$ defined in
\eqref{dfncK}.


\begin{lemma}\label{estbelow} Assume Hypothesis \ref{h1} (i).
Let $\Sigma=[a,b]$ be a set of parameters, and let $\{V_s\}_{s\in\Sigma}$ be a family of potentials such that the function $s\mapsto V_s$ is in $C^k(\Sigma; L^\infty(\Om))$ for some $k \in\{0,1,\dots\}$.
Then there exists a constant $c>0$ independent of $s$ such that $\|\tr u\|_\cH\ge c\|u\|_{H^1(\Om)}$ for any weak solution $u\in H^1(\Om)$ to the equation $L_su=0$ associated with the Schr\"odinger operator $L_s=-\Delta+V_s$. In particular, the restrictions $\tr |_{H^1_\Delta(\Om)}$ and $\tr |_{\cK_{L_s}}$ of the trace map are injective on the subspaces defined in \eqref{dfnHar} and \eqref{dfncK}, and the subspaces $\tr (H^1_\Delta(\Om))$ and $\tr (\cK_{L_s})$ are closed in $\cH=H^{1/2}(\dOm)\times H^{-1/2}(\dOm)$.
\end{lemma}
\begin{proof}
Take any $u\in\dom(\gaN)$. Then $L_su\in L^2(\Omega)$ by \eqref{dfcD}, \eqref{dfngaN}. Moreover, Green's formula \eqref{wGreen} yields
\begin{align*}
\langle L_su,u\rangle_{L^2(\Omega)}=&\langle -\Delta u,u\rangle_{L^2(\Omega)}+\langle V_su,u\rangle_{L^2(\Omega)}\\
=&\langle \nabla u,\nabla u\rangle_{L^2(\Omega)}-\langle\gaN  u,\gaD   v\rangle_{1/2}+\langle V_su,u\rangle_{L^2(\Omega)},
\end{align*}
and thus
\begin{align*}
\|\nabla u\|^2_{L^2(\Omega)}&\leq\|L_s u\|_{L^2(\Omega)}\| u\|_{L^2(\Omega)}+\| V_s\|_{L^{\infty}(\Omega)}\| u\|^2_{L^2(\Omega)}+\| \gaN  u\|_{H^{-1/2}}\| \gaD  u\|_{H^{1/2}}\\
&\leq\frac{1}{2}\|L_s u\|^2_{L^2(\Omega)}+\frac{1}{2}\| u\|^2_{L^2(\Omega)}+\| V_s\|_{L^{\infty}(\Omega)}\| u\|^2_{L^2(\Omega)}
\\&\qquad\qquad
+\frac{1}{2}\| \gaN  u\|^2_{H^{-1/2}}+\frac{1}{2}\| \gaD  u\|^2_{H^{1/2}}.
\end{align*} 
Since $\sup_s\| V_s\|_{L^{\infty}(\Omega)}<\infty$,
 for all $u\in\dom(\gaN)$ we therefore have
\begin{align}\label{ei}
\| u\|^2_{H^1(\Omega)}\leq c\,\big(\|L_s u\|^2_{L^2(\Omega)}+\| u\|^2_{L^2(\Omega)}
+\| \tr  u\|^2_\cH\big),
\end{align} 
where $c$ is independent of $s$ and $u\in\dom(\gaN)$. 
We want to show that, in fact,
\begin{align}\label{2.35nn}
\| u\|^2_{H^1(\Omega)}\leq c\,\big(\|L_s u\|^2_{L^2(\Omega)}
+\| \tr  u\|^2_\cH\big).
\end{align} 
Seeking a contradiction, we suppose there exist sequences  $\{s_n\}$ in $\Sigma$ and $\{u_n\}$ in $\dom(\gaN)$ such that $\| u_n\|^2_{L^2(\Omega)}=1$ and
\begin{align}\label{2.35n}
\| u_n\|^2_{H^1(\Omega)}\geq n\big(\|L_{s_n} u_n\|^2_{L^2(\Omega)}
+\| \tr  u_n\|^2_\cH\big),\, n=1,2,\dots.
\end{align} 
The sequence $\{u_n\}$ is bounded in $H^1(\Omega)$ by \eqref{ei}. Thus, $\| \tr  u_n\|_\cH\rightarrow0$ and $\|L_{s_n} u_n\|_{L^2(\Omega)}\rightarrow0$ as $n\to\infty$ by \eqref{2.35n}. Passing to a subsequence, we can assume that $u_n\rightharpoonup u$ weakly in  $H^1(\Omega)$ as $n\to\infty$. Since the inclusion $H^1(\Omega)\hookrightarrow L^2(\Omega)$ is compact, it follows that $u_n\rightarrow u$ in $L^2(\Omega)$ as $n\to\infty$, thus $\| u\|^2_{L^2(\Omega)}=1$.
Next, we claim that
\begin{equation}\label{cl2.35}
u\in\dom(\gaN),\,
\tr u=0 \,\text{ and
$L_{s_0}u=0$ for some $s_0\in\Sigma$.}
\end{equation} 
Assuming the claim, we conclude that $u=0$ by the boundary unique continuation property in Lemma \ref{UCP} for a solution to the equation $L_{s_0}u=0$ with zero Dirichlet and Neumann boundary values. This contradicts $\|u\|_{L^2(\Om)}=1$ and proves \eqref{2.35nn} and the lemma.

 Starting the proof of claim \eqref{cl2.35}, and passing to a subsequence if necessary, we may assume that $V_{s_n}\rightarrow V_{s_0}$ in $L^\infty(\Omega)$ as $n\to\infty$ for some $s_0\in\Sigma$.
Since the map $-\Delta\colon H^1(\Omega)\rightarrow(H^1_0(\Omega))^*$ is continuous, and $-\Delta u_n=L_{s_n}u_n-V_{s_n}u_n$ in $(H^1_0(\Omega))^*$ with $L_{s_n}u_n\rightarrow0$ in $L^2(\Omega)$ and $V_{s_n}u_n\rightarrow V_{s_0}u$ in $L^2(\Omega)$ as $n\to\infty$, we infer that $-\Delta u_n\rightarrow w$ in $L^2(\Omega)$ for some $w\in L^2(\Omega)$. Since the map  $-\Delta\colon H^1(\Omega)\rightarrow(H^1_0(\Omega))^*$ is also weakly closed and $u_n\rightharpoonup u$ weakly in  $H^1(\Omega)$,
the fact that $-\Delta u_n$ converges to $w$ in $(H^1_0(\Omega))^*$ implies that $w=-\Delta u$. Since $w\in L^2(\Om)$ we thus conclude that $u\in\dom(\gaN)$ as required in \eqref{cl2.35}. Furthermore, the sequence 
\[ L_{s_n}u_n-L_{s_0}u=-\Delta u_n+\Delta u+V_{s_n}(u_n-u)+(V_{s_n}-V_{s_0})u\] converges to $0$ in $L^2(\Omega)$ as $n\to\infty$. Since $L_{s_n}u_n\rightarrow0$ in $L^2(\Omega)$, we thus conclude that $L_{s_0}u=0$ as also required in \eqref{cl2.35}.
Next, we recall that $\gaN u_n\rightarrow0$ in $H^{-1/2}(\dOm)$ and 
$\gaD u_n\rightarrow0$ in $H^{1/2}(\dOm)$ since $\| \tr  u_n\|_\cH\rightarrow0$ by \eqref{2.35n}.
Passing to the limit as $n\to\infty$ in Green's formula 
\begin{align*}
\langle\Delta u_n,\Phi\rangle_{L^2(\Omega)}+\langle\nabla u_n,\nabla \Phi\rangle_{L^2(\Omega)}
=\langle\gaN  u_n,\gaD   \Phi\rangle_{1/2},
\end{align*}
we have 
$\langle\Delta u,\Phi\rangle_{L^2(\Omega)}+\langle\nabla u,\nabla \Phi\rangle_{L^2(\Omega)}=0$ for all $\Phi\in H^1(\Om)$. Now Green's formula also gives 
\begin{align*}
\langle\gaN  u,\gaD   \Phi\rangle_{1/2}=\langle\Delta u,\Phi\rangle_{L^2(\Omega)}+\langle\nabla u,\nabla \Phi\rangle_{L^2(\Omega)}=0
\end{align*}
yielding $\gaN u=0$. Finally, $\gaD u_n\rightharpoonup\gaD u$ weakly in $H^{1/2}(\dOm)$ as $n\to\infty$ since the operator $\gaD$ is continuous, and thus $\gaD u=0$, finishing the proof of claim \eqref{cl2.35}. 
\end{proof}

\subsection{Weak solutions and a Birman--Schwinger type operator}
\label{ss:wsbso}
We will now give a convenient description of the set $\cK_L$, defined in \eqref{dfncK}, of weak solutions to the homogeneous equation associated with the Schr\"odinger operator $L=-\Delta+V$ with a potential $V\in L^\infty(\Om)$, using the Birman--Schwinger type operator $G=I_{H^1(\Om)}+(-\Delta)^{-1} V$ defined in \eqref{dfFG} below. The main assertion in this subsection, Lemma \ref{lem:propFG}, summarizes developments in \cite{DJ11} and supplies necessary details for the case of Lipschitz domains at hand.

Let us recall the standard embeddings
\begin{equation}\label{emb}
H^1_0(\Om)\underset{j_0}\hookrightarrow H^1(\Om)\underset{j_1}\hookrightarrow L^2(\Omega)\underset{j}\hookrightarrow H^{-1}(\Omega)=\big(H^1_0(\Omega)\big)^*,
\, j=(j_1\circ j_0)^*.
\end{equation}
The Rellich--Kondrachov Compactness Theorem implies $j\in\cB_\infty(L^2(\Om),H^{-1}(\Om))$ and $j_1\in\cB_\infty(H^1(\Om),L^2(\Om))$; see e.g. \cite[Theorem 5.7.1]{E10} for the case of a $C^1$ boundary. We now define the operator
\begin{equation}\label{dfFG}
G=I_{H^1(\Om)}+j_0\circ(-\Delta)^{-1}\circ j\circ V\colon H^1(\Om)\to H^1(\Om),
\end{equation}
and notice that $G-I_{H^1(\Om)}\in\cB_\infty(H^1(\Om))$ by the compactness of $j$.
Recall that we regard $L=-\Delta+V$ as an operator from $H^1(\Om)$ into $H^{-1}(\Om)=(H^1_0(\Om))^*$.
\begin{lemma}\label{lem:propFG} Assume $V\in L^\infty(\Om)$ and $V(x)=V(x)^\top$ for all $x\in\Om$. Then the following assertions hold for the operators $L=-\Delta+V$ and $G$ defined in \eqref{dfFG} and the set $\cK_L$ of the weak solutions to the equation $Lu=0$ defined in \eqref{dfncK}:
\begin{itemize}\item[(i)] \quad $\overline{\ran(L)}=H^{-1}(\Om)$;
\item[(ii)] \quad $\ran(G)+H^1_\Delta(\Om)=H^1(\Om)$;
\item[(iii)] \quad $\cK_L=\big\{u\in H^1(\Om): Gu\in H^1_\Delta(\Om)\big\}$.
\end{itemize}
\end{lemma}
\begin{proof} 
In the course of the proof we will have to pay attention to a subtle difference between operators acting from $H_0^1(\Om)$ and from $H^1(\Om)$ into $H^{-1}(\Om)=(H^1_0(\Om))^*$. Specifically, in the course of the proof we temporarily denote by $-\Delta$ the selfadjoint operator acting from $H^1_0(\Om)$ into $(H^1_0(\Om))^*$, and by $-\Delta\pi_1$ the operator acting from $H^1(\Om)$ into $(H^1_0(\Om))^*$. As mentioned in the beginning of Section \ref{sec:EllPr}, the operator $-\Delta\in\cB(H^1_0(\Om), (H^1_0(\Om))^*)$ is an isomorphism, with inverse $(-\Delta)^{-1}$. We denote by $V\colon H^1(\Om) \rightarrow L^2(\Om)$ the operator of multiplication by the potential, via the formula $(Vu)(x)=V(x)u(x)$ for $x\in\Om$. Thus, using $j$ from \eqref{emb}, we may consider the operator $jV$ from $H^1(\Om)$ into $(H^1_0(\Om))^*$. With this notation the Schr\"odinger  operator $L$ is given, strictly speaking, by the formula $L=-\Delta\pi_1+jV$ when it is considered from $H^1(\Om)$ into $(H^1_0(\Om))^*$. The adjoint operator $L^\ast=(-\Delta\pi_1)^\ast+(jV)^\ast$ acts from $H^1_0(\Om)$ into $(H^1(\Om))^*$.

We note that $G-I_{H^1(\Om)}=j_0\circ(-\Delta)^{-1}\circ j\circ V$, and remark that $G-I_{H^1(\Om)}=\pi_1 (G-I_{H^1(\Om)})$ since $\ran(j_0)=H^1_0(\Om)$. Also, plugging $jV=L-(-\Delta\pi_1)$ into formula \eqref{dfFG} for $G$, we infer:
\begin{align*}
G&=I_{H^1(\Om)}+j_0(-\Delta)^{-1}\big(L-(-\Delta)\pi_1\big)\\&=I_{H^1(\Om)}+j_0(-\Delta)^{-1}L-\pi_1=\pi_2+j_0(-\Delta)^{-1}L,
\end{align*}
since $j_0(-\Delta)^{-1}(-\Delta)\pi_1=\pi_1$, and therefore 
\beq\lb{cl0V}
\pi_1G=j_0(-\Delta)^{-1}L.\enq

 We will need two more auxiliary assertions. First, we claim that 
 \beq\label{cl1V}
 \text{ if $v\in H^1_0(\Om)$ then $(jV)^*v\in L^2(\Om)$}. \enq
 Indeed, here $jV:H^1(\Om)\to (H^1_0(\Om))^*$ and the adjoint operator $(jV)^*$ acting from $H^1_0(\Om)$ into $(H^1(\Om))^*$ is defined such that
 \beq\label{fir}
 {}_{(H^1(\Om))^*}\langle (jV)^*v,\phi\rangle_{H^1(\Om)}=
 {}_{(H^1_0(\Om))^*}\langle (jV)\phi,v\rangle_{H^1_0(\Om)}
 \text{ for all }\phi\in H^1(\Om).\enq
 Since $V\phi\in L^2(\Om)$, the right-hand side of \eqref{fir} is in fact the integral
 \[\langle V\phi,v\rangle_{L^2}=\langle Vv,\phi\rangle_{L^2}.\]
 Thus, the functional $(jV)^\ast v\in (H^1(\Om))^\ast$ is represented by the function $Vv\in L^2(\Om)$ as claimed in \eqref{cl1V}. Second, we claim that 
 \beq\label{cl2V}\text{if $v\in H^1_0(\Om)$ then $(-\Delta\pi_1)^\ast v=(-\Delta) v$ in $(H^1_0(\Om))^\ast$},\enq
 where $-\Delta\pi_1\colon H^1(\Om)\to(H^1_0(\Om))^\ast$ and $-\Delta=(-\Delta)^\ast\colon H^1_0(\Om)\to(H^1_0(\Om))^\ast$. Indeed, the functional $(-\Delta\pi_1)^*v\in(H^1(\Om))^*\hookrightarrow (H^1_0(\Om))^*$ by definition of $(-\Delta\pi_1)^*$ satisfies the identity 
 \[{}_{(H^1(\Om))^*}\langle (-\Delta\pi_1)^*v,\phi\rangle_{H^1(\Om)}=
 {}_{(H^1_0(\Om))^*}\langle (-\Delta\pi_1)\phi,v\rangle_{H^1_0(\Om)}
 \,\text{ for all } \phi\in H^1(\Om).\]
 But when $v\in H^1_0(\Om)$ and $\phi\in H^1_0(\Om)$ we infer
 \begin{align*}
 {}_{(H^1_0(\Om))^*}&\langle (-\Delta\pi_1)^*v,\phi\rangle_{H^1_0(\Om)}=
  {}_{(H^1_0(\Om))^*}\langle (-\Delta\pi_1)\phi,v\rangle_{H^1_0(\Om)}\\&=
   {}_{(H^1_0(\Om))^*}\langle (-\Delta)\phi,v\rangle_{H^1_0(\Om)}
   ={}_{(H^1_0(\Om))^*}\langle (-\Delta)v,\phi\rangle_{H^1_0(\Om)}
    \,\text{ for all } \phi\in H^1_0(\Om),
 \end{align*}
 which gives the required claim \eqref{cl2V}. We are ready to prove assertions (i)--(iii).
 
 \, (i) It suffices to show that $\ker L^*=\{0\}$. 
Assume that $v\in H^1_0(\Om)$ is such that $L^*v=0$ for the operator $L^\ast=(-\Delta\pi_1)^\ast+(jV)^\ast$ acting from $H^1_0(\Om)$ into $(H^1(\Om))^*$. Using claims \eqref{cl1V} and \eqref{cl2V} we conclude that
\beq\lb{clVm} (-\Delta)v=(-\Delta\pi_1)^*v=-(jV)^*v\in L^2(\Om),\enq
and thus $v\in\cD=\dom(\gaN)$, see \eqref{dfcD}, \eqref{dfngaN}. We also recall that if 
 $u\in H^1_\Delta(\Om)=\{u\in H^1(\Om)\mid\Delta u=0 \,\hbox{in}\,H^{-1}(\Om)\}$ then $u\in\dom\gaN$.
Let $u\in H^1_\Delta(\Om)$. Then by Green's formula \eqref{wGreen} we infer
\begin{align}
\langle \gaN  v,\gaD   u\rangle_{1/2} &= \langle{\nabla v}, \nabla
u\rangle_{L^2} + \langle \Delta v,u\rangle_{L^2},\label{feq}\\
\langle \gaN  u,\gaD   v\rangle_{1/2} &= \langle{\nabla u}, \nabla
v\rangle_{L^2} + \langle \Delta u,v\rangle_{L^2}.\lb{seq}
\end{align} Since $\gaD v=0$ and $\Delta u=0$, it follows from \eqref{seq} that $\langle{\nabla u}, \nabla
v\rangle_{L^2}=0$. Hence, using \eqref{clVm}, \eqref{feq}, $(-\Delta)\pi_1u=0$ and $L^*v=0$, we conclude:
\begin{align*}
\langle&\gaN  v,\gaD   u\rangle_{1/2}=\langle \Delta v,u\rangle_{L^2}
=\langle(jV)^*v,u\rangle_{L^2}=
{}_{(H^1(\Om))^*}\langle(jV)^*v,u\rangle_{H^1(\Om)}\\&=
{}_{(H^1(\Om))^*}\langle(jV)^*v,u\rangle_{H^1(\Om)}
+{}_{(H^1_0(\Om))^*}\langle(-\Delta\pi_1)u,v\rangle_{H^1_0(\Om)}\\&=
{}_{(H^1(\Om))^*}\langle(jV)^*v,u\rangle_{H^1(\Om)}
+{}_{(H^1(\Om))^*}\langle(-\Delta\pi_1)^*v,u\rangle_{H^1(\Om)}\\&=
{}_{(H^1(\Om))^*}\langle\big((-\Delta\pi_1)^*+(jV)^*\big)v,u\rangle_{H^1(\Om)}
={}_{(H^1(\Om))^*}\langle L^*v,u\rangle_{H^1(\Om)}=0.\end{align*} So, $\langle\gaN  v,\gaD   u\rangle_{1/2}=0$ for all $u\in H^1_\Delta(\Om)$. The restriction
$\gaD|_{H^1_\Delta(\Om)}$ of the trace operator $\gaD$ is surjective into $H^{1/2}(\dOm)$. Therefore, $\langle\gaN  v, f\rangle_{1/2}=0$ for all $f\in H^{1/2}(\Om)$ and thus
$\gaN  v=0$. As a result, we have that $v$ solves the boundary value problem \[(-\Delta+V)v=0,\, \gaD   v=0,\, \gaN  v=0\] with  both Dirichlet and Neumann zero boundary values. This implies $v=0$ by the boundary uniqueness continuation property in Lemma \ref{UCP}, and thus proves assertion (i).

\, (ii) First, we show that $\ran(G)+\ran(\pi_2)$ is a closed subspace. Indeed, the operator $G$ is Fredholm since $G-I_{H^1(\Om)}\in\cB_\infty(H^1(\Om))$ by the compactness of $j$.
Thus $\codim (\ran(G))<\infty$ and therefore $\codim (\ran(G)+\ran(\pi_2))<\infty$ and thus the subspace $\ran(G)+\ran(\pi_2)$ is closed.

 Next, we show that $\ran(G)+\ran(\pi_2)$ is dense in $H^1(\Om)$. 
 Fix $w\in H^1(\Om)$ and decompose $w=\pi_1w+\pi_2w$. By assertion (i) just proved,
 the functional $(-\Delta\pi_1)w$ from  $H^{-1}(\Om)$ can be approximated by elements in $\ran(L)$, that is, 
 there exist functions $u_n\in H^1(\Om)$ such that $\|Lu_n-(-\Delta\pi_1)w\|_{(H^1_0(\Om))^*}\to0$ as $n\to\infty$. Applying the isomorphism $(-\Delta)^{-1}$, this yields  
 $\|(-\Delta)^{-1}Lu_n-\pi_1w\|_{H^1_0(\Om)}\to0$  or
 $\|j_0(-\Delta)^{-1}Lu_n-\pi_1w\|_{H^1(\Om)}\to0$ as $n\to\infty$. Using \eqref{cl0V}, we further notice that
  $\|\pi_1Gu_n-\pi_1w\|_{H^1(\Om)}\to0$ as $n\to\infty$.
Letting $w_n=Gu_n+\pi_2(-Gu_n+w)$ so that $w_n\in\ran(G)+\ran(\pi_2)$ we thus conclude
that 
\[w_n-w= Gu_n+\pi_2(-Gu_n+w)-w=\pi_1Gu_n-\pi_1w\to0\text{ as } n\to\infty\]
in $H^1(\Om)$, proving (ii).
  
\, (iii) The equation $Lu=0$ can be equivalently rewritten as $(-\Delta)\pi_1 u=-jVu$. Applying
$(-\Delta)^{-1}$, an isomorphism in $\cB(H^{-1}(\Om),H_0^1(\Om))$, this is equivalent to $\pi_1 u= -(G-I_{H^1(\Om)})u$ or $Gu=\pi_2u$, yielding (iii).
\end{proof}
\begin{remark}\label{rem:LDLN}
We now discuss what happens if we equip $L=-\Delta+V$ with the standard Dirichlet  boundary conditions, that is, we consider the operator $L_{\cH_D}$.
We assume Hypothesis \ref{h1} (i')   
and define the set $\cK_{L_{\cH_D}}$ of weak solutions to $Lu=0$ satisfying the Dirichlet condition,
\begin{align*}
\cK_{L_{\cH_D}}&=\big\{u\in H^1(\Om): Lu=0\text{ in } H^{-1}(\Om)\text{ and } \gaD u=0\big\}\\
&=\big\{u\in H^1(\Om): -\Delta u=-j\circ Vu \text{ in } H^{-1}(\Om)\text{ and } \gaD u=0\big\}.
\end{align*}
It follows that $\cK_{L_{\cH_D}}=\ker(L_{\cH_D})$, because $u\in\cK_{L_{\cH_D}}$ implies $u\in H^2(\Om)$, hence $u\in\ker(L_{\cH_D})$, by the standard elliptic theory (see, e.g., \cite[Theorem 6.3.4]{E10} or \cite[Lemma A.1]{GLMZ05}, \cite[Theorem 2.10]{GM08}, \cite[Lemma 2.14]{GM08}). 
Similar to Lemma \ref{lem:propFG}(iii), we have the relation
\beq\label{BSKD} \cK_{L_{\cH_D}}=\big\{u\in H^1(\Om): Gu=0\big\}.\enq
In other words, $0$ is not a Dirichlet eigenvalue of $L$ if and only if $-1$ is not an eigenvalue of the compact operator $G-I_{L^2(\Om)}$. The latter assertion is a part of a powerful Birman--Schwinger principle (see, e.g., \cite{GLMZ05,GM08,GM11,GMZ07} and the literature cited therein). To show \eqref{BSKD}, we note that if $u=\pi_1u\in H^1_0(\Om)$ then $-\Delta u=-j\circ Vu$ yields $u=-j_0\circ(-\Delta)^{-1}\circ j\circ Vu$ or $Gu=0$ proving ``$\subseteq$'' in \eqref{BSKD}, while
if $Gu=0$ then $Lu=0$ by Lemma \ref{lem:propFG} (iii) and also
$u=-(G-I_{H^1(\Om)})u=-\pi_1(G-I_{H^1(\Om)})u\in H^1_0(\Om)$,
thus proving ``$\supseteq$'' in \eqref{BSKD}.
\hfill$\Diamond$\end{remark}
\section{The Maslov index in symplectic Hilbert spaces}\label{secMindFrL}
In this section we collect the main definitions relevant to the Maslov index and prove two simple abstract results that will be used below to show the smoothness of the path in the Fredholm--Lagrangian Grassmannian (of the boundary subspace $\cG$) formed by the traces of weak solutions to the rescaled eigenvalue equations.

\subsection{The Fredholm--Lagrangian Grassmannian} 
Let $\cX$ be a real Hilbert space equipped with a symplectic form $\omega$, that is, a bounded, skew-symmetric, nondegenerate bilinear form. Let  $J \in \cB(\cX)$ be the associated complex structure, which satisfies $\omega(u,v)=\langle Ju,v\rangle_\cX$ for the scalar product in $\cX$, $J^2=-I_\cX$ and $J^*=-J$.
\begin{definition}\label{dfnFLG}
 {\bf (i)}\, We say that two closed linear subspaces $\cK,\cL$ of $\cX$ form a {\em Fredholm pair} if their intersection $\cK\cap\cL$ has finite dimension and their sum $\cK+\cL$ has finite codimension (and hence is closed \cite[Section IV.4.1]{Kato}). Given a closed linear subspace $\cK$ of $\cX$, we define the {\em Fredholm Grassmannian
$F(\cK)$ of $\cK$} as the following set of closed linear subspaces of $\cX$:
\beq\label{defFG}
F(\cK)=\big\{\cL: \cL \text{ is closed and $(\cK,\cL)$ is a Fredholm pair} \big\},\enq
and the {\em reduced Fredholm Grassmannian
$\Fr(\cK)$ of $\cK$} as the following set of closed linear subspaces of $\cX$:
%
\begin{equation}\label{defFr}
\Fr(\cK)=\big\{A(\cK): A\in\cB(\cX) \text{ is boundedly invertible and } A-I_\cX\in\cB_\infty(\cX)\big\}.\end{equation}

{\bf (ii)}\, A closed linear subspace $\cK$ of $\cX$ is called {\em Lagrangian} if the form $\omega$ vanishes on $\cK$, that is, $\omega(u,v)=0$ for all $u,v\in\cK$, and $\cK$ is maximal, that is, 
 if $u\in\cX$ and $\omega(u,v)=0$ for all $v\in\cK$, then  $u\in\cK$. We denote by $\Lambda(\cX)$ the set of all Lagrangian subspaces in $\cX$.
Given $\cK\in\Lambda(\cX)$, we define the {\em Fredholm--Lagrangian Grassmannian $F\Lambda(\cK)$} to be the set of Lagrangian subspaces $\cL \subset \cX$ such that $(\cK,\cL)$ is a Fredholm pair; in other words
$F\Lambda(\cK)=F(\cK)\cap\Lambda(\cX)$.
\hfill$\Diamond$
\end{definition}

We will need the following elementary facts.
\begin{lemma}\label{lem3433}
Let $\cK,\cL,\cM$ be closed linear subspaces in $\cX$.
\begin{itemize}\item[(i)] $\cL\in F(\cK)$, resp. $\cL\in\Fr(\cK)$ if and only if $\cK\in F(\cL)$, resp. $\cK\in\Fr(\cL)$.
 \item[(ii)] If $\cM\in F(\cL)$  and $\cL\in\Fr(\cK)$, then $\cM\in F(\cK)$.
\item[(iii)] If $\cK\in\Lambda(\cX)$, $\cL\in\Fr(\cK)$ and $\omega$ vanishes on $\cL$, then $\cL\in\Lambda(\cX)$.
\end{itemize}
\end{lemma}
\begin{proof} {\bf (i)}\,The statement follows from the definitions \eqref{defFG}, \eqref{defFr} of the Fredholm Grassmannian and the reduced Fredholm Grassmannian.

{\bf (ii)}\,
Let $P$ denote a projection in $\cX$ onto $\cM$ and let $N_{\cL\cM}$ denote the ``node'' operator associated with the pair of subspaces $(\cL,\cM)$, mapping $\cL$ into $\ker({P})$ by the rule $N_{\cL\cM}u=(I_\cX-P)u$ for $u\in\cL$. It is known that the pair of subspaces $(\cL,\cM)$ is Fredholm in $\cX$ if and only if the operator $N_{\cL\cM}\in\cB(\cL,\ker({P}))$ is Fredholm (and also that $\cL\cap\cM=\ker(N_{\cL\cM})$ and $\ran(N_{\cL\cM})\dot{+}\cM=\cL+\cM$), see, e.g., \cite[Lemma 5.1]{LT05} or \cite[Proposition 2.27]{F}. In particular, by the hypothesis of the lemma we know that the operator $N_{\cL\cM}\in\cB(\cL,\ker({P}))$ is Fredholm. Since there is an isomorphism $A\in\cB(\cL,\cK)$ so that $A-I_\cX\in\cB_\infty(\cX)$, we may represent the node operator $N_{\cK\cM}\in\cB(\cK,\ker({P}))$ 
as the product $N_{\cK\cM}=(I_\cX-P)A\cdot A^{-1}$, where $A^{-1}\in\cB(\cK,\cL)$ and $(I_\cX-P)A\in\cB(\cL,\ker({P}))$. From the fact above it follows that the pair $(\cK,\cM)$ is Fredholm provided the operator $(I_\cX-P)A$ is Fredholm. But this is indeed the case: $(I_\cX-P)A-N_{\cL\cM}\in\cB_\infty(\cL,\ker({P}))$ because $A-I_\cX\in\cB_\infty(\cX)$, and $N_{\cL\cM}\in\cB(\cL,\ker({P}))$ is Fredholm.

{\bf (iii)}\,
Since $\cL\in\Fr(\cK)$ we have $\cL=A(\cK)$, where $A\in\cB(\cX)$ is a boundedly  invertible operator such that  $A-I\in\cB_{\infty}(\cX)$.
The subspace  $\cL$ is closed since $\cL=A(\cK)$, $A$ is boundedly invertible and $\cK$ is closed. It remains to prove that $\cL$ is maximal. 

To begin the proof let us assume that $u\in\cX$ is such that $\omega(u,v)=0$ for all $v\in\cL$. Our objective is to show that $u\in\cL$.  Since $\cL=A(\cK)$, we have $\omega(u,A w)=0$ for all $w\in\cK$. Using the complex structure $J$, this implies
$ \langle Ju,A w \rangle_\cX=0$ or equivalently $\langle A^*Ju,w \rangle_\cX=0$ for all $w\in\cK$.
In other words,  $\omega(J^{-1}A^*Ju,w)=0$ for all $w\in\cK$. Denote $z=J^{-1}A^*Ju$. Since $\cK$ is Lagrangian by assumption, it is maximal and thus $ z\in\cK$. Rewriting 
 $z=J^{-1}A^*JAA^{-1}u$ and introducing the operator $B=J^{-1}A^*JA\in\cB(\cX)$, we therefore have $u=AB^{-1}z$, where $z\in\cK$ and the operator $B$ is boundedly invertible in $\cX$ because $A$ is boundedly invertible by assumption.
  
Since $\omega(v,w)=0$ for all $v,w\in\cL$ and $\cL=A(\cK)$, we have $\omega(A v_0,Aw_0)=0$ for all $v_0,w_0\in\cK$. Therefore $ \langle JAv_0,Aw_0 \rangle_\cX=0$ for all $v_0,w_0\in\cK$, which implies $A^*JAv_0\in\cK^{\perp}$ for any $v_0\in\cK$. 
Since $\cK\in\Lambda(\cX)$, by \cite[Proposition 2.7(b)]{F} we have $J(\cK)=\cK^{\perp}$ and $J^{-1}(\cK^\perp)=\cK$. It follows that $Bv_0=J^{-1}A^*JAv_0\in\cK$ for any $v_0\in\cK$. In other words, the operator $B$ leaves $\cK$ invariant. Let $B|_\cK$ denote the restriction of $B$ to $\cK$. We claim that $B|_\cK\in\cB(\cK)$ is boundedly invertible. Assuming the claim, the conclusion $u\in\cL$ follows since by the previous paragraph $u=AB^{-1}z=A(B|_\cK)^{-1}z$, where $z\in\cK$ and $\cL=A(\cK)$. 

To prove the claim, we first note that $B|_\cK$ is injective since $B$ is boundedly invertible in $\cX$. On the other hand,
$A-I_\cX\in\cB_{\infty}(\cX)$ yields $B-I_\cX\in\cB_{\infty}(\cX)$. Therefore
$B|_\cK-I_\cK\in\cB_{\infty}(\cK)$ and the Fredholm index of $B|_\cK\in\cB(\cK)
$ is zero. Since the operator $B|_\cK$ is injective, it is also surjective, thus proving the claim and completing the proof of the required assertion. 
\end{proof} 

Let  $\Sigma=[a,b]\subset\bbR$ be a set of parameters.
\begin{definition}\label{defcksb}
Let $k \in \{0,1,\ldots\}$. We say that a family of subspaces $\{\cK_s\}_{s\in\Sigma}$ of a Hilbert space $\cX$ is \emph{$C^k$ smooth} if for each $s_0\in\Sigma$ there exists a neighborhood $\Sigma_0$ in $\Sigma$ containing $s_0$, and a family of projections $\{P_s\}_{s\in\Sigma_0}$ on $\cX$ (possibly depending on $s_0$), such that $\ran(P_s) = \cK_s$ for each $s \in \Sigma_0$ and the function $s\mapsto P_s$ belongs to $C^k(\Sigma_0;\cB(\cX))$. \hfill$\Diamond$\end{definition}
\begin{remark}\label{locglob}
Suppose that a family of subspaces $\{\cK_s\}_{s\in\Sigma}$ in $\cX$ is $C^k$ smooth. The projections $P_s$ on $\cK_s$ are not unique (e.g., they may depend on $s_0$). Let $\Pi_s$ be the {\em orthogonal} projection onto $\cK_s$ for each $s\in\Sigma$. This family of projections is unique and can be obtained from $P_s$ by the  formula 
\begin{equation}\label{PPi}
\Pi_s=P_sP_s^*\big(P_sP_s^*+(I_\cX-P_s^*)(I_\cX-P_s)\big)^{-1},\, s\in\Sigma,
\end{equation}
see, e.g., \cite[Lemma 12.8]{BW93}.
We therefore conclude that if a family of subspaces is $C^k$ smooth, then the function $s\mapsto\Pi_s$ is in $C^k(\Sigma,\cB(\cX))$.\hfill$\Diamond$
\end{remark}

\begin{remark}\label{DalKr} We will often use the following {\em transformation operators} $W_s$, see, e.g., \cite[Section IV.1]{DK} and cf.\ \cite[Remark 6.11]{F}.
Let $\Sigma=[a,b]$ be a set of parameters containing $0$ and let $\{Q_s\}_{s\in\Sigma}$ be a family of projections on a Hilbert space $\cX$ such that the function $s\mapsto Q_s$ is in $C^k(\Sigma;\cB(\cX))$ for some $k\in \{0,1,\ldots\}$. Our objective is to construct a family of boundedly invertible operators $\{W_s\}_{s\in\Sigma_0}$ on $\cX$ that split the projections $Q_0$ and $Q_s$ in the sense that \beq\label{splitW}
W_sQ_0=Q_sW_s\,\text{ for all }\, s\in\Sigma_0,\enq and the function $s\mapsto W_s$ is in $C^k(\Sigma_0;\cB(\cX))$ for a sufficiently small neighborhood $\Sigma_0$ in $\Sigma$ containing $0$. The operator $W_s$ isomorphically maps $\ran(Q_0)$ onto $\ran(Q_s)$, which allows one to ``straighten'' a smooth family of subspaces $\{\ran(Q_s)\}$. We introduce the operators $W_s$ by the formula
\beq\label{dfWs}
W_s=Q_sQ_0+(I_\cX-Q_s)(I_\cX-Q_0),\, s\in\Sigma.
\enq
A simple calculation shows that 
\beq\lb{modWs}
W_s-I_\cX=Q_s(Q_0-Q_s)+(I_\cX-Q_s)(Q_s-Q_0).\enq
Since the function $s\mapsto Q_s$ is in $C^0(\Sigma;\cB(\cX))$, we may choose a small enough neighborhood $\Sigma_0$ in $\Sigma$ containing $0$ so that $\|W_s-I_\cX\|_{\cB(\cX)}\le 1/2$ for all $s\in\Sigma_0$. This implies that the operators $W_s$ are boundedly invertible in $\cX$ for  all $s\in\Sigma_0$, and thus $W_s\colon\ran(Q_0)\to\ran(Q_s)$ and $W_s^{-1}\colon\ran(Q_s)\to\ran(Q_0)$ are isomorphisms.

In Section \ref{sec:mainres} we will use another family of  transformation operators, $U_s$, splitting the projections $Q_0$ and $Q_s$ as in \eqref{splitW}, i.e. $U_sQ_0=Q_sU_s$ for all $s\in\Sigma_0$, cf.\ \cite[Section I.4.6]{Kato}. This family is obtained by multiplying the operators $\{W_s\}$ defined in \eqref{dfWs} by the normalizing prefactor  $\big(I_\cX-(Q_s-Q_0)^2\big)^{-1/2}$ so that
\begin{align}\label{Usf}
U_s&=\big(I_\cX-(Q_s-Q_0)^2\big)^{-1/2}\big(Q_sQ_0+(I_\cX-Q_s)(I_\cX-Q_0)\big),\\
(U_s)^{-1}&=\big(I_\cX-(Q_s-Q_0)^2\big)^{-1/2}\big(Q_0Q_s+(I_\cX-Q_0)(I_\cX-Q_s)\big).\label{Usf1}
\end{align}
The operators $U_s$ are sometimes more convenient than $W_s$ due to the symmetry in \eqref{Usf} and \eqref{Usf1}.
\hfill$\Diamond$\end{remark} 

\subsection{The Maslov index}
We next recall the definition of the Maslov index of a continuous path $\Upsilon\colon\Sigma\to F\Lambda(\cG)$ with respect to a fixed Lagarangian subspace $\cG\subset\cX$, see Definition
\ref{dfnDef3.6} below. This will require some preliminaries borrowed from \cite{BF98,F}. Given a real Lagrangian subspace $\cG$ in a real Hilbert space $\cX$ equipped with a symplectic form $\omega(u,v)=\langle Ju,v\rangle_\cX$, we assume that the path $\Upsilon\colon\Sigma\to F\Lambda(\cG)$ takes values in the Fredholm--Lagrangian Grassmannian $F\Lambda(\cG)$ and is continuous, that is, the function $s\mapsto\Pi_s$ is in $C^0(\Sigma,\cB(\cX))$, where $\Pi_s$ is the orthogonal projection in $\cX$ onto $\Upsilon(s)$. In  Definition
\ref{dfnDef3.6} below we follow \cite{F}, see also the illuminating discussion in \cite{BF98,CLM94}, and use the spectral flow of a family of unitary operators related to $\Pi_s$ and $\Pi_\cG$, the orthogonal projection on $\cG$.

To begin, we introduce a complex Hilbert space $\cX_J$ associated with the complex structure $J$ on the real Hilbert space $\cX$, defining scalar multiplication by the rule
\[
(\alpha + i\beta) u = \alpha u + \beta Ju, \quad u \in \cX, \alpha + i\beta \in \bbC,
\]
and the complex scalar product by
\[\langle u,v\rangle_{\cX_J}=\langle u,v\rangle_\cX-i\omega(u,v),\quad u,v\in\cX.\]
It is important to note that, considered as a real vector space, $\cX_J$ is identical to $\cX$, and not its complexification $\cX \otimes_{\bbR} \bbC$. (In the finite-dimensional case, if $\cX \cong \bbR^{2n}$, then $\cX_J \cong \bbC^n$ while $\cX \otimes_{\bbR} \bbC \cong \bbC^{2n}$.) However, it is easy to see that $\cX_J \cong \cG \otimes_{\bbR} \bbC$ for any  Lagrangian subspace $\cG \in \Lambda(\cX)$.

For each $s\in\Sigma$ we choose a unitary operator $U_s$ acting on the complex Hilbert space $\cX_J$ such that $\Upsilon(s)=U_s(\cG^\perp)$ and $U_s-I_{\cX_J}\in\cB_\infty(\cX_J)$. This choice is possible by \cite[Proposition 1.1]{BF98}. Next, we define the unitary operator $W_s$ in $\cX_J$ by $W_s=U_sU_s^\ast$ and notice that $W_s-I_{\cX_J}\in\cB_\infty(\cX)$. The following properties of the operator $W_s$ can be found in \cite[Lemma 1.3]{BF98} or \cite[Proposition 2.44]{F}.
\begin{lemma}\label{propWs} If $\cG$ is a real Lagrangian subspace in $\cX$, $\Upsilon\colon\Sigma=[a,b]\to F\Lambda(\cG)$ is a continuous path, $\Pi_s$ and $\Pi_\cG$ are the orthogonal projections onto $\Upsilon(s)$ and $\cG$, and $U_s$ is the unitary operator on $\cX_J$ such that $\Upsilon(s)=U_s(\cG^\perp)$, then the unitary operator $W_s=U_sU_s^\ast$ satisfies $W_s-I_{\cX_J}\in\cB_\infty(\cX)$ and
\begin{itemize}\item[(i)] $W_s=(I_{\cX_J}-2\Pi_s)(2\Pi_\cG-I_{\cX_J})$;
\item[(ii)] $\ker(W_s+I_{\cX_J})$ is isomorphic to $(\Upsilon(s)\cap\cG)\oplus J(\Upsilon(s)\cap\cG)\cong(\Upsilon(s)\cap\cG)\otimes_{\bbR}\bbC$;
\item[(iii)] $\dim_\bbR(\Upsilon(s)\cap\cG)=\dim_\bbC\ker(W_s+I_{\cX_J})$.
\end{itemize}
\end{lemma}
We will define the Maslov index of $\{\Upsilon(s)\}_{s\in\Sigma}$ as the spectral flow of the operator family $\{W_s\}_{s\in\Sigma}$ through $-1$, that is, as the net count of the eigenvalues of $W_s$ crossing the point $-1$ counterclockwise on the unit circle minus the number of eigenvalues crossing $-1$ clockwise as the parameter $s$ changes. The exact formulas for the count go back to \cite{P96} and are given in \cite{BF98,F}. Specifically, let us choose a partition $a=s_0<s_1<\dots<s_n=b$ of $\Sigma=[a,b]$ and numbers $\epsilon_j\in(0,\pi)$ so that $\ker\big(W_s-e^{i(\pi\pm\epsilon_j)}\big)=\{0\}$, that is, $e^{i(\pi\pm\epsilon_j)}\in\bbC\setminus\Sp(W_s)$, for $s_{j-1}<s<s_j$ and $j=1,\dots,n$.  This choice is indeed possible because $W_s-e^{i\pi}I_{\cX_J}$ is a Fredholm operator since $W_s-I_{\cX_J}\in\cB_\infty(\cX_J)$ has discrete eigenvalues accumulating only at zero. By the same reason, for each $j=1,\dots,n$ and any $s\in[s_{j-1},s_j]$ there are only finitely many values $\theta\in[0,\epsilon_j]$ for which $e^{i(\pi+\theta)}\in\Sp(W_s)$.
We are ready to define the Maslov index.
\begin{definition}\label{dfnDef3.6}  
Let $\cG$ be a Lagrangian subspace in a real Hilbert space $\cX$ and let $\Upsilon\colon\Sigma=[a,b]\to F\Lambda(\cG)$ be a continuous path in the Fredholm--Lagrangian Grassmannian. The Maslov index $\mas(\Upsilon,\cG)$ is defined by
\begin{equation}
\mas(\Upsilon,\cG)=\sum_{j=1}^n(k(s_j,\epsilon_j)-k(s_{j-1},\epsilon_j)),
\end{equation}
where
$k(s,\epsilon_j)=\sum_{0\leq\theta\le\epsilon_j}\dim_\bbC\ker\big(W_s-e^{i(\pi+\theta)}I_{\cX_J}\big)$ for $s_{j-1}\leq s\leq s_j.$
\end{definition}
We refer to \cite[Theorem 3.6]{F} for a list of basic properties of the Maslov index; in particular, the Maslov index is a homotopy invariant for homotopies keeping the endpoints fixed, and is additive under catenation of paths.

Our next objective is to specialize the definition of the Maslov index for (piecewise) smooth paths. In this case we will compute in Lemma \ref{propcrossform}(ii) below the Maslov index via a crossing form which is more convenient for practical computations, cf. the proofs of Lemmas \ref{prelmon} and \ref{preltmon}. We begin with 
the following elementary fact  needed to define the Maslov crossing form. Although this fact is well known, cf.\ the proof of \cite[Lemma 2.22]{F}, we were unable to locate its proof in the literature and present it for completeness. 
\begin{lemma}\lb{crossex}
Let $\{\Pi_s\}_{s\in\Sigma}$ be a family of orthogonal projections on $\cX$ such that the function $s\mapsto\Pi_s$ is in $C^k(\Sigma;\cB(\cX))$ for some $k\in\{0,1,\dots\}$. Then for any $s_0\in\Sigma$ there exists a neighborhood $\Sigma_0$ in $\Sigma$ containing $s_0$ and a family of operators $\{R_s\}$ from $\ran(\Pi_{s_0})$ into $\ker(\Pi_{s_0})$ such that 
the function $s\mapsto R_s$ is in $C^k(\Sigma_0;\cB(\ran(\Pi_{s_0}),\ker(\Pi_{s_0})))$ and for all $s\in\Sigma_0$, using the decomposition 
$\cX=\ran(\Pi_{s_0})\oplus\ker(\Pi_{s_0})$, we have
\beq\lb{greq}
\ran(\Pi_s)=\gr(R_s)=\{q+R_sq: q\in\ran(\Pi_{s_0})\}.
\enq 
Moreover,
\beq\lb{smallRs}R_s\to0\,\text{ in $\cB(\ran(\Pi_{s_0}),\ker(\Pi_{s_0}))$ as $s\to s_0$}.\enq
\end{lemma}
\begin{proof} Reparametrizing, we may assume that $s_0=0$. 
To begin the proof we first establish that 
\beq\lb{claim61a}
\Pi_0\Pi_s\colon\ran(\Pi_s)\to\ran(\Pi_0)\,\text{ is an isomorphism for all $s\in\Sigma_0$}\enq provided $\Sigma_0$ is small enough. Indeed, 
let us consider the boundedly invertible transformation operators $W_s$ from Remark \ref{DalKr} associated
with the projections $\{\Pi_s\}$ so that  $\Pi_sW_s=W_s\Pi_0$. Then $\Pi_0\Pi_s=\Pi_0W_s\Pi_0W_s^{-1}$ is an isomorphism between $\ran(\Pi_s)$ and $\ran(\Pi_0)$ if and only if $\Pi_0W_s\Pi_0$ is an isomorphism of $\ran(\Pi_0)$. But the latter fact holds because \eqref{modWs} shows that $W_s-I_\cH\to0$ in $\cB(\cH)$ as $s\to0$ and thus $\Pi_0W_s\Pi_0-\Pi_0\to0$ in $\ran(\Pi_0)$ as $s\to0$. This proves \eqref{claim61a} and also shows that $s\mapsto\Pi_0W_s\Pi_0$ is in $C^k(\Sigma_0;\cB(\ran(\Pi_0)))$. 

Let us denote by $(\Pi_0\Pi_s)^{-1}\colon\ran(\Pi_0)\to\ker(\Pi_s)$ the inverse of $\Pi_0\Pi_s\colon \ran(\Pi_s)\to\ran(\Pi_0)$, so that
$(\Pi_0\Pi_s)^{-1}=W_s(\Pi_0W_s\Pi_0)^{-1}$, and define the operator
\beq\lb{dfnRs}R_s=(I_\cX-\Pi_0)(\Pi_0\Pi_s)^{-1}\colon\ran(\Pi_0)\to\ker(\Pi_0).\enq
Since both functions $s\mapsto\Pi_s$ and $s\mapsto W_s$ are smooth, we conclude that $s\mapsto R_s$ is in $C^k(\Sigma_0;\cB(\ran(\Pi_0),\ker(\Pi_0)))$. Aslo, \eqref{smallRs} holds  
because \[R_s=(I_\cH-\Pi_0)W_s(\Pi_0W_s\Pi_0)^{-1}\to(I_\cH-\Pi_0)I_\cH\Pi_0\,\text{ as }\,s\to0.\] To finish the proof it remains to show that $\ran(\Pi_s)=\gr(R_s)$ where
$\gr(R_s)$ is given by the last equality in \eqref{greq}.
Indeed, for $p\in\ran(\Pi_s)$ we let $q=\Pi_0\Pi_sp\in\ran(\Pi_0)$ and $r=(I_\cH-\Pi_0)\Pi_sp\in\ker(\Pi_0)$ so that $p=\Pi_sp=(\Pi_0\Pi_s)^{-1}q$ and $r=(I_\cH-\Pi_0)(\Pi_0\Pi_s)^{-1}q=R_sq$. Then $p=q+r=q+R_sq\in\gr(R_s)$ thus proving $\ran(\Pi_s)\subset\gr(R_s)$. On the other hand, for $q\in\ran(\Pi_0)$ we let $p=(\Pi_0\Pi_s)^{-1}q\in\ran(\Pi_s)$ so that $(I_\cH-\Pi_0)p=R_sq$. Then 
\[q+R_sq=\Pi_0\Pi_sp+(I_\cH-\Pi_0)p=\Pi_0p+(I_\cH-\Pi_0)p=p\]
thus proving $\gr(R_s)\subset\ran(\Pi_s)$ and finishing the proof of  \eqref{greq}.
\end{proof}
We will now define the Maslov crossing form for a smooth path in the Fredholm--Lagrangian Grassmannian. Let $\cG\subset\Lambda(\cX)$ be a Lagrangian subspace in the real Hilbert space $\cX$ equipped with the symplectic form $\omega$. Consider a $C^1$ path $\Upsilon\colon\Sigma\to F\Lambda(\cG)$, that is, a family $\{\Upsilon(s)\}_{s\in\Sigma}$
of Lagrangian subspaces such that the pair $(\cG,\Upsilon(s))$ is Fredholm for each $s\in\Sigma$ and the function $s\mapsto\Pi_s$ is in $C^1(\Sigma;\cB(\cX))$, where $\Pi_s$ denotes the orthogonal projection in $\cX$ onto the subspace $\Upsilon(s)$. Fix any $s_0\in\Sigma$ and use Lemma \ref{crossex} to find a neighborhood $\Sigma_0$ in $\Sigma$ containing $s_0$ and a $C^1$-smooth family of operators $R_s$ acting from $\Upsilon(s_0)=\ran(\Pi_{s_0})$ into $\ker(\Pi_{s_0})$ such that for all $s\in\Sigma_0$, using the decomposition $\cX=\ran(\Pi_{s_0})\oplus\ker(\Pi_{s_0})$, we have
\beq\lb{eqgr}\begin{split}
\Upsilon(s)=\ran(\Pi_s)=\gr(R_s)=\{q+R_sq:
q\in\ran(\Pi_{s_0})\}.\end{split}
\enq
\begin{definition}\lb{dfnMasF}
(i)\, We call $s_0\in\Sigma$ a {\em conjugate time} or {\em crossing} if $\Upsilon(s_0)\cap\cG\neq\{0\}$.

(ii)\, The finite-dimensional, symmetric bilinear form
\beq\lb{dfnMasForm}
{\Mf}_{s_0,\cG}(q,p)=\frac{d}{ds}\omega(q,R_sp)|_{s=s_0}=
\omega(q,\dot{R}(s_0)p)\,\text{ for }\, q,p\in\Upsilon(s_0)\cap\cG
\enq is called the {\em Maslov crossing form at $s_0$}.

(iii)\, The crossing $s_0$ is called {\em regular} if the crossing form $\Mf_{s_0,\cG}$ is nondegenerate;
it is called {\em positive} if the form is positive definite and {\em negative} if the form is negative definite.
\end{definition}
\begin{remark}\lb{rem:propcf}
The crossing form ${\Mf}_{s_0,\cG}$ in Definition \ref{dfnMasF} (ii) is finite dimensional since the pair of subspaces $(\cG,\Upsilon(s_0))$ is Fredholm. The form is symmetric since the subspace $\Upsilon(s)=\ran(\Pi_s)=\gr(R_s)$ is Lagrangian and thus the equality $\omega(q+R_sq,p+R_sp)=0$ holds for all $p,q\in\ran(\Pi_0)$. As any symmetric form, 
${\Mf}_{s_0,\cG}$ can be diagonalized; we will denote by $n_+({\Mf}_{s_0,\cG})$,
respectively, $n_-({\Mf}_{s_0,\cG})$ the number of positive, respectively negative squares of ${\Mf}_{s_0,\cG}$ and by $\sgn({\Mf}_{s_0,\cG})=n_+({\Mf}_{s_0,\cG})-n_-({\Mf}_{s_0,\cG})$ its signature. It can be shown, see e.g. \cite[Proposition 3.26]{F}, that the subspace $\ker(\Pi_{s_0})$ used in \eqref{eqgr} to construct the crossing form can be replaced by any subspace $\widetilde{\Upsilon}(s_0)$ of $\cX$ such that $R_s\in\cB(\Upsilon(s_0),\widetilde{\Upsilon}(s_0))$ and $\cX=\Upsilon(s_0)+\widetilde{\Upsilon}(s_0)$. Here the sum is not necessarily orthogonal, or even direct. The crossing form then does not depend on the choice of the subspace $\widetilde{\Upsilon}(s_0)$.\hfill$\Diamond$\end{remark}
The following properties of the crossing form are taken from \cite[Section 3]{F}.
\begin{lemma}\lb{propcrossform} (i)\, Regular crossings of a $C^1$ path $\Upsilon\colon\Sigma\to F\Lambda(\cG)$  are isolated in $\Sigma$.

(ii)\, If $s_0$ is the only regular crossing in a segment $\Sigma_0=[a_0,b_0]\subset\Sigma$ then the Maslov index $\Mas(\Upsilon|_{\Sigma_0},\cG)$ can be computed as follows:
\beq\lb{MIcomp}
\Mas(\Upsilon|_{\Sigma_0},\cG)=\begin{cases}\sgn({\Mf}_{s_0,\cG})=n_+({\Mf}_{s_0,\cG})-n_-({\Mf}_{s_0,\cG})& \text{ if } s_0\in(a_0,b_0),\\
-n_-({\Mf}_{s_0,\cG})& \text{ if } s_0=a_0,\\
n_+({\Mf}_{s_0,\cG})& \text{ if } s_0=b_0.
\end{cases}
\enq
\end{lemma}

The crossing form can be used to compute the Maslov index of a $C^1$-smooth path, and hence any piecewise $C^1$-smooth path, by computing each segment individually and summing. In the event that a crossing occurs at an endpoint, its contribution will depend both on the eigenvalues of the crossing form and the endpoint at which it occurs. According to \eqref{MIcomp}, a crossing at the initial endpoint $a_0$ can only contribute nonpositively to the Maslov index, and conversely for a crossing at the final endpoint $b_0$. For instance, in Lemma  \ref{lambdaMon} it is shown that the curve parameterized by $\lambda \in (-\infty,0]$ is negative, so the Maslov index does not change if there is a crossing right at $\lambda=0$. The fact that only negative crossings contribute at $a_0$, and similarly at $b_0$, arises from an essentially arbitrary choice made in the definition of the Maslov index, analogous to the way some authors define the contribution at each endpoint to be one half of the signature, to ensure an overcount does not occur when two segments are added up (see \cite{rs93}).

\subsection{Two abstract perturbation results} In this subsection we present two simple but useful abstract perturbation results needed in the  proof of Proposition \ref{smoothinFLG} below. In particular, they allow one to replace the preimage of a linear subspace under the action of a Fredholm operator of index zero by its preimage under the action of an invertible operator obtained by a finite rank perturbation.
\begin{remark}\label{Ainv}
We will repeatedly use the following elementary fact: if $A\in\cB(\cX)$ is boundedly invertible and $A-I_\cX\in\cB_\infty(\cX)$ then $A^{-1}-I_\cX\in\cB_\infty(\cX)$. This holds since $\cB_\infty(\cX)$ is
an ideal in $\cB(\cX)$.
\hfill$\Diamond$
\end{remark}

\begin{lemma}\label{lemL6}
Let $\{A_s\}_{s\in\Sigma}$, be a family of  operators in $\cB(\cX)$ and $\cL$ be a closed linear subspace of $\cX$ such that the following assumptions hold for each $s\in\Sigma$:
\begin{itemize}\item[(a)] $A_s-I_\cX\in\cB_\infty(\cX)$;
\item[(b)] $\ran(A_s)+\cL=\cX$;
\item[({c})] the function $s\mapsto A_s\in\cB(\cX)$ is in $C^k(\Sigma;\cB(\cX))$ for some $k\in\{0,1,\dots\}$.
\end{itemize}
Then the preimages $\cK_s$ of $\cL$ under the action of $A_s$, that is, the subspaces
\begin{equation}\label{dfK}\cK_s=\big\{u\in\cX: A_su\in\cL\big\}, \, s\in\Sigma,\end{equation}
form a $C^k$-smooth family in the sense of Definition \ref{defcksb}.

Specifically, for each $s_0\in\Sigma$ there is neighborhood $\Sigma_0$ in $\Sigma$ containing $s_0$ and a family of operators $\{B_s\}_{s\in\Sigma_0}$ in $\cB(\cX)$ such that $A_s-B_s$ is of finite rank and the following assertions hold for each $s\in\Sigma_0$:
\begin{itemize}\item[(i)]  $B_s$ is a boundedly invertible operator and $B_s-I_\cX$, $B_s^{-1}-I_\cX\in\cB_\infty(\cX)$;
\item[(ii)] $\cK_s=B_s^{-1}(\cL):=\{B_s^{-1}u: u\in\cL\}$;
\item[({iii})] the function $s\mapsto B_s\in\cB(\cX)$ is in $C^k(\Sigma_0;\cB(\cX))$.
\end{itemize}
Furthermore, if $Q$ is any projection in $\cX$ such that $\ran(Q)=\cL$, then $P_s=B_s^{-1}QB_s$ is a projection in $\cX$ such that $\ran(P_s)=\cK_s$, $s\in\Sigma_0$,  and the function $s\mapsto P_s$ is in $C^k(\Sigma_0;\cB(\cX))$.
\end{lemma}
\begin{proof} First, we claim that if  (i) holds then (ii) is equivalent to the assertion
\begin{itemize}\item[(ii')] $\ker\big((I_{\cX}-Q)A_s\big)=\ker\big((I_{\cX}-Q)B_s\big)$.
\end{itemize}
Indeed, since $\cL=\ran(Q)=\ker(I_{\cX}-Q)$ we have $\cK_s=\ker\big((I_{\cX}-Q)A_s\big)$ by \eqref{dfK}. If (ii') holds then
\begin{align*}
\cK_s&=\ker\big((I_{\cX}-Q)A_s\big)=\ker\big((I_{\cX}-Q)B_s\big)\\
&=\ker\big(B_s^{-1}(I_{\cX}-Q)B_s\big)=\ran\big(B_s^{-1}QB_s\big)
=\ran\big(B_s^{-1}Q\big)=B_s^{-1}(\cL),
\end{align*}
yielding (ii). On the other hand, if (ii) holds then
\begin{align*}
\ker\big((I_{\cX}-Q)A_s\big)&=\cK_s=B_s^{-1}(\cL)
=\ran\big(B_s^{-1}Q)\big)=\ran\big(B_s^{-1}QB_s\big)\\&=\ker\big(B_s^{-1}(I_{\cX}-Q)B_s\big)=\ker\big((I_{\cX}-Q)B_s\big),
\end{align*}
yielding (ii') and completing the proof of the claim. In addition,
we have proved the identity $\cK_s=\ran (P_s)=\ran\big(B_s^{-1}QB_s\big)$
needed for  the last statement in the lemma.
Also, if $B_s$ is boundedly invertible  then 
$ B_s-I_\cX\in\cB_\infty(\cX)$ implies $B_s^{-1}-I_\cX\in\cB_\infty(\cX)$ by Remark \ref{Ainv}.

Next, we will construct $B_s$ satisfying (i), (ii'), (iii). Fix any $s_0\in\Sigma$. Re-parametrizing, with no loss of generality we may assume that $s_0=0$. 
Since $A_0$ is a continuous operator, $\cK_0$ is closed.
Since $\ker(A_0)\subset\cK_0$ by \eqref{dfK}, there is a closed subspace $\cN\subset\cK_0$ such that $\cK_0=\ker(A_0)\oplus\cN$.
Since $\cN\subset\big(\ker(A_0)\big)^\bot$, the operator $A_0\colon\cN\to A_0(\cN)=\{A_0u:u\in\cN\}$ is a bijection. In particular, the subspace 
 $A_0(\cN)$ is closed. Since $\cN\subset\cK_0$, we have $A_0(\cN)\subset\cL$ and thus there is a subspace $\cM\subset\cL$ such that $A_0(\cN)\oplus\cM=\cL$. We claim that $\ran(A_0)\cap\cM=\{0\}$. Indeed, $v=A_0u\in\cM\subset\cL$ yields $u\in\cK_0$ and thus $v\in A_0(\cN)$ which implies $v=0$ because $A_0(\cN)\cap\cM=\{0\}$, justifying the claim. Using assumption (b) and $A_0(\cN)\oplus\cM=\cL$ we have
$\cX=\ran(A_0)+\cL=\ran(A_0)\dot{+}\cM$,
and thus $\dim(\cM)=\codim\big(\ran(A_0)\big)=\dim(\ker(A_0))$ since $A_0$ has zero index by assumption (a). Let $S_0$ denote any (finite-dimensional) isomorphism $S_0\colon\ker(A_0)\to\cM$ and let $R_0$ denote the orthogonal projection of $\cX$ onto $\ker(A_0)$. We now introduce the operators $B_s$ by the formula
\begin{equation}\label{dfBs}
B_s=A_s+S_0R_0,\quad s\in\Sigma.
\end{equation} Then assertion (ii') holds because
\[ (I_{\cX}-Q)B_s=(I_{\cX}-Q)A_s+(I_{\cX}-Q)S_0R_0=(I_{\cX}-Q)A_s\]
due to the inclusions $\ran(S_0R_0)\subset\cM\subset\cL=\ran (Q)=\ker(I_{\cX}-Q)$. Clearly (iii) follows from (c) and \eqref{dfBs}. It remains to show that $B_s$ is invertible for all $s$ in a small neighborhood $\Sigma_0$ of $\Sigma$ containing $0$. Since the function $s\mapsto A_s$ is in $C^0(\Sigma;\cB(\cX))$ by (c), the required assertion follows from \eqref{dfBs} as soon as we know that $B_0=A_0+S_0R_0$ is an invertible operator from $\cX=(\ker(A_0))^\bot\oplus\ker(A_0)$ onto $\cX=\ran(A_0)\dot{+}\cM$. But this is indeed the case since the restricted operators \[B_0\big|_{(\ker(A_0))^\bot}=A_0\big|_{(\ker(A_0))^\bot}
\quad\text{and}\quad B_0\big|_{\ker(A_0)}=S_0\]
are isomorphisms onto $\ran(A_0)$ and $\cM$ respectively.
\end{proof}
\begin{lemma}\label{lemAF}
Let $\Sigma=[a,b]\subset\bbR$ be a set of parameters, and let $\{Q_s\}_{s\in\Sigma}$ be a family of projections in $\cX$ such that the function $s\mapsto Q_s$ is in $C^k(\Sigma;\cB(\cX))$ for some $k\in\{0,1,\dots\}$. Let $D_s$ be a family of operators acting from $\ran(Q_s)$ into $\cX$ such that for the family of operators $A_s=D_sQ_s+(I_\cX-Q_s)$ the following assumptions hold:
\begin{itemize}\item[(a)] $A_s-I_\cX\in\cB_\infty(\cX)$;
\item[(b)] the operator $D_s\in\cB(\ran(Q_s),\cX))$ is injective;
\item[({c})] the function $s\mapsto A_s$ is in $C^k(\Sigma;\cB(\cX))$.
\end{itemize}
Then for each $s_0\in\Sigma$ there is neighborhood $\Sigma_0$ in $\Sigma$ containing $s_0$ and a family of operators $\{F_s\}_{s\in\Sigma_0}$ in $\cB(\cX)$ such that $A_s-F_s$ is of finite rank and the following assertions hold for each $s\in\Sigma_0$:
\begin{itemize}
\item[(i)]  $F_s$ is a boundedly invertible operator and $F_s-I_\cX$, $F_s^{-1}-I_\cX\in\cB_\infty(\cX)$;
\item[(ii)] $\ran(F_sQ_s)=\ran(A_sQ_s)$;
\item[(iii)] the function $s\mapsto F_s\in\cB(\cX)$ is in $C^k(\Sigma_0;\cB(\cX))$.
\end{itemize} 
\end{lemma}
\begin{proof} Fix any $s_0\in\Sigma$. Reparametrizing, with no loss of generality we may assume that $s_0=0$.
First, we claim that
\beq\label{dKdK}
\dim\big(\ker (A_0)\big)=\dim\big((I_\cX-Q_0)\ker(A_0)\big).
\enq
Indeed, the inequality ``$\ge$'' in \eqref{dKdK} is trivial since if vectors $\{u_j\}\subset\ker(A_0)$ are linearly dependent then the vectors $\{(I_\cX-Q_0)u_j\}$ are linearly dependent. To prove the inequality ``$\le$ in \eqref{dKdK}, let us choose linearly independent vectors $\{u_j\}\subset\ker(A_0)$ and suppose that  the vectors $\{(I_\cX-Q_0)u_j\}$ are linearly dependent so that $\sum c_j(I_\cX-Q_0)u_j=0$. Since $A_0=D_0Q_0+(I_\cX-Q_0)$, we observe that if $u\in\ker(A_0)$ then $(-D_0)Q_0u=(I_\cX-Q_0)u$.  Applying the latter equality to $u=\sum c_ju_j$, and recalling that $D_0$ is injective by assumption (b), we conclude that the vectors $\{Q_0u_j\}$ are also linearly dependent in contradiction with the linear independence of  $u_j=Q_0u_j+(I_\cX-Q_0)u_j$, thus concluding the proof of claim \eqref{dKdK}.

The operator  $A_0$ is Fredholm and has index zero by assumption (a). Using this and claim
\eqref{dKdK} we observe that $\dim\big((\ran(A_0))^\bot\big)=\dim\big((I_\cX-Q_0)\ker(A_0)\big)$. Let $S_0\colon(I_\cX-Q_0)\ker(A_0)\to(\ran(A_0))^\bot$ be any (finite-dimensional) bijection and let $R_0$ be the orthogonal projection of $\cX$ onto $(I_\cX-Q_0)\ker(A_0)$. We now introduce the operators $F_s$ by the formula
\begin{equation}\label{dfFs}
F_s=A_s+S_0R_0(I_\cX-Q_0)(I_\cX-Q_s),\quad s\in\Sigma.
\end{equation}
Clearly $A_s-F_s$ is of finite rank since $S_0$ is of finite rank by definition. Thus $F_s-I_\cX\in\cB_\infty(\cX)$ by assumption (a) for all $s\in\Sigma$. If $F_s$ is boundedly invertible then $F_s-I_\cX\in\cB_\infty(\cX)$ implies $F_s^{-1}-I_\cX\in\cB_\infty(\cX)$ by Remark \ref{Ainv}. Assertion (ii) follows since $A_sQ_s=F_sQ_s$ by \eqref{dfFs} while (iii) holds by the assumptions on $Q_s$ and ({c}). 
It remains to show that $F_s$ is invertible for all $s$ in a small neighborhood $\Sigma_0$ of $\Sigma$ containing $0$. Since the function $s\mapsto F_s$ is in $C^0(\Sigma;\cB(\cX))$, the required assertion follows as soon as we know that
$F_0$ is invertible.

Since $F_0-I_\cX\in\cB_\infty(\cX)$,  the invertibility of $F_0$ follows from $\ker(F_0)=\{0\}$. To begin the proof of the latter assertion, we choose $u\in\ker(F_0)$. Then by \eqref{dfFs} at $s=0$ the vector $A_0u=-S_0R_0(I_\cX-Q_0)u$ belongs to both subspaces $\ran(A_0)$ and $\ran(S_0)=(\ran(A_0))^\bot$ and therefore is the zero vector. Using $A_0u=0$ and $A_0=D_0Q_0+(I_\cX-Q_0)$ we conclude that $(I_\cX-Q_0)u=-D_0Q_0u$. Since $S_0$ is injective, using $S_0R_0(I_\cX-Q_0)u=0$ we conclude that
$R_0(I_\cX-Q_0)u=0$. But $u\in\ker(A_0)$ and thus $(I_\cX-Q_0)u\in(I_\cX-Q_0)\ker(A_0)$. Since $R_0$ projects onto $(I_\cX-Q_0)\ker(A_0)$ we therefore conclude that
$0=R_0(I_\cX-Q_0)u=(I_\cX-Q_0)u=-D_0Q_0u$. Since $D_0$ is injective by assumption (b), this also shows  $Q_0u=0$ and thus $u=0$, as needed. \end{proof}

\section{A symplectic view of the eigenvalue problem}\label{sec4}

In this section we describe the eigenvalue problem \eqref{BVP1}, \eqref{BVP2} in terms of the intersections of a path of Lagrangian subspaces with a fixed subspace $\cG$. The path is formed by transforming the boundary value problems on the shrunken domain $\Om_t$ back to the original domain $\Omega$, and taking boundary traces of the weak solutions to the rescaled eigenvalue equations. We introduce the rescaled differential operators involved in our analysis and study their selfadjointness. We also prove continuity and piecewise smoothness of the path
and show how it can be described via the Dirichlet-to-Neumann and Neumann-to-Dirichlet operators.

\subsection{Rescaling and the related selfadjoint differential operators}\label{sub4.1}
We begin by rescaling \eqref{BVP1} and \eqref{BVP2} to obtain the family of the eigenvalue problems \eqref{tBVP1} and \eqref{tBVP2} parametrized by $t\in(0,1]$. Recalling that $\Om\subset\bbR^d$ and
\[\Omega_t=\{z\in\Omega: z=t' y\,\text{ for }\, t'\in[0,t),\,y\in\dOm\},\] let us consider the following unitary operators:
\begin{align*}
U_t\colon &L^2(\Om_t)\to L^2(\Om),\quad (U_tw)(x)=t^{d/2}w(tx),\,x\in\Om,\\
U_t^\partial\colon &L^2(\dOm_t)\to L^2(\dOm),\quad (U_t^\partial h)(y)=t^{(d-1)/2}h(ty),\,y\in\dOm,\\
U_{1/t}^\partial\colon &L^2(\dOm)\to L^2(\dOm_t),\quad (U_{1/t}^\partial f)(z)=t^{-(d-1)/2}f(t^{-1}z),\,z\in\dOm_t,
\end{align*}
so that $(U_t^\partial)^*=U_{1/t}^\partial$ on $L^2$. These also define bounded operators on the appropriate Sobolev spaces, i.e., $U_t\in\cB(H^1(\Om_t),H^1(\Om))$ and 
$U_t^\partial\in\cB(H^{1/2}(\dOm_t),H^{1/2}(\dOm))$. The boundedness of the operator $U_t^\partial$ follows from the boundedness of $U_t$ and the boundedness of $\gaD$. Also, we define $U_t^\partial\colon H^{-1/2}(\dOm_t)\to H^{-1/2}(\dOm)$ by letting
\begin{equation}\lb{dfnUdt}
\langle U_t^\partial g,\phi\rangle_{1/2}=
{}_{H^{-1/2}(\dOm_t)}\langle  g,U_{1/t}^\partial \phi\rangle_{H^{1/2}(\dOm_t)},
\, \phi\in H^{1/2}(\dOm).
\end{equation}
For a subspace $\cG\subset\cH$ of the boundary space 
$\cH=H^{1/2}(\dOm)\times H^{-1/2}(\dOm)$ we let  $\cG_t\subset\cH_t$ be 
the subspace of  
$\cH_t=H^{1/2}(\dOm_t)\times H^{-1/2}(\dOm_t)$ defined by
\begin{equation}\label{defcGt}
\cG_t=U_{1/t}^\partial(\cG)=\big\{(U_{1/t}^\partial f, U_{1/t}^\partial g):\, (f,g)\in\cG\big\}.
\end{equation}
Recalling the definition of the rescaled trace map $\tr_tu = (\gaD  u,t^{-1}\gaN  u)$ from \eqref{defOmt}, and defining the Dirichlet and Neumann trace maps $\gamma_{{}_{D,\dOm_t}}$ and $\gamma_{{}_{N,\dOm_t}}$ on  $\dOm_t$ as in \eqref{6.1} and \eqref{dfngaN}, we have the following result.
\begin{lemma}\label{lemResc} Assume Hypothesis \ref{h1}(i). Let $\cG$ be a given subspace of the boundary space $\cH$ and let $\cG_t$ be the subspace defined in \eqref{defcGt} for some $t\in(0,1]$.
Then a function $w\in H^1(\Om_t)$ is a weak solution to the boundary
value problem 
\begin{align}\label{BVPomt1}
-\Delta w&+V(x)w=\lambda w,\quad x\in\Omega_t,\\
\tr_{\dOm_t} w&=(\gamma_{{}_{D,\dOm_t}}w,\gamma_{{}_{N,\dOm_t}}w)\in\cG_t,
\label{BVPomt2}
\end{align} 
if and only if $u=U_tw\in H^1(\Om)$ is a weak solution to the boundary value problem 
\begin{align}\label{BVPom1}
-\Delta u&+t^2V(tx)u=t^2\lambda u,\quad x\in\Omega,\\
\tr_t u&=(\gaD u,t^{-1}\gaN u)\in\cG.
\label{BVPom2}\end{align}
\end{lemma}
\begin{proof} We note that $(U_t)^*=(U_t)^{-1}=U_{1/t}$. We denote by $z$ points of $\Om_t$ and by $w$ functions on $\Om_t$. Applying the chain rule to $w\in H^1(\Om_t)$ yields
\beq\lb{r1}
(\nabla_xU_tw)(x)=t^{d/2}\cdot t(\nabla_zw)(tx)=t(U_t\nabla_zw)(x),\, x\in\Om, z=tx\in\Om_t.
\enq
This implies that the operator $U_t$ is indeed in $\cB(H^1(\Om_t),H^1(\Om))$ because
\[\langle\nabla_xU_tw,\nabla_xU_tw\rangle_{L^2(\Om)}=
t^2\langle U_t\nabla_zw,U_t\nabla_xw\rangle_{L^2(\Om)}=t^2\|\nabla_z w\|_{L^2(\Om_t)}^2,\]
and also that 
\beq\lb{r2}
(\Delta_xU_tw)(x)=t^{d/2}\cdot t^2(\Delta_zw)(tx)=t^2(U_t\Delta_zw)(x),\,x\in\Om.\enq
It follows that $w\in H^1(\Om_t)$ is a weak solution of equation \eqref{BVPomt1} on $\Om_t$ if and only if $u=U_tw\in H^1(\Om)$ is a weak solution of \eqref{BVPom1} on $\Om$. Indeed, to see this we  note that $u\in H^1(\Om)$ is a weak solution of \eqref{BVPom1} provided
\[\langle\nabla_xu,\nabla_x\Phi\rangle_{L^2(\Om)}+
t^2\langle V(tx)u,\Phi\rangle_{L^2(\Om)}=\lambda t^2\langle u,\Phi\rangle_{L^2(\Om)}
\,\text{ for all } \Phi\in H^1_0(\Om).\] Letting $\Phi=U_t\Psi$, $u=U_tw$ and using \eqref{r1} we conclude that the last equation holds if and only if
\[\langle\nabla_zw,\nabla_z\Psi\rangle_{L^2(\Om_t)}+
\langle V(z)w,\Psi\rangle_{L^2(\Om_t)}=\lambda \langle w,\Psi\rangle_{L^2(\Om_t)}
\,\text{ for all } \Psi\in H^1_0(\Om_t),\]
as required. It remains to take care of the boundary conditions in \eqref{BVPomt2} and \eqref{BVPom2}. We claim the following assertions for the boundary trace operators:
\begin{align}\lb{d-rule}
&\text{If $w\in H^1(\Om_t)$ then $\gaD U_t w=t^{1/2}U_t^\partial \gamma_{{}_{D,t}} w$}\\
&\text{If $w\in\dom(\gamma_{{}_{N,t}})$ then $U_tw\in\dom(\gaN)$ and 
$\gaN U_tw=t^{3/2}U_t^\partial  \gamma_{{}_{N,t}} w$}.\lb{n-rule}
\end{align}
Assuming the claim, we conclude the proof of the lemma as follows: For any $w\in\dom(\gamma_{{}_{N,t}})=\dom(\tr_{\dOm_t})$ we have
\begin{align*}
\tr_tU_tw&=(\gaD U_tw,t^{-1}\gaN U_tw)\\&=(t^{1/2}U_t^\partial  \gamma_{{}_{D,t}} w,
t^{1/2}U_t^\partial  \gamma_{{}_{N,t}} w)=t^{1/2}U_t^\partial\tr_{\dOm_t}w,
\end{align*}
and thus $\tr_tU_tw\in\cG$ if and only if $\tr_{\dOm_t}w\in\cG_t$ by \eqref{defcGt}.

It remains to justify the claim. To prove \eqref{d-rule}, it suffices to work with $\gaD^0$ from \eqref{2.4} and $w\in C^0(\overline{\Om})$ for which we have 
\begin{align*}
\gaD^0U_tw&=t^{d/2}w(tx)|_{x\in\dOm}=t^{d/2}w(z)|_{z\in\dOm_t},\\
U_t^\partial\gaD^0w&=t^{(d-1)/2}w(z)|_{z=tx\in\dOm_t}=t^{-1/2}t^{d/2}w(z)|_{z\in\dOm_t}.
\end{align*}
This yields \eqref{d-rule}. To begin the proof of  \eqref{n-rule}, we apply Green's formula \eqref{wGreen} for 
$u=U_tw$ and $\Phi=U_t\Psi$, use \eqref{r1}, \eqref{r2}, the fact that $U_t$ is a unitary operator, 
and then Green's formula again to infer:
\begin{align}
\langle\gaN(U_tw),&\gaD(U_t\Psi)\rangle_{1/2}=
\langle\nabla_x(U_tw),\nabla_x(U_t\Psi)\rangle_{L^2(\Om)}+\langle\Delta(U_tw),U_t\Psi\rangle_{L^2(\Om)}\nonumber\\&=
t^2\langle U_t\nabla_zw,U_t\nabla_z\Psi\rangle_{L^2(\Om)}+
t^2\langle U_t\Delta_zw,U_t\Psi\rangle_{L^2(\Om)}\nonumber\\&=
t^2\langle \nabla_zw,\nabla_z\Psi\rangle_{L^2(\Om_t)}+
t^2\langle \Delta_zw,\Psi\rangle_{L^2(\Om_t)}\nonumber\\&=
t^2{}_{H^{-1/2}(\dOm_t)}\langle\gamma_{{}_{N,t}}w,\gamma_{{}_{D,t}}\Psi\rangle_{H^{1/2}(\dOm_t)}.\lb{dfnUdt1}
\end{align}
Next, we use the definition of $U_t^\partial$ in \eqref{dfnUdt},
\begin{equation*}
\langle U_t^\partial \gamma_{{}_{N,t}} w,\gaD\Phi\rangle_{1/2}=
{}_{H^{-1/2}(\dOm_t)}\langle  \gamma_{{}_{N,t}} w,U_{1/t}^\partial \gaD\Phi\rangle_{H^{1/2}(\dOm_t)}.
\end{equation*}
 Applying \eqref{d-rule} 
we have $U_{1/t}^\partial\gaD\Phi=t^{1/2}\gamma_{{}_{D,t}}U_{1/t}\Phi=t^{1/2}\gamma_{{}_{D,t}}\Psi$ yielding
\[\langle U_t^\partial \gamma_{{}_{N,t}} w,\gaD\Phi\rangle_{1/2}=t^{1/2}{}_{H^{-1/2}(\dOm_t)}\langle  \gamma_{{}_{N,t}} w,\gamma_{{}_{D,t}}\Psi\rangle_{H^{1/2}(\dOm_t)}.\] 
Combined with \eqref{dfnUdt1} this implies  
$\langle\gamma_{{}_{N}} U_tw,\phi\rangle_{1/2}=
t^{3/2}\langle U^\partial_t\gamma_{{}_{N,t}} w,\phi\rangle_{1/2}$
for any $\phi=\gaD\Phi\in H^{1/2}(\dOm)$ and thus \eqref{n-rule} holds.
\end{proof}
We will now define a path $\Upsilon$ in the set of Lagrangian subspaces in $H^{1/2}(\dOm)\times H^{-1/2}(\dOm)$ by taking traces of weak solutions to the rescaled equation \eqref{BVPom1} introduced in Lemma \ref{lemResc}; see \eqref{dfnups} below. 

To begin, let us fix $\tau \in[0,1]$ and $\Lambda>0$ and introduce the parameter sets
\beq\lb{dfnSigmaj}\begin{split}
\Sigma_1&=[-\Lambda,0],\, \Sigma_2=[0,1-\tau ],\\ \Sigma_3&=[1-\tau ,1-\tau +\Lambda],\,
\Sigma_4=[1-\tau +\Lambda, 2(1-\tau )+\Lambda],\, \Sigma=\cup_{j=1}^4\Sigma_j.\end{split}
\enq
Next, we introduce functions $t(\cdot)$, $\lambda(\cdot)$ in a way that the boundary $\Gamma=\cup_{j=1}^4\Gamma_j$ of the square $[-\Lambda,0]\times[\tau ,1]$ is parametrized with $(\lambda(s),t(s))\in\Gamma_j$ when $s\in\Sigma_j$ for $j=1,\dots,4$ (see Figure \ref{fig1}) and $\Gamma$ is oriented counterclockwise:
\beq\lb{dfnlambdat}
\begin{split}
\lambda(s)&=s,\quad t(s)=\tau ,\quad s\in\Sigma_1,\\
\lambda(s)&=0,\quad t(s)=s+\tau ,\quad s\in\Sigma_2,\\
\lambda(s)&= -s+1-\tau ,\quad t(s)= 1,\quad s\in\Sigma_3,\\
\lambda(s)&=-\Lambda,\quad
t(s)=-s+2-\tau +\Lambda,\quad s\in\Sigma_4.
\end{split}
\enq
In what follows, cf. Lemmas \ref{lem:Gamma4} and \ref{lem:Gamma4Dir}, for a fixed $\tau \in[0,1]$ we will choose $\Lambda=\Lambda(\tau )$ and then define $\Sigma_j=\Sigma_j(\tau )$ and $\Gamma_j=\Gamma_j(\tau )$ as in \eqref{dfnSigmaj} and \eqref{dfnlambdat}.

For $\tau \in[0,1]$, using the functions $\lambda(\cdot)$ and $t(\cdot)$ defined in \eqref{dfnSigmaj} and \eqref{dfnlambdat}, we define the following family of operators
$L_s=L_s(\tau )$ acting from $H^1(\Omega)$ into $H^{-1}(\Omega)$,
\beq\lb{dfnLst0}
L_su=-\Delta u+t^2(s)V(t(s)x)u-\lambda(s)t^2(s)u,\quad s\in\Sigma=\Sigma(\tau),\enq
where, as usual, we define the Laplacian in the weak sense and at this stage do not impose any boundary conditions. We stress that $L_s$ (and all other operators in this section) are acting on the real space $H^1(\Om;\bbR^N)$, although one can make the space and the operator complex by a standard procedure, cf., e.g., \cite[Section 5.5.3]{We80}.

For the operator $L_s=L_s(\tau)$ defined in \eqref{dfnLst0} we denote by $\cK_s$ the set of weak solutions to the equation $L_su=0$, that is, for each $s\in\Sigma$ we set
\beq\lb{dfnKS}\begin{split}
\cK_s&=\big\{u\in H^1(\Om): L_su=0\text{ in } H^{-1}(\Om)\big\}\\
&=\big\{u\in H^1(\Om): \langle\nabla u,\nabla\Phi\rangle_{L^2(\Om)}
\\&\hskip1cm+\langle t^2(s)\big(V(t(s)x)-\lambda(s)\big)u,\Phi\rangle_{L^2(\Om)}=0\text{ for all $\Phi\in H^1_0(\Om)$}\big\}.
\end{split}\enq
For any $\tau\in(0,1]$ and $s\in\Sigma(\tau)$ we introduce the rescaled trace map $T_s$ by the formula
\beq\lb{dfnts}
T_su=(\gaD u, (t(s))^{-1}\gaN u),\, u\in\dom(\gaN),\, s\in\Sigma=\Sigma(\tau).
\enq
In particular, if $s\in\Sigma_2$ then $T_s=\tr_t$ with $t=t(s)\in[\tau,1]$ for the rescaled trace defined in \eqref{defOmt}.
Finally, we define the desired path of subspaces in $\cH=H^{1/2}(\dOm)\times H^{-1/2}(\dOm)$ as follows:
\beq\lb{dfnups}
\Upsilon(s)=T_s(\cK_s),\quad s\in\Sigma=\Sigma(\tau).
\enq
In particular, for $s\in\Sigma_2$ and $t=t(s)$, $\Upsilon\big|_{\Sigma_2}\colon t\mapsto\tr_t(\cK_t)$ is the path used in Theorems \ref{tim:Dbased} and \ref{tim:Nbased}.

We will now define operators on $L^2(\Om)$ associated with the differential expression $L_s=L_s(\tau)$ from \eqref{dfnLst0}. 
Given a subspace $\cG$ in $\cH=H^{1/2}(\dOm)\times H^{-1/2}(\dOm)$ we define an operator $L_{s,\cG}(\tau )$ on $L^2(\Om)$ for each $\tau \in[0,1]$ by
\beq\lb{dfnLst01}\begin{split}
L_{s,\cG}(\tau )u&=L_su, \quad s\in\Sigma=\Sigma(\tau),\\
\dom(L_{s,\cG}(\tau ))&=\big\{u\in H^1(\Omega): \Delta u\in L^2(\Omega), (t(s)\gaD u,\gaN u)\in\cG\big\},\end{split}
\enq
where $t(s)$ is defined in \eqref{dfnlambdat}. 
In particular, if $s=0$, then $t(0)=\tau $ and $\lambda(0)=0$, hence
\beq\lb{dfnLst00}\begin{split}
L_{0,\cG}(\tau)u&=L_0(\tau)u,\,
L_0(\tau )u=-\Delta u+\tau ^2V(\tau x)u,\quad \tau \in[0,1],\\
\dom(L_{0,\cG}(\tau ))&=\big\{u\in H^1(\Omega): \Delta u\in L^2(\Omega), (\tau \gaD u,\gaN u)\in\cG\big\}.\end{split}
\enq
If $\tau\in(0,1]$ then $t(s)\neq0$ for all $s\in\Sigma$ and the condition $(t(s)\gaD u,\gaN u)\in\cG$ in \eqref{dfnLst01}, respectively, $(\tau\gaD u,\gaN u)\in\cG$ in \eqref{dfnLst00} is equivalent to the condition $(\gaD u, (t(s))^{-1}\gaN u)\in\cG$, respectively, $(\gaD u,\tau^{-1}\gaN u)\in\cG)$. 
 In particular, if $\tau =1$ and $s=0$ then $\lambda(0)=0$, $t(0)=1$ and $L_{0,\cG}(1)=L_\cG$,
 where $L_\cG$ is the operator defined in \eqref{dfLG} whose Morse index we intend to characterize.
 
 \begin{remark}\lb{sbb}
 Hypothesis \ref{h1} (iii) requires the operators $L_{s,\cG}(\tau)$ to be semibounded from below, uniformly for $s\in\Sigma$. For this to hold it is enough to require uniformity for $s\in\Sigma_2$. To see this, we note first that if $s'\in\Sigma_4$ then $s=-s'+2(1-\tau)+\Lambda\in\Sigma_2$ satisfies $t(s)=t(s')$ and thus $L_s=L_{s'}-\Lambda (t(s))^2$. We conclude that $L_s-L_{s'}$ is a multiple of the identity and thus the uniform boundedness from below for $s\in\Sigma_2$ implies the uniform boundedness for $s'\in\Sigma_4$. Next, if $s'\in\Sigma_1$, then $s=0\in\Sigma_2$ is such that $t(s)=\tau=t(s')$ while if $s\in\Sigma_3$ then $s=1-\tau\in\Sigma_2$ is such that $t(s)=1=t(s')$, and so the lower bound extends uniformly to $\Sigma_1$ and $\Sigma_3$.
 \hfill$\Diamond$\end{remark}

We recall Definition \ref{defDiNeB}, in which the subspace $\cG$ is described as a graph or inverse graph of compact, selfadjoint operators $\Theta, \Theta'$ and note that a sufficient condition for the operator $\Theta'$, respectively $\Theta$, to be compact is $\Theta'\in\cB\big(H^{-1/2}(\dOm),H^{1/2+\varepsilon}(\dOm)\big)$, respectively $\Theta\in\cB\big(H^{1/2-\varepsilon}(\dOm),H^{-1/2}(\dOm)\big)$, for some $\varepsilon>0$.

If $\cG=\gr'(\Theta')$ is Dirichlet-based with $\Theta'\in\cB(H^{-1/2}(\dOm), H^{1/2}(\dOm))$, and $\tau>0$,
then the operator $L_{s,\cG}(\tau)$ 
is given by
\beq\label{LsGD}
\begin{split}
L_{s,\cG}(\tau)u(x)&=-\Delta u(x)+t^2(s)V(t(s)x)u(x)
-\lambda(s)t^2(s)u(x),\, x\in\Om,\\
 \dom(L_{s,\cG}(\tau))&=\big\{u\in H^1(\Omega)\,|\,\Delta u\in L^2(\Omega) \,\text{ and }\\
&\hskip3cm (t(s)\gaD-\Theta'\gaN)u=0\,\,\hbox{in}\,\,H^{1/2}(\dOm) \big\}.
\end{split}
\enq
Similarly, if $\cG=\gr(\Theta)$ is Neumann-based with $\Theta\in\cB(H^{1/2}(\dOm), H^{-1/2}(\dOm))$, and $\tau>0$,
then the operator $L_{s,\cG}(\tau)$ 
is given by
\beq\label{LsGN}\begin{split}
L_{s,\cG}(\tau)u(x)&=-\Delta u(x)+t^2(s)V(t(s)x)u(x)
-\lambda(s)t^2(s)u(x),\, x\in\Om,\\
 \dom(L_{s,\cG}(\tau))&=\big\{u\in H^1(\Omega)\,|\,\Delta u\in L^2(\Omega) \,\text{ and }\\
&\hskip3cm (\gaN-t(s)\Theta\gaD)u=0\,\,\hbox{in}\,\,H^{-1/2}(\dOm) \big\}.
\end{split}\enq
In the Neumann-based case the sesquilinear form associated with the operator $L_{s,\cG}(\tau)$, cf.\
Theorem \ref{th26GM} below, is given by
\begin{align}
\mathfrak{l}_{s,\cG}(\tau)(u,v)&=\langle\nabla u,\nabla v\rangle_{L^2(\Om)}+t^2(s)\langle \big(V(t(s)x)-\lambda(s)\big)u,v\rangle_{L^2(\Om)}\no\\&\hskip2cm
-t(s)\langle\Theta\gaD  u,\gaD   v\rangle_{1/2},
\lb{dfnlGst}\\
\dom(\mathfrak{l}_{s,\cG} (\tau))&=H^1(\Omega)\times H^1(\Omega).\no
\end{align}
If $\cG=\gr(\Theta)$ is Neumann-based and $\tau =0$, then the operator $L_{0,\cG}(0)$ for $s=0$,
\beq\lb{dfnLst00n}\begin{split}
L_{0,\cG}(0)u&=-\Delta u,\\
\dom(L_{0,\cG}(0))&=\big\{u\in H^1(\Omega): \Delta u\in L^2(\Omega), (0,\gaN u)\in\gr(\Theta)\big\},\end{split}
\enq
is just the Neumann Laplacian $-\Delta_N$ equipped with the boundary condition $\gaN u=0$. In particular, $L_{0,\cG}(0)$ is independent of the boundary operator $\Theta$.
 
 We recall our main assumptions on the operators $\Theta, \Theta'$ and the subspace $\cG$ (and therefore on the operators $L_{s,\cG}(\tau)$) summarized in Hypotheses \ref{h1} (ii), (iii).
The following lemma describes the spectrum of $L_{s,\cG}(\tau)$ in terms of the conjugates points on the path  $\{\Upsilon(s)\}_{s\in\Sigma}$ defined in \eqref{dfnups}.

\begin{lemma}\lb{GevL} Assume Hypothesis \ref{h1}. Let $\tau\in(0,1]$, let $L_{s,\cG}(\tau)$ be the family of operators on $L^2(\Om;\bbR^N)$ defined in \eqref{dfnLst01} and let $\Upsilon(s)=T_s(\cK_s)$ be the family of the subspaces in $\cH=H^{1/2}(\dOm)\times H^{-1/2}(\dOm)$ defined in \eqref{dfnups} for $s\in\Sigma=\Sigma(\tau)$. Then for any $s_0\in\Sigma$ we have 
\begin{align}\lb{eq:4.23}
&0\in\Sp(L_{s_0,\cG}(\tau)) \text{ if and only if $
T_{s_0}(\cK_{s_0})\cap\cG\neq\{0\}$,  and } \\
& \dim_\bbR\big(\ker(L_{s_0,\cG}(\tau))\big)=\dim_\bbR\big(T_{s_0}(\cK_{s_0})\cap\cG\big).\label{eqdims}\end{align}
Moreover, 
\begin{align}
0\in\Sp(L_{s_0,\cG}(\tau ))\, &\text{ iff }\, \lambda(s_0)t^2(s_0)=s_0\tau ^2\in\Sp(L_{0,\cG}(\tau ))\,
\text{ for } s_0\in\Sigma_1,\lb{4.22}\\
0\in\Sp(L_{s_0,\cG}(\tau ))\, &\text{ iff }\,  \lambda(s_0)=-s_0+1-\tau \in\Sp(L_\cG)\, \text{for }\, s_0\in\Sigma_3,\lb{4.21}
\end{align}
that is, $s_0\in\Sigma_3$ is a crossing if and only if $\lambda=\lambda(s_0)$ is an eigenvalue of the operator $L_\cG$ defined in \eqref{dfLG}.
\end{lemma}
\begin{proof} By Hypothesis \ref{h1}(iii), $0\in\Sp(L_{s_0,\cG}(\tau))$ if and only if 
$\ker\big(L_{s_0,\cG}(\tau)\big)\neq\{0\}$; also, $\dim_\bbR\big(\ker\big(L_{s_0,\cG}(\tau)\big)\big)<\infty$. If $u\in\ker\big(L_{s_0,\cG}(\tau)\big)$, then $u\in\cK_{s_0}$ by \eqref{dfnLst0} and \eqref{dfnKS}, hence $T_{s_0}u\in T_{s_0}(\cK_{s_0})\cap\cG$ by definition \eqref{dfnLst01}. Conversely, if $u\in\cK_{s_0}$ and $T_{s_0}u\in\cG$, then $u\in\ker\big(L_{s_0,\cG}(\tau)\big)$. Thus  $\ker\big(L_{s_0,\cG}(\tau)\big)$ and $T_{s_0}(\cK_{s_0})\cap\cG$ are isomorphic. Assertions \eqref{4.22} and \eqref{4.21} follow from \eqref{dfnlambdat}.
\end{proof}
\begin{remark}\lb{DC4.2}
The operator $L_{s,\cH_D}(\tau)$ equipped the Dirichlet boundary conditions corresponds to the case $\cG=\cH_D$ (equivalently, $\Theta'=0$). If Hypothesis \ref{h1} (i') holds then this operator  satisfies Hypothesis \ref{h1}(iii) and even Hypothesis \ref{h1} (iii'), hence $\dom(L_{s,\cH_D}(\tau))\subseteq H^2(\Om)$, see, e.g., \cite[Theorem 2.10]{GM08}. However, we are not aware of  any general sufficient conditions on the operator $\Theta'\neq0$ that guarantee Hypothesis \ref{h1}(iii).
\hfill$\Diamond$\end{remark}

We now concentrate on the Neumann-based case. In particular, when $\Theta=0$ this covers the case of Neumann boundary conditions, i.e. $\cG=\cH_N$. As we will see, the following hypothesis implies Hypothesis \ref{h1} (iii).

\begin{hypothesis}\lb{hypTHETA}
For the Neumann-based case we assume that
the operator $\Theta$ is associated with a closed, semibounded (and therefore symmetric) bilinear form  $\aT$ on $L^2(\dOm)$ with domain $\dom(\aT)=H^{1/2}(\dOm)\times H^{1/2}(\dOm)$ such that the following two conditions hold:
\begin{enumerate}\item[(a)] The form $\aT$  is $L^2(\dOm)$-semibounded from above by a constant $c_\Theta\in\bbR$ (not necessarily positive), that is,
\[\aT(f,f)\le c_\Theta\|f\|_{L^2(\dOm)}^2\quad\text{ for all }\quad f\in H^{1/2}(\dOm),\]
\item[(b)] The form $\aT$  is $H^{1/2}(\dOm)$-bounded, that is, there is a positive constant $c_{\mathfrak{a}}$ such that
\[|\aT(f,f)|\le c_{\mathfrak{a}}\|f\|_{H^{1/2}(\dOm)}^2\quad\text{ for all }\quad f\in H^{1/2}(\dOm).\]
\end{enumerate}
\hfill$\Diamond$
\end{hypothesis}

Assuming Hypothesis \ref{hypTHETA}, there exists a unique bounded, selfadjoint operator $\Theta\in\cB(H^{1/2}(\dOm), H^{-1/2}(\dOm))$, defined for $f\in H^{1/2}(\dOm)$ by 
\beq\lb{qftheta}
\langle\Theta f,g\rangle_{1/2}=\aT(f,g)\,\text{ for }\, g\in H^{1/2}(\dOm),
\enq so that $\|\Theta\|_{\cB(H^{1/2}(\dOm), H^{-1/2}(\dOm))}\le c_{\mathfrak{a}}$ and the following assertions hold:
\begin{align}\lb{propTheta1}
\langle \Theta f,g\rangle_{1/2}&=
\langle \Theta g,f\rangle_{1/2}\quad\text{for all}\quad f,g\in H^{1/2}(\dOm),\\
\langle \Theta f,f\rangle_{1/2}&\le c_\Theta\|f\|_{L^2(\dOm)}^2
\quad\text{for all}\quad f\in H^{1/2}(\dOm).\lb{propTheta2}
\end{align}
 We note again that the operator $\Theta$ here and throughout is acting between the real Hilbert spaces $H^{1/2}(\dOm; \bbR^N)$ and $H^{-1/2}(\dOm; \bbR^N)$; it can be make complex by a standard procedure, cf. \cite[Section 5.5.3]{We80}, but this will not be needed until Section \ref{sec:mainres}. We note the following result  that can found  in Theorem 2.6 and Corollary 2.7 of \cite{GM08}.
\begin{theorem}\lb{GMthm} Assume Hypothesis \ref{h1} (i) and Hypothesis \ref{hypTHETA}, and consider a potential $V\in L^\infty(\Om)$. Let $\Theta\in\cB(H^{1/2}(\dOm), H^{-1/2}(\dOm))$ be the selfadjoint operator associated by \eqref{qftheta} with the form $\aT$ and let $\cG=\gr(\Theta)$. Let $L_\cG\colon\dom(L_\cG)\to L^2(\Om)$ denote the operator defined on  $L^2(\Om)$ by
\beq\lb{fdLG}L_\cG u=-\Delta u+Vu,\,
\dom(L_\cG)=\big\{u\in H^1(\Om): \Delta u\in L^2(\Om), \tr u\in\cG\big\}.
\enq
Then Hypothesis \ref{h1}(iii) holds, that is, the operator $L_\cG$ is selfadjoint and bounded from below, has compact resolvent (and therefore only discrete spectrum) and satisfies
$\dom\big(|L_\cG|^{1/2}\big)=H^1(\Om)$.
\end{theorem}
Applying this result to the operator $t(s)\Theta$ and the potential $V_s$ defined in \eqref{defVsLs} implies the respective assertions regarding the operators $L_{s,\cG}(\tau)$ from \eqref{LsGN} described in Hypothesis \ref{h1}(iii). We emphasize the importance of the fact that $\Theta$ is semibounded from below as indicated in \eqref{propTheta2}; that $\Theta$ is selfadjoint as indicated in \eqref{propTheta1} is not enough to conclude that $L_{s,\cG}(\tau)$ is a semibounded operator. We also stress that our standing assumption $\Theta\in\cB_\infty(H^{1/2}(\dOm), H^{-1/2}(\dOm))$, see Hypothesis \ref{h1}(ii) and Definition \ref{defDiNeB}, is enforced throughout the paper. 

Condition \eqref{propTheta2} holds for the case of Robin boundary conditions when $\Theta$ is the operator of multiplication by a function $\theta\in L^\infty(\dOm)$ composed with the embedding $H^{1/2}(\dOm)\hookrightarrow H^{-1/2}(\dOm)$. In this case $c_\Theta\le c\|\theta\|_{L^\infty}$, since the left-hand side of \eqref{propTheta2} can be written as an integral over $\dOm$.

Next, we will discuss Hypothesis \ref{h1} (iii'). The Morse index of the differential operator  $L_{s,\cG}(\tau)$ from \eqref{LsGN} can be defined as the number of negative eigenvalues (counted with multiplicity) corresponding to either weak or strong solutions to the eigenfunction equation $L_{s,\cG}(\tau)u=\lambda u$. Thus far we have used weak solutions to define the Morse index.
Under some additional condition on $\Theta$ one has the inclusion $\dom(L_{s,\cG}(\tau))\subseteq H^2(\Omega)$, cf.\ Hypothesis \ref{h1} (iii').
In particular, this inclusion  implies that a weak solution to the boundary value problem $L_su=0$, $\gaN u-t(s)\Theta\gaD u=0$ is in fact a strong solution, and thus both notions of the Morse index coincide. We must assume additional smoothness of the domain (as in Hypothesis \ref{h1}(i')) and a condition on the operator $\Theta$ to ensure $\dom(L_{s,\cG}(\tau))\subseteq H^2(\Omega)$.
\begin{hypothesis} \lb{h3}
Assume Hypothesis \ref{h1}(i'), (ii) and Hypothesis \ref{hypTHETA}. In addition, assume that $\Theta\in\mathcal{B}_{\infty}(H^{3/2}(\dOm),H^{1/2}(\dOm))$. 
\end{hypothesis}
 In other words, if Hypothesis \ref{h3} holds, then
$$\Theta\in\mathcal{B}(H^{1/2}(\dOm),H^{-1/2}(\dOm))\cap\mathcal{B}_{\infty}(H^{3/2}(\dOm),H^{1/2}(\dOm))$$ and $\Theta^*=\Theta$ as an operator from $H^{1/2}(\dOm)$ into $H^{-1/2}(\dOm)$. 
A sufficient condition for $\Theta\in\mathcal{B}_{\infty}(H^{3/2}(\dOm),H^{1/2}(\dOm))$ is $\Theta\in\mathcal{B}(H^{3/2-\varepsilon}(\dOm),H^{1/2}(\dOm))$ for some $\varepsilon>0$.
The following result can be found in \cite[Theorem 2.17]{GM08}.
\begin{theorem}\lb{GMthmH2}
Assume Hypothesis \ref{h3}. Then $\dom(L_{s,\cG}(\tau))\subseteq H^2(\Om)$, that is, Hypothesis \ref{h1}(iii') holds.
\end{theorem}
As the following example shows, the condition 
$\Theta\in\mathcal{B}_\infty(H^{1/2}(\dOm),H^{-1/2}(\dOm))$ in Hypothesis \ref{h1}(iii) is not enough to guarantee the inclusion $\dom(L_{s,\cG}(\tau))\subseteq H^2(\Om)$.
We also refer to the related discussion in \cite{DJ11}, where the Morse index is defined in terms of the eigenspaces of the relevant operators understood in the strong sense. One of the main claims in \cite{DJ11}, see Section 5, was the {\em equality} of the Morse index of a certain differential operator (with domain in $H^2(\Omega)$) and the Maslov index. 
The next example therefore shows that under the condition $\Theta\in\mathcal{B}_\infty(H^{1/2}(\dOm),H^{-1/2}(\dOm))$ imposed in \cite{DJ11} only the {\em estimate from above} of the dimension in terms of the Maslov index indeed holds (e.g., only the inequality ``$\le$'' in \cite[Equation (5.3)]{DJ11} holds).

\begin{example}\lb{ex:nostr} Here we give an example of a Neumann-based boundary condition for which the sets of weak and strong solutions to the homogeneous boundary value problem do not coincide.
 Let $u\in H^1_\Delta(\Om)$ be a weakly harmonic function such that $\|\nabla u\|_{L^2(\Om)}$ is not zero  and $u$ is not contained in $H^2(\Om)$. We define a rank one operator $\Theta$ from $H^{1/2}(\partial\Omega)$ to $H^{-1/2}(\partial\Omega)$ as follows,
\begin{align*}
&\Theta h:=\frac{\langle\gaN  u,h\rangle_{1/2}}{\langle\gaN  u,\gaD  u\rangle_{1/2}}\gaN   u,\, h\in H^{1/2}(\dOm),
\end{align*}
using Green's formula \eqref{wGreen} to see that $\langle\gaN  u,\gaD  u\rangle_{1/2}=\|\nabla u\|^2_{L^2(\Om)}$ is nonzero.
Then $\Theta\in\mathcal{B}_{\infty}(H^{1/2}(\dOm),H^{-1/2}(\dOm))$ and $\Theta^*=\Theta$.
Moreover, $u$ is a weak solution to the boundary value problem
\begin{align*}
    -\Delta u=0 \,\hbox{ in }\, H^{-1}(\Om),\quad
    \gaN   u-\Theta\gaD u=0  \,\hbox{ in }\, H^{-1/2}(\dOm),
\end{align*}
which is not a strong solution since $u \notin H^2(\Om)$. \hfill$\Diamond$
\end{example}

\subsection{Regularity of the path of boundary traces of weak solutions}
Our next result, Proposition \ref{smoothinFLG}, is the main assertion in this section; it says that the traces $\Upsilon(s)=T_s(\cK_s)$  of the weak solutions to $L_su=0$ form a smooth path in the Fredholm--Lagrangian Grassmannian $F\Lambda(\cG)$ of a given Lagrangian subspace $\cG\subseteq\cH$ which is either Dirichlet- or Neumann-based, cf.\ Definition \ref{defDiNeB}. Here, $\cH=H^{1/2}(\dOm)\times H^{-1/2}(\dOm)$ is the boundary space equipped with the symplectic structure defined in \eqref{dfnomega} and $F\Lambda(\cG)$ is defined in Definition \ref{dfnFLG} (ii). We formulate this proposition for a general $C^k$-smooth family $s\mapsto V_s$ of potentials, but in the sequel will only need it for $V_s$ defined in \eqref{defVsLs}. Assuming $V(\cdot)\in C^0(\overline{\Om})$ as needed in Hypothesis \ref{h1}(iv), the function  $s\mapsto V_s$ from \eqref{defVsLs} is $C^0$ for $s\in\Sigma$ and is $C^1$ for $s\in\Sigma_1\cup\Sigma_3$ since $t(\cdot)\big|_{\Sigma_1\cup\Sigma_3}$ is constant by \eqref{dfnlambdat}.
A stronger assumption $V(\cdot)\in C^k(\overline{\Om})$ will generate a (piecewise) smooth function  $s\mapsto V_s$, for which Proposition \ref{smoothinFLG} gives a stronger result, but this will only be used in the subsequent analysis with $k=1$, cf.\ Lemma \ref{preltmon}. 
We recall Definition \ref{defcksb} of a smooth family of subspaces.

\begin{proposition}\label{smoothinFLG} Let $\Sigma=[a,b]$ be a set of parameters, $t(\cdot)\in C^k(\Sigma;[\tau ,1])$ a given function for some $\tau \in(0,1]$, and $\{V_s\}_{s\in\Sigma}$ a family of potentials such that the function $s\mapsto V_s$ is in $C^k(\Sigma; L^\infty(\Om))$ for some $k\in\{0,1,\dots\}$. For $s\in\Sigma$ and $u\in\dom(\gaN )$ we define the rescaled map $T_s$ by $T_su=\big(\gaD u,(t(s))^{-1}\gaN u\big)$ and let 
\[L_s=-\Delta+V_s,\, \cK_s=\big\{u\in H^1(\Om): L_su=0 \text{ in } H^{-1}(\Om)\big\}.\]
Then the subspaces $\{T_s(\cK_s)\}_{s\in\Sigma}$
form a $C^k$-smooth family in $\cH=H^{1/2}(\dOm)\times H^{-1/2}(\dOm)$. Moreover, if $\cG\in\Lambda(\cH)$ is a Dirichlet- or Neumann-based Lagrangian subspace, then  $T_s(\cK_s)$  belongs to the Fredholm--Lagrangian Grassmannian $F\Lambda(\cG)$ for each $s$.
\end{proposition}

Throughout the proof, see in particular assertion (iv) in \eqref{Assiiv}, we will use Lemmas \ref{lem:propFG} and \ref{lemL6} to establish the existence 
 of a family of projections onto $\{\cK_s\}$ which is contained $C^k(\Sigma_0;\cB(H^1(\Om)))$.

\begin{proof} If $u\in\cK_s$, then $-\Delta u=-V_su\in L^2(\Om)$ yields $u\in\dom(\gaN )$, and thus $\cK_s\subset\dom(\gaN )$ shows that $T_s(\cK_s)$ is well defined.
Fix any $s_0\in\Sigma$. Reparametrising, with no loss of generality  we may assume that $s_0=0$. By Remark \ref{locglob}, to establish that $T_s(\cK_s)$ is a smooth family we need to construct  a family of projections $P_s$ in $\cH$ and a neighborhood $\Sigma_0\ni 0$ such that $\ran(P_s)=T_s(\cK_s)$ and the function $s\mapsto P_s$ is in $C^k(\Sigma_0;\cB(\cH))$.

To begin the construction, let us consider the operators $G_s$ on $H^1(\Om)$ defined as in \eqref{dfFG} but with $V$ replaced by $V_s$ so that $G_s-I_{\cH^1(\Om)}\in\cB_\infty(\cH^1(\Om))$. Our assumptions on the function $s\mapsto V_s$ yield that $s\mapsto G_s$ is in $C^k(\Sigma; \cB(H^1(\Om)))$. Thanks to Lemma \ref{lem:propFG} we know that
\beq\label{Ksplus}
\cK_s=\big\{u\in H^1(\Om): G_su\in H^1_\Delta(\Om)\big\},\quad
H^1(\Om)=\ran(G_s)+H^1_\Delta(\Om), \, s\in\Sigma.
\enq
We can now apply Lemma \ref{lemL6} with $\cX=H^1(\Om)$, $A_s=G_s$ and $\cL=H^1_\Delta(\Om)$ (that is, $\cL=\ran(\pi_2)$). As a result, using formula \eqref{dfBs}, we obtain a family of operators $B_s$ which are finite-rank perturbations of $G_s$, and a neighborhood $\Sigma_0$ containing $0$ such that for all $s\in\Sigma_0$ the following assertions hold:
\begin{equation} \label{Assiiv} \begin{split}
{\rm (i)}\,  &\text{ $B_s$ is boundedly invertible, } \\
&\qquad\text{ and
$B_s-I_{H^1(\Om)} $,
$B_s^{-1}-I_{H^1(\Om)}\in\cB_\infty(H^1(\Om))$};\\
{\rm (ii)}\, &\text{ $\cK_s=B_s^{-1}(H^1_\Delta(\Om))$;}\\
{\rm (iii)}\, &\text{ the function $s\mapsto B_s$ is in $C^k(\Sigma;\cB(H^1(\Om)))$};\\
{\rm (iv)}\, &\text{ $B_s^{-1}\pi_1B_s$ is a projection in $H^1(\Om)$ onto $\cK_s$}\\
 &\text{ such that the function $s\mapsto B_s^{-1}\pi_1B_s$ is in $C^k(\Sigma;\cB(H^1(\Om)))$.}
\end{split}
\end{equation}

We will sometimes denote by $T_{s,\Delta}=T_s\big|_{H^1_\Delta(\Om)}$ the restriction of the rescaled trace map $T_s$ to the space $H^1_\Delta(\Om)$ of the harmonic functions. Due to Lemma \ref{estbelow} with $V=0$ we know that $T_{s,\Delta}(H^1_\Delta(\Om))$ is a closed subspace of $\cH=H^{1/2}(\dOm)\times H^{-1/2}(\dOm)$ and that the operator $T_{s,\Delta}\colon H^1_\Delta(\Om)\to T_{s,\Delta}(H^1_\Delta(\Om))$ is a bijection, and thus is an isomorphism. We will denote its inverse by $(T_{s,\Delta})^{-1}\colon T_{s,\Delta}(H^1_\Delta(\Om))\to H^1_\Delta(\Om)$. Also, we recall formula \eqref{wnop}, that is,
\beq\label{TMDNnew}
T_{s,\Delta}(H^1_\Delta(\Om))=\gr(-t(s)^{-1}N_{-\Delta} )
=\big\{(f,-t(s)^{-1}N_{-\Delta}  f): f\in H^{1/2}(\dOm)\big\},
\enq
where $N_{-\Delta} \in\cB(H^{1/2}(\dOm),H^{-1/2}(\dOm))$ is the (weak)  Neumann operator defined in Definition \ref{DNO}. Viewing elements of $\cH=H^{1/2}(\dOm)\times H^{-1/2}(\dOm)$ as $(2\times 1)$ column vectors, we introduce a family $\{Q_s\}$ of projection in $\cH$ with 
\beq\label{RKQ}\begin{split}
\ran(Q_s)&=\gr(-t(s)^{-1}N_{-\Delta} )=T_{s,\Delta}(H^1_\Delta(\Om)),\\
\ker(Q_s)&=\big\{(0,g)^\top: g\in H^{-1/2}(\dOm)\big\}\end{split}\enq
 by the formula
\beq\label{dfQs} Q_s=\begin{bmatrix}I_{H^{1/2}(\dOm)}&0\\-t(s)^{-1}N_{-\Delta} &0_{H^{-1/2}(\dOm)}\end{bmatrix},\, s\in\Sigma.
\enq
The function $s\mapsto Q_s$ is in $C^k(\Sigma;\cB(\cH))$ by our assumption on the function $t(\cdot)$.

We remark that the family of operators $(T_{s,\Delta})^{-1}Q_s\in\cB(\cH,H^1_\Delta(\Om))$ is well defined since  $T_{s,\Delta}(H^1_\Delta(\Om))=\ran(Q_s)$. 
We intend to show that this family is smooth; this is not obvious because the domain
of the operator $(T_{s,\Delta})^{-1}$ changes with $s$. To overcome this difficulty we will use the family of boundedly invertible  transformation operators $W_s\in\cB(\cH)$,   splitting the projections $Q_s$ and $Q_0$ so that $W_sQ_0=Q_sW_s$ as discussed in Remark \ref{DalKr}.
 With these $W_s$, we consider the family of operators $W_s^{-1}T_{s,\Delta}\colon H^1_\Delta(\Om)\to\cH$ and observe that the function $s\mapsto W_s^{-1}T_{s,\Delta}$ is in $C^k(\Sigma_0;\cB(H^1_\Delta(\Om),\cH))$.
Since the operators $T_{s,\Delta}\colon H^1_\Delta(\Om)\to T_{s,\Delta}(H^1_\Delta(\Om))=\ran(Q_s)$ and $W_s\colon \ran(Q_0)\to\ran(Q_s)$ are both isomorphisms between the respective spaces, we conclude that 
$W_s^{-1}T_{s,\Delta}\colon H^1_\Delta(\Om)\to\ran(Q_0)$ is an isomorphism for each $s\in\Sigma_0$ and, moreover, that the function
\beq\label{Ckst}\text{$s\mapsto\big(W_s^{-1}T_{s,\Delta}\big)^{-1}=(T_{s,\Delta})^{-1}W_s$ is in
$C^k(\Sigma_0;\cB(\ran(Q_0),H^1_\Delta(\Om)))$.}\enq
 Since $W_s$ splits $Q_0$ and $Q_s$, that is, $Q_s=W_sQ_0W_s^{-1}$, it follows that
\[(T_{s,\Delta})^{-1}Q_s=(T_{s,\Delta})^{-1}W_sQ_0W_s^{-1}=\big(W_s^{-1}T_{s,\Delta}\big)^{-1}Q_0W_s^{-1}.\]
Combined with \eqref{Ckst}, this shows that the function
\beq\label{CkstT}\text{$s\mapsto(T_{s,\Delta})^{-1}Q_s$ is in
$C^k(\Sigma_0;\cB(\cH,H^1(\Om)))$.}\enq

We introduce the families of operators $A_s\in\cB(\cH)$ and $D_s\in\cB(\ran(Q_s),\cH)$ by
\beq\label{dfAs}
A_s=T_sB_s^{-1}(T_{s,\Delta})^{-1}Q_s+(I_\cH-Q_s),\, D_s=T_sB_s^{-1}(T_{s,\Delta})^{-1},\, s\in\Sigma_0,
\enq
so that $A_s=D_sQ_s+(I_\cH-Q_s)$, and claim that conditions (a), (b) and ({c}) in Lemma \ref{lemAF} hold with $\cX=\cH$
and $\Sigma$ replaced by $\Sigma_0$. Starting the proof of the claim, we remark that the operators $D_s$ (and therefore 
$A_s$) are well defined because $(T_{s,\Delta})^{-1}\in\cB(\ran(Q_s),H^1_\Delta(\Om))$ and $B_s^{-1}(H^1_\Delta(\Om))=\cK_s\subset\dom(T_s)$ by assertion (ii) in \eqref{Assiiv}. Since $T_s(B_s^{-1}-I_{H^1(\Om)})(T_{s,\Delta})^{-1}Q_s\in\cB_\infty(\cH)$ by assertion (i) in \eqref{Assiiv} and  
\begin{align*}
A_s&=T_s\big(I_{H^1(\Om)}+(B_s^{-1}-I_{H^1(\Om)})\big)(T_{s,\Delta})^{-1}Q_s+(I_\cH-Q_s)\\&
=Q_s+T_s(B_s^{-1}-I_{H^1(\Om)})(T_{s,\Delta})^{-1}Q_s+(I_\cH-Q_s),\end{align*}
condition (a) in Lemma \ref{lemAF} holds.
Since $\cK_s=B_s^{-1}(H^1_\Delta(\Om))$ by assertion (ii) in \eqref{Assiiv}, and since $T_s$ restricted on $\cK_s$ is injective by Lemma \ref{estbelow}, $D_s$ in \eqref{dfAs} is a composition of three injective operators, and thus condition (b) in Lemma \ref{lemAF} holds. By our assumptions on the function $t(\cdot)$, the function $s\mapsto T_s$ is in $C^k(\Sigma;
\cB(H^1_\Delta(\Om),\cH))$. Applying assertion (iii) in \eqref{Assiiv} and \eqref{CkstT} yields 
condition ({c}) in Lemma \ref{lemAF}, thus finishing the proof of the claim.

Assertion (ii) in \eqref{Assiiv} yields
\beq\label{TsK}
T_s(\cK_s)=T_sB_s^{-1}(H^1_\Delta(\Om))=T_sB_s^{-1}(T_{s,\Delta})^{-1}T_{s,\Delta}(H^1_\Delta(\Om)), \, s\in\Sigma_0,
\enq
and, using \eqref{dfAs}
and $\ran(Q_s)=T_{s,\Delta}(H^1_\Delta(\Om))$, we conclude
\beq\label{TsAs}
T_s(\cK_s)=\ran(A_sQ_s), \, s\in\Sigma_0.
\enq 
We will now apply Lemma \ref{lemAF}. Using \eqref{TsAs} and the smooth family of the operators $\{F_s\}$
satisfying assertions (i), (ii) and (iii) of Lemma \ref{lemAF} we define the projection $P_s$ in $\cH$ onto 
$T_s(\cK_s)$ so that  
\beq\label{TsKend}
T_s(\cK_s)=\ran(F_sQ_s)=\ran(P_s),\quad P_s=F_sQ_sF_s^{-1},\, s\in\Sigma_0.
\enq
Clearly, the function $s\mapsto P_s$ is in $C^k(\Sigma_0,\cB(\cH))$. 
This proves that the family of subspaces $\{T_s(\cK_s)\}$ is $C^k$-smooth as required.

In addition, the formulas $T_s(\cK_s)=\ran(F_sQ_s)$ and $\ran(Q_s)=T_{s,\Delta}(H^1_\Delta(\Om))$ show that $T_s(\cK_s)=F_s(T_{s,\Delta}(H^1_\Delta(\Om)))$ where $F_s$ is a boundedly  invertible operator such that $F_s-I_\cH\in\cB_\infty(\cH)$. In other words, in terms of Definition \ref{dfnFLG}(i), we have proved that 
\beq\label{TsKsinFG}
T_s(\cK_s)\in\Fr(T_{s,\Delta}(H^1_\Delta(\Om))),\,s\in\Sigma_0.
\enq

It remains to show that $T_s(\cK_s)\in F\Lambda(\cG)$ provided $\cG\in\Lambda(\cH)$ is either Dirichlet- or Neumann-based. 
To begin the proof, we use the subspaces $\cH_D$ and $\cH_N$ 
of $\cH=H^{1/2}(\dOm)\times H^{-1/2}(\dOm)$ from Definition \ref{defDiNeB}.
Our first claim is the inclusions $T_s(H^1_\Delta(\Om))\in F(\cH_D)$ and 
 $T_s(H^1_\Delta(\Om))\in F(\cH_N)$. Recall formula \eqref{TMDNnew} where the Neumann operator $N_{-\Delta} $ is a Fredholm operator of index zero, see, e.g.,
 \cite[Section 7.11]{T96}.
 Since $T_s(H^1_\Delta(\Om))$ is the graph of the operator $-t(s)^{-1}N_{-\Delta}$, $T_s(H^1_\Delta(\Om))\cap\cH_D=\{0\}$ because $H^1_\Delta(\Om)\cap H^1_0(\Om)=\{0\}$.
Also, $T_s(H^1_\Delta(\Om))+\cH_D=\cH$ because 
\[(f,g)=\big(f,-t(s)^{-1}N_{-\Delta}  f)+(0,g+t(s)^{-1}N_{-\Delta}  f\big)\text{ for all } (f,g)\in\cH.\]
Thus, $T_s(H^1_\Delta(\Om))\dot{+}\cH_D=\cH$ and the pair $(T_s(H^1_\Delta(\Om)),\cH_D)$ is Fredholm. On the other hand, formula \eqref{TMDNnew} also shows
 $\dim\big(T_s(H^1_\Delta(\Om))\cap\cH_N\big)=\dim(\ker(N_{-\Delta} ))$ while the relation 
\[T_s(H^1_\Delta(\Om))+\cH_N=\big\{(f,-t(s)^{-1}N_{-\Delta}  f+g): f\in H^{1/2}(\dOm), g\in H^{-1/2}(\dOm)\big\}\]
shows that $\codim\big(T_s(H^1_\Delta(\Om))+\cH_N\big)=\codim(\ran(N_{-\Delta} ))$, thus proving that the pair $(T_s(H^1_\Delta(\Om)),\cH_N)$ is Fredholm since $N_{-\Delta} $ is a Fredholm operator.

Our second claim is that $\cH_D\in\Fr(\cG)$ provided $\cG$ is Dirichlet-based, and
$\cH_N\in\Fr(\cG)$  provided $\cG$ is Neumann-based. Indeed, if $\cG$ is Dirichlet-based then $\cG=\gr'(\Theta')$ for a compact operator $\Theta'\colon H^{-1/2}(\dOm)\to H^{1/2}(\dOm)$ while if $\cG$ is Neumann-based then $\cG=\gr(\Theta)$ for a compact operator $\Theta\colon H^{1/2}(\dOm)\to H^{-1/2}(\dOm)$. Viewing elements of $\cH$ as $(2\times 1)$ column vectors and denoting transposition by $\top$, we infer
\begin{align*}
\cG&=\big\{(\Theta'g,g)^\top: g\in H^{-1/2}(\dOm)\big\}=A_D\big(\big\{(0,g)^\top:g\in H^{-1/2}(\dOm)\big\}\big)=A_D(\cH_D),\\
\cG&=\big\{(f,\Theta f)^\top: f\in H^{1/2}(\dOm)\big\}=A_N\big(\big\{(f,0)^\top:f\in H^{1/2}(\dOm)\big\}\big)=A_N(\cH_N).
\end{align*}
Here, the invertible operators $A_D,A_N\in\cB(\cH)$ are defined
as 
\[A_D=\begin{bmatrix}I_{H^{1/2}}(\dOm)&\Theta'\\0&I_{H^{-1/2}}(\dOm)\end{bmatrix},
\, A_N=\begin{bmatrix}I_{H^{1/2}}(\dOm)&0\\ \Theta&I_{H^{-1/2}}(\dOm)\end{bmatrix}\]
and satisfy the relations $A_D^{\pm 1}-I_\cH \in\cB_\infty(\cH)$, $A_N^{\pm 1}-I_\cH \in\cB_\infty(\cH)$. This proves the inclusions $\cH_D\in\Fr(\cG)$ and $\cH_N\in\Fr(\cG)$.

The two claims just proved and Lemma \ref{lem3433}(ii) show that $T_s(H^1_\Delta(\Om))\in F(\cG)$ provided $\cG$ is either Dirichlet- or Neumann-based.
But now \eqref{TsKsinFG} and Lemma \ref{lem3433}(i) and (ii) show that $T_s(\cK_s)\in F(\cG)$. 

It remains to show that the subspace $T_s(\cK_s)$ is Lagrangian.
For this we will apply Lemma \ref{lem3433}(iii) with $\cL=T_s(\cK_s)$ and $\cK=T_{s,\Delta}(H^1_\Delta(\Om))$. The inclusion $\cL\in\Fr(\cK)$ holds by \eqref{TsKsinFG}, and the form $\omega$ vanishes on $\cL=T_s(\cK_s)$
 by Lemma \ref{lemSF}. It remains to show that $\cK=T_{s,\Delta}(H^1_\Delta(\Om))$ is Lagrangian, that is, the form $\omega$ vanishes on $\cK$, and it is the maximal subspace with this property. The former assertion holds by Lemma \ref{lemSF} and the latter holds because $T_{s,\Delta}(H^1_\Delta(\Om))$ is the graph of a selfadjoint operator $N_{-\Delta,s} $, the rescaled Neumann operator,
 \[N_{-\Delta,s} =-(t(s))^{-1}N_{-\Delta}  \colon H^{1/2}(\dOm)\to H^{-1/2}(\dOm),\] see formula \eqref{TMDNnew}. Indeed, if $(f,g)\in\cH$ satisfies $\omega((f,g), (h,N_{-\Delta,s} h)) = 0$ for all $(h,N_{-\Delta,s} h)\in\gr(N_{-\Delta,s} )$, then 
  \begin{align*}
 0&=\omega((f,g), (h,N_{-\Delta,s} h))=\langle N_{-\Delta,s} h, f\rangle_{1/2}-
 \langle g, h\rangle_{1/2}\\&=
 \langle N_{-\Delta,s} f, h\rangle_{1/2}-
 \langle g, h\rangle_{1/2}
 \end{align*}
 because $N_{-\Delta,s} $ is selfadjoint. It follows that $(f,g)\in\gr(N_{-\Delta,s} )=T_{s,\Delta}(H^1_\Delta(\Om))$, hence the subspace is maximal.
 \end{proof}

\subsection{The  path and the Dirichlet-to-Neumann and Neumann-to-Dirichlet operators}\label{sec:DNNDO}
In this subsection we discuss connections of the Lagrangian view on the eigenvalue problems and the Dirichlet-to-Neumann and Neumann-to-Dirichlet operators, $N_{L}$ and $M_{L}$, associated with the Schr\"odinger operator $L=-\Delta+V$ (see 
Definition \ref{dfnDNNDmaps}). We will use the notation in Proposition \ref{smoothinFLG}: $\Sigma=[a,b]$ denotes a set of parameters, $t(\cdot)\in C^k(\Sigma;[\tau ,1])$ a given function for some $\tau \in(0,1]$, and $\{V_s\}_{s\in\Sigma}$ a family of potentials such that the function $s\mapsto V_s$ is in $C^k(\Sigma; L^\infty(\Om))$. Furthermore, for $s\in\Sigma$ and $u\in\dom(\gaN )\subset H^1(\Omega)$ we define the rescaled map $T_s$ by $T_su=\big(\gaD u,(t(s))^{-1}\gaN u\big)$ and let 
\beq\label{dfnLsKs0}
L_s=-\Delta+V_s,\, \cK_s=\big\{u\in H^1(\Om): L_su=0 \text{ in } H^{-1}(\Om)\big\}\enq
so that  $T_s(\cK_s)$ is the subspace of traces of weak solutions to the equation $L_su=0$. We recall that  $\Sp(L_{s,\cH_D})$ is the spectrum of the operator $L_{s,\cH_D}$ (equipped with Dirichlet boundary conditions) and $\Sp(L_{s,\cH_N})$ is the spectrum of the operator $L_{s,\cH_N}$ (with Neumann boundary conditions).
In addition, we will use notation $N_s=-(t(s))^{-1}N_{L_s}$ and 
$M_s=-(t(s))M_{L_s}$ for the rescaled  Dirichlet-to-Neumann and Neumann-to-Dirichlet operators associated with $L_s$, using Definition \ref{dfnDNNDmaps} with the operator $L-\lambda$ replaced by $L_s$. We recall from Lemma \ref{propDNND} that $N_s\in\cB(H^{1/2}(\dOm),H^{-1/2}(\dOm))$ and $-M_s\in\cB(H^{-1/2}(\dOm),H^{1/2}(\dOm))$ are mutually inverse and selfadjoint operators with respect to the duality $H^{-1/2}(\dOm)=(H^{1/2}(\dOm))^*$. We also recall the Riesz duality map $J_{\rm R}\colon H^{-1/2}(\dOm)\to H^{1/2}(\dOm)$ discussed in Remark \ref{adjNM}. The following result gives a simple proof of Proposition \ref{smoothinFLG} in the case that $\lambda\notin\Sp(L_{\cH_D})\cap\Sp(L_{\cH_N})$.


\begin{proposition}\label{prop:TsKsDN} Suppose the assumptions in Proposition \ref{smoothinFLG} hold. If $0\notin\Sp(L_{s,\cH_D})$, then $T_s(\cK_s)$ 
is the graph of the rescaled Dirichlet-to-Neumann operator associated with $L_s$:
\begin{align}\label{TSdn}
	T_s(\cK_s)&=\gr(N_s).
\end{align}
Moreover,
\begin{align}\label{PsDNND}
	P_{s,DN}&=\begin{bmatrix}I_{H^{1/2}(\dOm)}&0\\
	N_s&0_{H^{-1/2}(\dOm)}\end{bmatrix}
\end{align}
defines a projection onto $T_s(\cK_s)$, and the orthogonal projection $\Pi_s=\begin{bmatrix} \Pi_s^{(ij)}\end{bmatrix}_{i,j=1}^2$ onto $T_s(\cK_s)$ in the direct sum $\cH=H^{1/2}(\dOm)\times H^{-1/2}(\dOm)$ is given by
\begin{align}
	\begin{split}
	& \Pi_s^{(11)}=\big(I_{H^{1/2}(\dOm)} +(J_{\rm R} N_s)^2\big)^{-1},\\
	&\Pi_s^{(12)}= J_{\rm R} N_sJ_{\rm R} \big(I_{H^{-1/2}(\dOm)}+ (N_sJ_{\rm R})^2\big)^{-1},\\ 
	&\Pi_s^{(21)}=N_s\big(I_{H^{1/2}(\dOm)} +(J_{\rm R} N_s)^2\big)^{-1},\\
	& \Pi_s^{(22)}= (N_sJ_{\rm R})^2 \big(I_{H^{-1/2}(\dOm)} + (N_sJ_{\rm R})^2\big)^{-1}.
	\end{split}
	\label{PsDNND1}
\end{align}

Similarly, if $0\notin\Sp(L_{s,\cH_N})$, then $T_s(\cK_s)$ is the inverse graph of the rescaled Neumann-to-Dirichlet rescaled operator associated with $L_s$:
\begin{align}
	T_s(\cK_s)&=\gr'(M_s).\label{TSnd}
\end{align}
Moreover, 
\begin{align}
	P_{s,ND}&=\begin{bmatrix}0_{H^{1/2}(\dOm)}&-M_s\\
	0&I_{H^{-1/2}(\dOm)}\end{bmatrix}\label{PsNDND}
\end{align}
defines a projection onto $T_s(\cK_s)$, and the orthogonal projection is given by
\begin{align}\begin{split}
&\Pi_s^{(11)} =  (M_s J_{\rm R}^{-1})^2 \big( I_{H^{1/2}(\partial\Omega)} + (M_s J_{\rm R}^{-1})^2 \big)^{-1},\\ 
& \Pi_s^{(12)}=  -M_s \big( I_{H^{-1/2}(\partial\Omega)} + (J_{\rm R}^{-1} M_s)^2 \big)^{-1},\\ 
& \Pi_s^{(21)}= -J_{\rm R}^{-1} M_s J_{\rm R}^{-1} \big( I_{H^{1/2}(\partial\Omega)} + (M_s J_{\rm R}^{-1})^2 \big)^{-1}, \\
&\Pi_s^{(22)} = \big( I_{H^{-1/2}(\partial\Omega)} + (J_{\rm R}^{-1} M_s)^2 \big)^{-1}.\end{split}\label{PsNDND1}
\end{align}

As a result, the function $s\mapsto \Pi_s$ is in $C^k(\Sigma_0; \cB(\cH))$ as long as $0$ is not in the intersection of the Dirichlet and Neumann spectra of the operator $L_s$ for any $s\in\Sigma_0$.
\end{proposition}

Observe that formulas \eqref{PsDNND1} and \eqref{PsNDND1}  for $\Pi_s$ must be equivalent when $0\notin\Sp(L_{s,\cH_D})\cup\Sp(L_{s,\cH_N})$, because the orthogonal projection is unique. This can be seen explicitly using the relation $M_s = -N_s^{-1}$.
\begin{proof} 
If $u\in\cK_s$, then $L_su=0$ by \eqref{dfnLsKs0} and $\gaN u=-N_{L_s}\gaD u$ by \eqref{dfnwDNo} with $L-\lambda$ replaced by $L_s$. Since $T_su=(\gaD u, (t(s)^{-1}\gaN u)$, \eqref{TSdn} holds. The proof of \eqref{TSnd} is similar.
Formula \eqref{PsDNND} holds since $T_s(\cK_s)$ is the graph of $N_s$.
 To show \eqref{PsDNND1}, we compute the adjoint $P_s^* =\begin{bmatrix} I_{H^{1/2}}(\dOm) & N_s^* \\ 0 & 0 \end{bmatrix}$ using the scalar products in $H^{1/2}(\dOm)$ and $H^{-1/2}(\dOm)$ (as opposed to the duality of the spaces  $H^{1/2}(\dOm)$ and $H^{-1/2}(\dOm)$).
Using \eqref{PPi}, the corresponding orthogonal projection is given by
\begin{align*}
	\Pi_s &= P_s P_s^*(P_sP_s^* + (I_\cH-P_s^*)(I_\cH-P_s))^{-1} \\
	&= \begin{bmatrix} I_{H^{1/2}(\dOm)} & N_s^* \\ N_s & N_sN_s^* \end{bmatrix} \cdot
	\begin{bmatrix}I_{H^{1/2}(\dOm)} +N_s^*N_s & 0 \\ 0 & I_{H^{-1/2}(\dOm)} + N_sN_s^* \end{bmatrix}^{-1} \\
	&= \begin{bmatrix} (I_{H^{1/2}(\dOm)} +N_s^*N_s)^{-1} & N_s^*(I_{H^{-1/2}(\dOm)} + N_sN_s^*)^{-1} \\ N(I_{H^{1/2}(\dOm)} +N_s^*N_s)^{-1} & N_sN_s^* (I_{H^{-1/2}(\dOm)}+ N_sN_s^*)^{-1} \end{bmatrix}.
\end{align*}
To simplify further we compute $N_s^*$, recalling that $N_s$ is selfadjoint with respect to the dual pairing by Lemma \ref{propDNND}$(i_2)$. Equivalently, $(J_{\rm R} N_s)^* = J_{\rm R} N_s$, where $J_{\rm R}\colon H^{-1/2}(\partial \Omega) \rightarrow H^{1/2}(\partial \Omega)$ is the Riesz duality map. Since $J_{\rm R}^* = J_{\rm R}^{-1}$, we have $N_s^* = J_{\rm R} N_s J_{\rm R}$. Therefore $N_s N_s^* = (N_sJ_{\rm R})^2$ and $N_s^* N_s = (J_{\rm R} N_s)^2$, yielding \eqref{PsDNND1}. 

The proof of \eqref{PsNDND1} is similar, applying \eqref{PPi} to the projection \eqref{PsNDND}.


The last assertion follows from the fact that the functions $s\to N_s,M_s$ are in $C^k(\Sigma; \cB(\cH))$. This  in turn follows from the formulas for the Dirichlet-to-Neumann and Neumann-to-Dirichlet operators in terms of resolvents and Neumann traces given in Lemma \ref{propDNND} ($i_4$), ($ii_4$).
\end{proof}

We will now prove that the projections onto $T_s(\cK_s)$ defined in \eqref{TsKend} during the proof of Proposition \ref{smoothinFLG} are the same as the projections $P_{s,DN}$ onto $T_s(\cK_s)$ defined in \eqref{PsDNND} via the Dirichlet-to-Neumann operator $N_s$ associated with $L_s=-\Delta+V_s$, provided $0\notin\Sp(L_{s,\cH_D})$ (and thus the Dirichlet-to-Neumann operator is
well defined). Although the family of projections  is not well defined, and hence discontinuous, at the points $s$ where $0\in\Sp(L_{s,\cH_D})$, the range of $P_s$ at these points preserves continuity, thus keeping the family of the orthogonal projections $\Pi_s$ continuous. In this sense, the family of projections $\{P_{s,DN}\}$ has a ``removable singularity" at these points; this is obtained by adding a finite-rank perturbation to the underlying operator $A_s$ from the proof of Proposition \ref{smoothinFLG}.
\begin{theorem}\label{DNPs}
Assume Hypothesis \ref{h1} and $0\notin\Sp(L_{s,\cH_D})$. Let $P_{s,DN}$ be the projection in $\cH$ onto $T_s(\cK_s)$ defined in \eqref{PsDNND} and let $P_s$ be the projection onto $T_s(\cK_s)$ defined in \eqref{TsKend}. Then $P_s=P_{s,DN}$.
\end{theorem}
\begin{proof}
Let us review the arguments in the proof of Proposition \ref{smoothinFLG} leading to formula \eqref{TsKend}. With no loss of generality we will assume that $s=0$. Due to Remark \ref{rem:LDLN} we know that $0\notin\Sp(L_{0,\cH_D})$ if and only if the operator $G_0$ is invertible. If this is the case then formula \eqref{dfBs} shows that $B_0=G_0$ in assertions \eqref{Assiiv} and in formula \eqref{dfAs}. 

We first claim that the operator $A_0$ defined in \eqref{dfAs} is invertible. To prove this, let us use the projection $Q_0$ from \eqref{dfQs} and rewrite the operator $A_0$ in the direct sum decomposition $\cH=\ran(Q_0)\dot{+}\ran(I_\cH-Q_0)$ as the following operator matrix:
\beq\label{A0rep}
A_0=\begin{bmatrix}Q_0T_0G_0^{-1}(T_{0,\Delta})^{-1}Q_0&0\\
(I_\cH-Q_0)T_0G_0^{-1}(T_{0,\Delta})^{-1}Q_0&(I_\cH-Q_0)\end{bmatrix}\,.
\enq
Suppose that $h=(h_1,h_2)^\top\in\ker(A_0)$, where $h_1\in\ran(Q_0)$ and $h_2\in\ran(I_\cH-Q_0)$. Then $T_0G_0^{-1}(T_{0,\Delta})^{-1}Q_0h_1\in\ker(Q_0)=\cH_D$. In other words,
the function $u_1=G_0^{-1}(T_{0,\Delta})^{-1}Q_0h_1$ satisfies the Dirichlet boundary condition. In addition, $u_1\in\cK_0$ by assertion (iv) in \eqref{Assiiv} and also assertion (ii) in Lemma \ref{lemL6} with $B_0=G_0$ and therefore satisfies $L_0u_1=0$. This means that $u_1$ is in fact an eigenfunction of $L_{\cH_D}$; since $0\notin\Sp(L_{0,\cH_D})$ by assumption we conclude that $u_1=0$. Applying $G_0$ and $T_0$ yields $Q_0h_1=0$ or $h_1\in\ker(Q_0)$ proving $h_1=0$ since $h_1\in\ran(Q_0)$. Now $h\in\ker(A_0)$ yields $h_2=0$ by \eqref{A0rep}. Since $\ker(A_0)=\{0\}$ and $A_0-I_\cH\in\cB_\infty(\cH)$ as $A_0$ satisfies the conditions in Lemma \ref{lemAF}, we conclude that $A_0$ is indeed invertible.

Since $\ker(A_0)=\{0\}$, we have $R_0=0$ in \eqref{dfFs} and thus $F_0=A_0$ in \eqref{TsKend}, that is, $T_0(\cK_0)=\ran(P_0)=\ran(A_0Q_0)$ where $P_0=A_0Q_0A_0^{-1}$. We claim that $\ker(P_0)\subset\ker(Q_0)$. Assuming the claim, we finish the proof of the equality $P_0=P_{0,DN}$ as follows. Formulas \eqref{PsDNND} and \eqref{RKQ} show that $\ker(P_{0,DN})=\cH_D=\ker(Q_0)$. Since $\ran(P_0)=T_0(\cK_0)=\ran(P_{0,DN})$ and $\ran(P_0)\dot{+}\ker(P_0)=\cH=\ran(P_{0,DN})\dot{+}\ker(P_{0,DN})$, the inclusion $\ker(P_0)\subset\ker(Q_0)=\ker(P_{0,DN})$ shows that in fact $\ker(P_0)=\ker(P_{0,DN})$, which  proves the required assertion $P_0=P_{0,DN}$.

To prove the claim $\ker(P_0)\subset\ker(Q_0)$, let us take any $h=(h_1,h_2)^\top\in\cH=\ran(Q_0)\dot{+}\ker(Q_0)$ such that
$P_0h=A_0Q_0A_0^{-1}h=0$. Then $Q_0A_0^{-1}h=0$ since $A_0$ is invertible. Since the block matrix for $A_0$ in \eqref{A0rep} is triangular, the block matrix for $A_0^{-1}$ is also triangular and its upper left corner is an invertible operator on $\ran(Q_0)$. This operator is the same
as the upper left corner for $Q_0A_0^{-1}$ and thus $Q_0A_0^{-1}h=0$ yields $h_1=0$
or $h\in\ker(Q_0)$ as claimed.
\end{proof}

\section{Monotonicity and other properties of the Maslov index}\label{sec5}

In this section we prove that the operators $L_{s,\cG}(\tau)$ have no large negative eigenvalues. This is equivalent to showing there are no crossings on the left vertical part of the path $\Gamma$ in Figure \ref{fig1} for $\Lambda=\Lambda(\tau)>0$ large enough. We also show that the Maslov crossing form is sign definite on the horizontal parts of $\Gamma$. This allows one to compute the Morse index of $L_\cG$ via the Maslov index along the top horizontal part of $\Gamma$. Finally, we derive formulas for the Maslov crossing form on the right vertical part of $\Gamma$.

\subsection{No crossings on $\Gamma_4$} 
For fixed $\tau \in(0,1]$ and $\Lambda>0$ let us consider the square $[-\Lambda,0]\times[\tau ,1]$, with boundary $\Gamma=\cup_{j=1}^4\Gamma_j$ parametrized by the functions $\lambda(s)$, $t(s)$ defined in \eqref{dfnSigmaj} and \eqref{dfnlambdat}, see Figure \ref{fig1}. Let $\cG$ be a Dirichlet- or Neumann-based Lagrangian subspace of $\cH=H^{1/2}(\dOm)\times H^{-1/2}(\dOm)$, and let $L_s=L_s(\tau )$ be the family of operators in \eqref{dfnLst0}. We define $\cK_s$ to be the subspace in $H^1(\Om)$ of weak solutions to the equation $L_su=0$, as in \eqref{dfnKS}, and consider the path $\Upsilon\colon\Sigma\to F\Lambda(\cG)$ defined by $\Upsilon(s) = T_s(\cK_s)$ in \eqref{dfnups}. With a slight abuse of terminology, we say that $s_0$ is a crossing on $\Gamma_j$ if $s_0\in\Sigma_j$.

Our first objective is to show that there are no crossings on $\Gamma_4$ provided $\Lambda=\Lambda(\tau)$ is large enough. This holds for both the Dirichlet- and Neumann-based cases.

\begin{lemma}\lb{lem:Gamma4} Assume Hypothesis \ref{h1} (i) -- (iv), where $\cG$ is either Dirichlet- or Neumann-based. For each $\tau \in(0,1]$ there exists a positive $\Lambda=\Lambda(\tau )$ such that the path $\Upsilon$ defined in \eqref{dfnups} has no crossings on $\Gamma_4$. \end{lemma}
\begin{proof}
By \eqref{eq:4.23} in Lemma \ref{GevL} it is enough to show that there is a positive $c$ such that $\langle L_{s,\cG}(\tau)u,u\rangle_{L^2(\Om)}\ge c\|u\|_{L^2(\Om)}^2$ for all $s\in\Sigma_4$ and $u\in\dom(L_{s,\cG}(\tau))$. As in Remark \ref{sbb}, by \eqref{dfnlambdat}, if $s\in\Sigma_4$ then $s'=-s+2(1-\tau)+\Lambda\in\Sigma_2$ is such that $t(s)=t(s')$ and thus $\dom(L_{s,\cG}(\tau))=\dom(L_{s',\cG}(\tau))$ and $L_{s,\cG}(\tau)=L_{s',\cG}(\tau)+\Lambda (t(s))^2$. By Hypothesis \ref{h1} (iii) the operators $L_{s,\cG}(\tau)$ are semibounded from below uniformly for $s\in\Sigma$; in particular, there is a $c_\tau\in\bbR$ (not necessarily positive) such that  $\langle L_{s',\cG}(\tau)u,u\rangle_{L^2(\Om)}\ge c_\tau\|u\|_{L^2(\Om)}^2$ for all $s'\in\Sigma_2$ and all $u\in\dom(L_{s,\cG}(\tau))$. Taking $c=c_\tau+\Lambda\tau^2>0$ for $\Lambda$ large enough finishes the proof.
\end{proof}

\subsection{Crossings on $\Gamma_1$ and $\Gamma_3$}
In this subsection we show the monotonicity of the Maslov index (that is, sign definiteness of the Maslov crossing form) for the path $s\mapsto T_s(\cK_s)$ when $s\in\Sigma_1\cup\Sigma_3$. For these values of $s$  the parameter $\lambda=\lambda(s)$ is changing while $t=t(s)$ remains fixed, see \eqref{dfnlambdat} and Figure \ref{fig1}. We also show that if $\cG = \gr'(G')$ is Dirichlet-based and $\Theta'$ is nonpositive, then there are no crossings on $\Gamma_1$ when $\tau > 0$ is sufficiently small.

We begin with a general formula for the Maslov crossing form defined in \eqref{dfnMasForm}. Let $\cG$ be any Dirichlet- or Neumann-based Lagrangian subspace of $\cH=H^{1/2}(\dOm)\times H^{-1/2}(\dOm)$, see Definition \ref{defDiNeB}. Let $\Sigma=[a,b]$ be a set of parameters, $\{V_s\}_{s\in\Sigma}$ be a family of potentials and $t\colon\Sigma\to[\tau,1]$ for some $\tau>0$ be a given function. We denote $d/ds$ by dot.
\begin{lemma}\label{prelmon} Assume that $s\mapsto V_s$ is in $C^1(\Sigma; L^\infty(\Om))$ and $t(\cdot)\in C^1(\Sigma;[\tau,1])$. Let $L_s=-\Delta+V_s$ and $T_su=(\gaD u,(t(s))^{-1}\gaN u)$, let $\cK_s$ denote the set of weak solutions to $L_su=0$ and let $\Upsilon(s)=T_s(\cK_s)$ for $s\in\Sigma$. If $s_0\in\Sigma$ is a crossing and $q\in T_{s_0}(\cK_{s_0})\cap\cG$, then there exists a unique function $u_{s_0} \in \cK_{s_0}$ such that $q=T_{s_0} u_{s_0}$, and the Maslov crossing form satisfies
\beq\label{mqq}
\mathfrak{m}_{s_0}(q,q)=\frac1{t(s_0)}\langle \dot{V}_s\big|_{s=s_0}
 u_{s_0}, u_{s_0}\rangle_{L^2(\Om)}-\frac{\dot{t}(s_0)}{(t(s_0))^{2}}\,\langle
\gaN u_{s_0},\gaD u_{s_0}\rangle_{1/2}.
\enq
\end{lemma}
\begin{proof} 
We first prove that there exists a segment $\Sigma_0\subset\Sigma$ containing $s_0$ and a function $s\mapsto u_s$ in $C^1(\Sigma_0; H^1(\Om))$ such that $q=T_{s_0} u_{s_0}$ and $u_s\in\cK_s$ for each $s\in\Sigma_0$. For simplicity of notation we assume that $s_0=0$ and moreover that $0$ is in the interior of $\Sigma$. For $s$ in a small neighborhood $\Sigma_0$ in $\Sigma$
containing $0$ let $\Pi_s$ be the orthogonal projection in $\cH$ onto $T_s(\cK_s)$. 
 By Proposition \ref{smoothinFLG} the function $s\mapsto\Pi_s$ is in $C^1(\Sigma_0;\cB(\cH))$. By Lemma \ref{crossex}
there exist operators $\{R_s\}$ from $\ran(\Pi_0)$ into $\ker(\Pi_0)$ such that the  function $s\mapsto R_s$ is in $C^1(\Sigma_0;\cB(\ran(\Pi_0),\ker(\Pi_0)))$ and 
\beq\label{RsTsKs}
\gr(R_s)=\ran(\Pi_s)=T_s(\cK_s)\,\text{ for all }\, s\in\Sigma_0
\enq
provided $\Sigma_0$ is small enough. 
Let us fix $q\in T_0(\cK_0)=\ran(\Pi_0)$ and consider the curve 
$s\mapsto q+R_sq\in T_s(\cK_s)=\ran(\Pi_s)$. Since $s\mapsto R_s$ is smooth, this curve is in $C^1(\Sigma_0;\cH)$. Since the restriction $T_s|_{\cK_s}$ of the rescaled trace map is injective by Lemma \ref{estbelow}, for each $s\in\Sigma_0$ there exists a unique $u_s\in\cK_s$ such that $q+R_sq=T_su_s$.  Then $q=T_0u_0$ since $R_0=0$ by \eqref{smallRs}.

We now claim that the function
\beq\lb{seccl}
\text{$s\mapsto u_s$ is in $C^1(\Sigma_0; H^1(\Om))$}.
\enq
Starting the proof of \eqref{seccl} we denote by $(T_s|_{\cK_s})^{-1}\colon \ran(\Pi_s)\to\cK_s$ the isomorphism between $\ran(\Pi_s)=\ran(T_s|_{\cK_s})$ and $\cK_s$ which exists by Lemma \ref{estbelow}. Since $(T_s|_{\cK_s})u_s=(I_\cH+R_s)q$ by the definition of $u_s$, we infer that $u_s=(T_s|_{\cK_s})^{-1}(I_\cH+R_s)q$, where the operator $I_\cH+R_s$ is invertible in $\cH$ by \eqref{smallRs} and thus the function $s\mapsto(I_\cH+R_s)^{-1}$ is in $C^1(\Sigma_0;\cB(\cH))$ since $s\mapsto R_s$ is smooth. Thus \eqref{seccl} follows as soon as we know that $s\mapsto(T_s|_{\cK_s})^{-1}(I_\cH+R_s)$ is smooth. In turn, the latter fact follows as soon as we know that $s\mapsto(I_\cH+R_s)^{-1}T_s|_{\cK_s}$ is smooth. Since $t(\cdot)$ is smooth and 
\[T_su=(\gaD u,(t(s))^{-1}\gaN u)=\begin{bmatrix}I_{H^{1/2}(\dOm)}&0\\0&t(0)(t(s))^{-1}I_{H^{-1/2}(\dOm)}\end{bmatrix}\,T_0u,\] it is enough to check that $s\mapsto T_0|_{\cK_s}$ is in $C^1(\Sigma_0;\cB(H^1(\Om),\cH))$. This however is easy to show as we have constructed in assertion (iv) of \eqref{Assiiv} a family $s\mapsto B_s^{-1}\pi_1B_s$ of smooth projections in $H^1(\Om)$ onto $\cK_s$. Let $W_{\cK_s}$ denote the family of transformation operators associated with these projections (as in Remark \ref{DalKr}) so that $W_{\cK_s}\colon\cK_0\to\cK_s$ is an isomorphism and the function $s\mapsto W_{\cK_s}$ is in $C^1(\Sigma_0;\cB(H^1(\Om)))$. Then the function $s\mapsto T_0|_{\cK_s}=T_0W_{\cK_s}|_{\cK_0}$ is in $C^1(\Sigma_0;\cB(H^1(\Om),\cH)))$, finishing the proof of \eqref{seccl}.

Differentiating the equation $L_su_s=0$ with respect to $s$  we infer 
\beq\label{dots}
-\Delta \dot u_s+V_s\dot u_s+\dot V_su_s=0\,\text{ in the weak sense.}
\enq In particular, $\Delta \dot u_s\in L^2(\Omega)$, which means that $\dot u_s\in\dom(\gaN)$. This allows us to explicitly compute the Maslov crossing form $\mathfrak{m}_{0,\cG}(q,q)$ defined for
$q\in\Upsilon(0)\cap\cG$ according to Definition \ref{dfnMasF}. Indeed, using the equations $q+R_sq=T_su_s$ and $\dot{T}_s\big|_{s=0}u=-(0,\dot{t}(0)(t(0))^{-2}\gaN u)$ and the definition of $\omega$ \eqref{dfnomega} yields
\begin{align}\nonumber
{\Mf}_{0,\cG}(q,q)&=\frac{d}{ds}\omega(q,R_sq)\big|_{s=0}=\omega(T_0u_0,\dot{T}_s\big|_{s=0}u_0)+\omega(T_0u_0,T_0\dot{u}_s\big|_{s=0})\\&=
-\frac{\dot{t}(0)}{(t(0))^2}\langle\gaN u_0,\gaD   u_0\rangle_{1/2}\lb{qf}\\
&\qquad+\frac1{t(0)}\Big(\langle\gaN \dot u_s,\gaD   u_s\rangle_{1/2}-\langle\gaN u_s,\gaD   \dot u_s\rangle_{1/2}\Big)\big|_{s=0}.\nonumber
\end{align}
On the other hand,  using Green's formula \eqref{wGreen} and \eqref{dots}, we compute:
\begin{align}\nonumber
\langle&\gaN \dot u_s,\gaD   u_s\rangle_{1/2}= \langle{\nabla \dot u_s}, \nabla
u_s\rangle_{L^2(\Om)} + \langle \Delta \dot u_s,u_s\rangle_{L^2(\Om)}\\
&=\langle{\nabla \dot u_s}, \nabla
u_s\rangle_{L^2(\Om)} + \langle V_s\dot u_s+\dot V_su_s,u_s\rangle_{L^2(\Om)}\nonumber\\&=\langle\gaN u_s,\gaD  \dot u_s\rangle_{1/2}+\langle -\Delta u_s,\dot u_s\rangle_{L^2(\Om)}+\langle V_s\dot u_s+\dot V_su_s,u_s\rangle_{L^2(\Om)}
\nonumber\\&=\langle\gaN u_s,\gaD  \dot u_s\rangle_{1/2}+\langle -\Delta u_s+V_su_s,\dot u_s\rangle_{L^2(\Om)}+\langle\dot V_su_s,u_s\rangle_{L^2(\Om)}
\nonumber\\&=\langle\gaN u_s,\gaD  \dot u_s\rangle_{1/2}+\langle\dot V_su_s,u_s\rangle_{L^2(\Om)}.
\lb{ibpN}
\end{align}
Combining \eqref{qf} and \eqref{ibpN} completes the proof.
\end{proof}
We will now apply Lemma \ref{prelmon} for the parametrized family of the potentials $V_s(x)=(t(s))^2\big(V(t(s)x)-\lambda(s)\big)$, operators $L_s=-\Delta+V_s$ and rescaled trace maps $T_su=(\gaD u,(t(s))^{-1}\gaN u)$ defined  in \eqref{defVsLs} and \eqref{defTs}, with $\tau \in(0,1]$ fixed and the functions $t(s)$ and $\lambda(s)$ defined in \eqref{dfnSigmaj} and \eqref{dfnlambdat} for $s\in\Sigma=\Sigma(\tau)$.
\begin{lemma}\label{lambdaMon}
 Assume Hypothesis \ref{h1}.  Let  $\cK_s$ denote the set of weak solutions to $L_su=0$ and let $\Upsilon(s)=T_s(\cK_s)$. Then any crossing $s_0\in\Sigma_1$  is negative while any crossing $s_0\in\Sigma_3$ is positive. In particular, for the Morse indices of the operators $L_\cG=L_{0,\cG}(1)$ and $L_{0,\cG}(\tau)$ defined in \eqref{dfLG}, \eqref{dfnLst00} the following formulas hold:
\beq\label{MM13}
\mor(L_\cG)=\mas(\Upsilon\big|_{\Sigma_3},\cG),\,
\mor(L_{0,\cG}(\tau))=-\mas(\Upsilon\big|_{\Sigma_1},\cG).
\enq
\end{lemma}
\begin{proof} We will consider $s_0\in\Sigma_1$; the proof for $s_0\in\Sigma_3$ follows since $\Gamma_1$ and $\Gamma_3$ have opposite orientations.
Reparametrizing $\lambda(\cdot)\mapsto\lambda(\cdot)+a$ for some $a\in\bbR$, without loss of generality we will assume that $s_0=0$ and moreover that $0$ is in the interior of $\Sigma_1$. Since $t(s)=\tau$ for $s\in\Sigma_1$ by \eqref{dfnlambdat}, we have that the rescaled trace map
\beq\lb{constTs}
T_su=(\gaD u, \tau^{-1}\gaN u)=T_0u\,\text{ for all }\,s\in\Sigma_1\enq
is $s$-independent.
 Since $\lambda(s)=s$ for $s\in\Sigma_1$ by \eqref{dfnlambdat}, formula \eqref{defVsLs} shows that $s\mapsto V_s$ is in $C^1(\Sigma_1,\cB(L^\infty(\Om)))$, and thus we may apply Lemma \ref{prelmon}. If $s\in\Sigma_1$ then $L_s=-\Delta+V_s$ where $V_s=\tau^2V(\tau x)-\tau^2s$ by \eqref{defVsLs} and \eqref{dfnlambdat}.
 Using $\dot{t}(s)=0$ and $\dot{V_s}=-\tau^2$ we obtain from \eqref{mqq} the relation 
\begin{align}\lb{qfN}
\begin{split}
{\Mf}_{0,\cG}(q,q)=\tau^{-1}\langle\dot V_su_s,u_s\rangle_{L^2(\Om)}\big|_{s=0}=-\tau\|u_0\|^2_{L^2(\Om)}<0,
\end{split}
\end{align}
thus proving that all crossings on $\Gamma_1$ are negative. This gives the second equality in \eqref{MM13}. (Note that this equality holds whether or not a crossing occurs at the right endpoint of $\Gamma_1$ since the crossing form is negative definite, cf. \eqref{MIcomp}.)
\end{proof}

In the Dirichlet-based case we have that there are no crossings on $\Gamma_1$, provided $\tau>0$ is small enough and the boundary operator satisfies a sign condition.

\begin{lemma}\lb{lem:Gamma4Dir} Assume Hypothesis \ref{h1} (i), (ii'), (iii) and (iv), where $\cG=\gr'(\Theta')$ is Dirichlet-based and $\Theta'$ is nonpositive. Then there exists a small enough $\tau>0$ such that the path $\Upsilon$ defined in \eqref{dfnups} has no crossings on $\Gamma_1$ for any $\Lambda>0$.
\end{lemma}
\begin{proof}  Let $-\Delta_{s,\cG}(\tau)$ denote the Laplacian equipped with the boundary conditions from $\cG$, i.e.
\beq\label{DGD}
\begin{split}
-\Delta_{s,\cG}(\tau)u(x)&=-\Delta u(x),\, x\in\Om,\\
 \dom(-\Delta_{s,\cG}(\tau))&=\big\{u\in H^1(\Omega)\,|\,\Delta u\in L^2(\Omega) \,\text{ and }\\
&\hskip3cm (t(s)\gaD-\Theta'\gaN)u=0\,\,\hbox{in}\,\,H^{1/2}(\dOm) \big\}.
\end{split}
\enq
We first claim that the operator $-\Delta_{s,\cG}(\tau)$ is strictly positive on $L^2(\Om)$, that is, for some $c>0$ we have the inequality 
\beq\label{posLap}
\langle -\Delta_{s,\cG}(\tau)u,u\rangle_{L^2(\Om)}\ge c\|u\|_{L^2(\Om)}^2\,\text{ for all $u\in\dom\big(-\Delta_{s,\cG}(\tau)\big)$}.\enq
Indeed, Green's formula \eqref{wGreen} for $u\in\dom\big(-\Delta_{s,\cG}(\tau)\big)$ yields
\beq\label{eq4.1}
\begin{split}
\langle -\Delta_{s,\cG}(\tau)u,u\rangle_{L^2(\Om)}&=
\langle\nabla u,\nabla u\rangle_{L^2(\Om)}-\langle \gaN u,\gaD u\rangle_{1/2}\\
&=\|\nabla u\|_{L^2(\Om)}^2-(t(s))^{-1}\langle \gaN u,\Theta'\gaN u\rangle_{1/2}\ge0
\end{split}
\enq
since $\Theta'\le0$. It remains to show that $0$ is not an eigenvalue of $-\Delta_{s,\cG}(\tau)$. If $u\in\dom\big(-\Delta_{s,\cG}(\tau)\big)$ and $-\Delta_{s,\cG}(\tau)u=0$, then \eqref{eq4.1}
yields $\|\nabla u\|_{L^2(\Om)}=0$ and thus $\gaN u=0$. It follows that $\gaD u=(t(s))^{-1}\Theta'\gaN u=0$, and so $u=0$ by Lemma \ref{UCP}. This finishes the proof of claim \eqref{posLap}.

We now prove the assertion in the lemma.
By \eqref{4.22} and \eqref{eq:4.23} in Lemma \ref{GevL} it suffices to prove that $L_{0,\cG}(\tau)$ is strictly positive when $\tau>0$ is small enough.
Since $t(s)=\tau$ for $s\in\Sigma_1$ and $\lambda(0)=0$ by \eqref{dfnlambdat}, inequality
\eqref{posLap} yields
\begin{align*}
\langle L_{0,\cG}(\tau)u,u\rangle_{L^2(\Om)}&=
\langle -\Delta_{0,\cG}(\tau)u,u\rangle_{L^2(\Om)}+\tau^2\langle V(\tau x)u,u\rangle_{L^2(\Om)}\\
&\ge \left(c-\tau^2\|V\|_{L^\infty} \right)\|u\|_{L^2(\Om)}^2\ge\frac12c\|u\|_{L^2(\Om)}^2
\end{align*}
for all $u\in\dom\big(-\Delta_{0,\cG}(\tau)\big)=\dom\big(L_{0,\cG}(\tau)\big)$ provided $\tau$ is small enough.
\end{proof}
For the Neumann-based case we postpone the discussion of the crossings on $\Gamma_1$ until Section \ref{sec:mainres}.


\subsection{Crossings on $\Gamma_2$}\lb{ss:g2}
In this subsection we study the crossings of the path $s\mapsto\Upsilon(s)$ for $s\in\Sigma_2$. As usual, we  let $L_s=-\Delta+V_s$ and $T_su=(\gaD u,(t(s))^{-1}\gaN u)$, let $\cK_s$ denote the set of weak solutions to $L_su=0$ and let $\Upsilon(s)=T_s(\cK_s)$. Since 
 $s\in\Sigma_2$, in this subsection we have $\lambda(s)=0$, $t(s)=\tau+s$ and $V_s(x)=(t(s))^2V(t(s)x)$ by \eqref{dfnlambdat}. Throughout this subsection we assume Hypothesis \ref{h1} (i') and (iii') so that $\dom(L_{s,\cG}(\tau))\subseteq H^2(\Om)$ and Hypothesis \ref{h1} (iv') so that the potential $V$ is smooth. We begin with a general formula for the Maslov crossing form that holds for both the Dirichlet- and Neumann-based cases.
\begin{lemma}\lb{preltmon} Let $\cG$ be a Dirichlet- or Neumann-based subspace in the boundary space $\cH=H^{1/2}(\dOm)\times H^{-1/2}(\dOm)$. Assume Hypothesis \ref{h1} (i'), (ii), (iii') and  (iv'), and let $s_0\in\Sigma_2$ be a crossing. 
If $q\in T_{s_0}(\cK_{s_0})\cap\cG$ and $u_{s_0}\in\dom\big(L_{s_0,\cG}(\tau)\big)\subseteq H^2(\Om)$ are such that $L_{s_0}u_{s_0}=0$ and $q=T_{s_0}u_{s_0}$ then the Maslov crossing form on $\Upsilon\big|_{\Sigma_2}$ is given as follows:
\beq\lb{henF1}\begin{split}
&\mathfrak{m}_{s_0}(q,q)=\frac1{(t(s_0))^2}\int_{\dOm}\Big(\big(\nabla u_{s_0}\cdot \nabla u_{s_0})(\nu\cdot x)-2\langle\nabla u_{s_0}\cdot x,\gaN^{\rm s}u_{s_0}(x)\rangle_{\bbR^N}\\&+(1-d)\langle\gaN^{\rm s} u_{s_0}(x), u_{s_0}(x)\rangle_{\bbR^N}+\langle V_{s_0}(t(s_0)x)u_{s_0}(x),u_{s_0}(x)\rangle_{\bbR^N}(\nu\cdot x)\Big)\,dx.
\end{split}\enq
\end{lemma} 
An eigenfunction $u_{s_0}$ of the operator  $L_{s_0,\cG}(\tau)$ from \eqref{LsGD} satisfies $u_{s_0}\in H^2(\Om)$ by the assumptions in the lemma. It follows that the action of the functional $\gaN u_{s_0}=\gaN^{\rm s}u_{s_0}$ on functions in $H^{1/2}(\dOm)$ is given by integration over the boundary $\dOm$. In particular, formula \eqref{henF1} makes sense, where $\nu=\nu(x)$ is the outward pointing normal unit vector to the boundary at  $x\in\dOm$ and $\gaN^{\rm s}u_{s_0}=\nu\cdot\gaD\nabla u_{s_0}\in L^2(\dOm)$ is the strong Neumann trace from \eqref{6.2}.

\begin{proof}
 Since $V(\cdot)$ is smooth by the assumptions, we may apply Lemma \ref{prelmon}. Since $\dot{t}(s_0)=1$ for $s_0\in\Sigma_2$, formula 
\eqref{mqq} yields
\begin{align}\nonumber
\mathfrak{m}_{s_0}&(q,q)+\frac1{(t(s_0))^{2}}\,\langle
\gaN u_{s_0},\gaD u_{s_0}\rangle_{1/2}=\frac1{t(s_0)}\langle \dot{V}_s\big|_{s=s_0}
 u_{s_0}, u_{s_0}\rangle_{L^2(\Om)}\\&=2\langle V(t(s_0)x)u_{s_0}, u_{s_0}\rangle_{L^2(\Om)}+t(s_0)\langle\big(\nabla V(t(s_0)x)\cdot x\big)u_{s_0}, u_{s_0}\rangle_{L^2(\Om)}.\lb{1vdot}
\end{align}
In the first term on the right-hand side of \eqref{1vdot} we can replace $V(t(s_0)x)$ by $(t(s_0))^{-2}V_{s_0}(x)$ and $V_{s_0}(x)u_{s_0}(x)$ by $\Delta u_{s_0}$, since $L_{s_0}u_{s_0}=-\Delta u_{s_0}+V_{s_0}u_{s_0}=0$.
For the second term in \eqref{1vdot} we claim the following formula:
\begin{equation}\label{n5.19}\begin{split}
\langle&\big(\nabla V(t(s_0)x)\cdot x\big)u_{s_0}, u_{s_0}\rangle_{L^2(\Om)}\\=&
-\frac1{t(s_0)}\big\langle V(t(s_0)x)u_{s_0}, \Big(2u_{s_0}\big(\nabla u_{s_0}\cdot x\big)+du_{s_0}\Big)\big\rangle_{L^2(\Om)}\\
&\quad+\frac1{t(s_0)}\int_\dOm \langle V(t(s_0)x)u_{s_0}(x), (\nu\cdot x)u_{s_0}(x)\rangle_{\bbR^N}\,dx.
\end{split}\end{equation}
To begin the proof of the claim we choose approximating functions $u_{s_0}^{(n)}\in C^\infty_0(\bbR^d)$ such that $u_{s_0}^{(n)}\to u_{s_0}$ in $H^1(\Om)$ as $n\to\infty$. Since $V\in L^\infty(\Om)$, the claim will follow by passing to the limit as $n\to\infty$ as soon as we know that equality  \eqref{n5.19} holds with $u_{s_0}$ replaced 
by $u_{s_0}^{(n)}$, so it suffices to prove \eqref{n5.19} assuming that $u_{s_0}$ is smooth. Passing to the components $V(x)=(V_{ij}(x))_{i,j=1}^N$ and $u_{s_0}(x)=(u_j(x))_{j=1}^N$ we have
\begin{equation*}
\langle\big(\nabla V(t(s_0)x)\cdot x\big)u_{s_0}, u_{s_0}\rangle_{L^2(\Om)}=\sum_{i,j=1}^N\int_\Om\big(\nabla V_{ij}(t(s_0)x)\big)\cdot(u_ju_ix)\,dx.
\end{equation*}
We apply formula \eqref{dfnwnu} for $N=1$ and each $i$, $j$, choosing $w(x)=(u_j(x)u_i(x))x\in\bbR^d$ and 
$\Phi(x)=(t(s_0))^{-1}V_{ij}(t(s_0)x)$, using the fact that $u_{s_0}$ is smooth. Indeed, we have
$w\in\dom(\nu\cdot\ast)$ by \eqref{dfnwnu} since $\text{div}(w)\in L^2(\Om)$ as $u_{s_0}$ is smooth and therefore \eqref{dfnwnu} implies
that the left-hand side of \eqref{n5.19} can be written as 
\begin{align}
\langle&\big(\nabla V(t(s_0)x)\cdot x\big)u_{s_0}, u_{s_0}\rangle_{L^2(\Om)}=
\sum_{i,j=1}^N\Big(-\frac1{t(s_0)}\int_\Om V_{ij}(t(s_0)x)\text{ div}\big((u_ju_i) x\big)\,dx
\nonumber\\&\hskip3cm+\frac1{t(s_0)}
\langle\nu\cdot \big(u_ju_i x\big), \gaD (V_{ij}(t(s_0)x))\rangle_{1/2}\Big).\label{5.20}\end{align}
This is equal to the right-hand side of \eqref{n5.19} by our notational conventions,
since the boundary term in \eqref{5.20} can be replaced by the respective integral over $\dOm$ as $u_{s_0}(x)$ is smooth. This completes the prove of claim \eqref{n5.19}.
Replacing $V_{s_0}(x)u_{s_0}(x)$ by $\Delta u_{s_0}$ in \eqref{n5.19} and using all this in \eqref{1vdot} yields
\begin{align}
\mathfrak{m}_{s_0}&(q,q)+\frac1{(t(s_0))^{2}}\,\langle
\gaN u_{s_0},\gaD u_{s_0}\rangle_{1/2}\nonumber\\&=
(t(s_0))^{-2}\Big((2-d)\langle\Delta u_{s_0}, u_{s_0}\rangle_{L^2(\Om)}
-2\langle\Delta u_{s_0},\big(\nabla u_{s_0}\cdot x\big)\rangle_{L^2(\Om)}\Big)
\nonumber\\
&\hskip3cm+\int_\dOm \langle V(t(s_0)x)u_{s_0}(x),(\nu\cdot x)u_{s_0}(x)\rangle_{\bbR^N}\,dx.\lb{2vdot}
\end{align}
Using Green's formula \eqref{wGreen}, a direct computation of partial derivatives yields
\begin{align*}
\int_\Om\Delta & u_{s_0}(x)\big(\nabla u_{s_0}\cdot x\big)\,dx
=-\int_\Om\nabla u_{s_0}(x)\cdot\nabla\big(\nabla u_{s_0}\cdot x\big)\,dx
\\&\hskip4cm
+\langle\gaN u_{s_0}, \gaD\big(\nabla u_{s_0}\cdot x\big)\rangle_{1/2}\\&=
-\frac12\int_\Om\Big(\text{div }\big((\nabla u_{s_0}\cdot\nabla u_{s_0}) x\big)
+(2-d)(\nabla u_{s_0}\cdot\nabla u_{s_0})\Big)\,dx
\\&\hskip4cm
+\langle\gaN u_{s_0}, \gaD\big(\nabla u_{s_0}\cdot x\big)\rangle_{1/2}.
\end{align*}
This relation, Green's formula \eqref{wGreen} applied in the first term in the right-hand side of \eqref{2vdot}, and the divergence theorem yield
\begin{align*}
\mathfrak{m}&_{s_0}(q,q)=
(t(s_0))^{-2}\Big(\int_\Om\text{div }\big((\nabla u_{s_0}\cdot\nabla u_{s_0}) x\big)\,dx
-2\langle\gaN u_{s_0}, \gaD\big(\nabla u_{s_0}\cdot x\big)\rangle_{1/2}\\&+(1-d)\langle\gaN u_{s_0}, \gaD u_{s_0}\rangle_{1/2}+\int_\dOm \langle V_{s_0}(t(s_0)x)u_{s_0}(x), (\nu\cdot x)u_{s_0}(x)\rangle_{\bbR^N}\,dx\Big)\\=&
(t(s_0))^{-2}\int_\dOm\Big((\nabla u_{s_0}\cdot\nabla u_{s_0})(x\cdot\nu)-2\langle \gaN^{\rm s} u_{s_0}(x), (\nabla u_{s_0}\cdot x)\rangle_{\bbR^N}\\&+(1-d)\langle\gaN^{\rm s} u_{s_0}(x),u_{s_0}(x)\rangle_{\bbR^N}+ \langle V_{s_0}(t(s_0)x)u_{s_0}(x), (\nu\cdot x)u_{s_0}(x)\rangle_{\bbR^N}\Big)\,dx
\end{align*}
where we have replaced $\langle\cdot\,, \cdot\,\rangle_{1/2}$ by the integrals over $\dOm$ since $u_{s_0}\in H^2(\Om)$. 
\end{proof}

Next, we give a crude assumption on the potential $V$ and the boundary operators $\Theta$ or $\Theta'$ that ensure negativity of the Maslov crossing form at a crossing $s_0\in\Sigma_2$.
\begin{corollary}\lb{for:rough}
Let $\cG$ be a Dirichlet- or Neumann-based subspace of the boundary space $\cH=H^{1/2}(\dOm)\times H^{-1/2}(\dOm)$. Assume Hypothesis \ref{h1} (i'), (ii"), (iii') and  (iv') (in particular, we assume that the boundary operator, $\Theta$ or $\Theta'$, is nonpositive), and
let $s_0\in\Sigma_2$ be a crossing. 
Then the crossing is negative provided the $(N \times N)$ matrix 
$W(x)=2V(t(s_0)x)+t(s_0)\nabla V(t(s_0)x)\cdot x$
is negative definite for each $x\in\Om$.
\end{corollary}
\begin{proof}
If $\cG=\gr'(\Theta')$ then \eqref{1vdot} and the assumptions of the corollary imply 
\[\mathfrak{m}_{s_0}(q,q)=-(t(s_0))^{-2}\,\langle
\Theta'\gaD u_{s_0},\gaD u_{s_0}\rangle_{1/2}+\langle W(x)u_{s_0},u_{s_0}\rangle_{L^2(\Om)}<0;\]
the case $\cG=\gr(\Theta)$ is similar.
\end{proof}

Finally, we show the monotonicity of the Maslov index (i.e. sign definiteness of the Maslov crossing form) for the path $s\mapsto\Upsilon(s)=T_s(\cK_s)$ when $s\in\Sigma_2$ and $\Theta'=0$ (that is, we consider the Dirichlet case $\cG=\cH_D=\{(0,g): g\in H^{-1/2}(\dOm)\}$). As mentioned above in Remark \ref{DC4.2}, in this case Hypothesis \ref{h1} (iii') and the inclusion $\dom(L_{s,\cH_D}(\tau))\subseteq H^2(\Om)$ hold for the operator $L_{s,\cH_D}(\tau)$ from \eqref{LsGD}. We recall that
$\nu=\nu(x)$ is the outward-pointing normal unit vector to the boundary at  $x\in\dOm$ and $\gaN^{\rm s}u_{s_0}=\nu\cdot\gaD\nabla u_{s_0}\in L^2(\dOm)$ is the strong Neumann trace from \eqref{6.2} provided $u_{s_0}\in H^2(\Om)$.

\begin{corollary}\lb{tmon} Assume Hypothesis \ref{h1} with (i'), $\Theta'=0$ (that is, the Dirichlet case $\cG=\cH_D$) and  (iv'). Let $L_s=-\Delta+V_s$ and $T_su=(\gaD u,(t(s))^{-1}\gaN u)$, let $\cK_s$ denote the set of weak solutions to $L_su=0$ and let $\Upsilon(s)=T_s(\cK_s)$. Then any crossing $s_0\in\Sigma_2$ is negative. Specifically, if $q\in T_{s_0}(\cK_{s_0})\cap\cH_D$ and $u_{s_0}\in\dom\big(L_{s_0,\cH_D}(\tau)\big)=H^2(\Om)\cap H^1_0(\Om)$ are such that $L_{s_0}u_{s_0}=0$ and $q=T_{s_0}u_{s_0}$, then the value of the Maslov crossing form on  $\Upsilon\big|_{\Sigma_2}$ is given as follows:
\beq\lb{henF}
\mathfrak{m}_{s_0}(q,q)=-\frac1{(t(s_0))^2}\int_{\dOm}\|\gaN^{\rm s} u_{s_0}\|_{\bbR^N}^2(\nu\cdot x)\,dx<0.
\enq
\end{corollary} 


Similar crossing form computations have appeared in \cite[Eqn.(5.32)]{DJ11}, \cite[Eqn.(32)]{PW} and \cite[Eqn.(35)]{CJM}. In fact, the right-hand side of \eqref{henF} is the well-known Rayleigh--Hadamard formula for the derivative of a simple Dirichlet eigenvalue with respect to a domain perturbation; see e.g. \cite[Chapter 5]{Henry} and references therein.


\begin{proof} 
We apply formula \eqref{henF1}, noting that the vector fields $\nabla u_{s_0}$ and $\nu$ are parallel in $\bbR^d$ since $\dOm$ is a level surface for $u_{s_0}$ (because $\gaD u_{s_0}=0$). Writing $\nabla u_{s_0}=(\nu\cdot\nabla u_{s_0})\nu=(\gaN^{\rm s}u_{s_0})\nu$ in the first two terms in \eqref{henF1} and dropping the last two terms because $u_{s_0}(x)=0$ for $x\in\dOm$ yields \eqref{henF}. Finally, $(\nu\cdot x)>0$ since $\Om$ is star-shaped.\end{proof}

\section{The proofs of the main results}\label{sec:mainres}

Our objective in this section is to prove Theorems \ref{tim:Dbased} and \ref{tim:Nbased}, that is, to compute the Morse index of the operator $L_\cG$ on the real Hilbert space $L^2(\Om; \bbR^N)$. Throughout, we assume Hypothesis \ref{h1}, as indicated in the theorems. We refer to Theorems \ref{GMthm} and \ref{GMthmH2} in Section \ref{sec4} for a specific set of assumptions under which these hypotheses are satisfied. The proof for the Dirichlet-based case is almost immediate, while the Neumann-based case requires a careful asymptotic analysis of the eigenvalues of $L_{0,\cG}(\tau)$ as $\tau\to0$.
To this end we complexify the operator (following \cite[Section 5.5.3]{We80}) and consider it on the complex Hilbert space $L^2(\Om; \bbC^N)$. Since $\Theta$ and $\Theta'$ are real operators and the potential $V$ takes values in the set of real, symmetric matrices, the spectrum of the complexified operator is real and
the multiplicities of its eigenvalues over $\bbR$ and $\bbC$ are the same, thus the complexification does not alter our computation of the Morse index.

\subsection{The Dirichlet-based case}
We now present the proof of Theorem \ref{tim:Dbased} for the Dirichlet-based case.

\begin{proof}[Proof of Theorem \ref{tim:Dbased}]
Since the Maslov index is a homotopy invariant \cite[Theorem 3.6(b)]{F} and the curve $\Gamma$ can be contracted to a point, we have $\mas(\Upsilon\big|_\Sigma)  =0$. Thus \[\mas(\Upsilon\big|_{\Sigma_3})=-\mas(\Upsilon\big|_{\Sigma_1})-\mas(\Upsilon\big|_{\Sigma_2})-\mas(\Upsilon\big|_{\Sigma_4})\]
since the Maslov index is additive under catenation of paths
\cite[Theorem 3.6(a)]{F}. By Lemma \ref{lambdaMon} we know that $\mor(L_{\cG})=\mas(\Upsilon\big|_{\Sigma_3})$, and Lemmas \ref{lem:Gamma4} and  \ref{lem:Gamma4Dir} imply $\mas(\Upsilon\big|_{\Sigma_1})=\mas(\Upsilon\big|_{\Sigma_4})=0$. Combining these facts, we obtain formula \eqref{MasvsMorDir}. The last assertion in the theorem follows from Corollary \ref{tmon} and \eqref{eqdims}. Since the path $\Upsilon\big|_{\Sigma_2}$ is negative definite, a crossing on the top right corner of the square in Figure \ref{fig1} will not contribute to the Maslov index (cf. \eqref{MIcomp}), hence the sum on the right-hand side of \eqref{Dfor} does not include $t=1$.
\end{proof}

\subsection{Asymptotic expansions as $\tau\to0$}
We next consider the Neumann-based case. Here the analysis of the crossings on $\Gamma_1$ is more involved, and the proof requires several preliminaries. Specifically, we have to analyze the operator $L_{0,\cG}(\tau)$ from \eqref{dfnLst00}, which corresponds to the right end of the segment $\Gamma_1$ of the curve $\Gamma$ in Figure \ref{fig1}. Our ultimate goal is an asymptotic formula for the eigenvalues of $L_{0,\cG}(\tau)$ that bifurcate from the zero eigenvalue of $L_{0,\cG}(0)=-\Delta_N$ as $\tau\to0$; see equation \eqref{eq1.84} in the next subsection.
In this subsection we set the stage and establish in Lemma \ref{lem:simile} an asymptotic formula for a finite-dimensional part of the operator $L_{0,\cG}(\tau)$. Throughout, we assume Hypothesis \ref{h1} (i), (ii) and (iii).

Let $\cG=\gr(\Theta)$ be a Neumann-based subspace in $\cH=H^{1/2}(\dOm)\times H^{-1/2}(\dOm)$ and consider the sesquilinear form $\mathfrak{l}_{0,\cG} (\tau)$ on $L^2(\Om)$, defined for $\tau\in[0,1]$
by
\begin{align}
\mathfrak{l}_{0,\cG}(\tau)(u,v)&=\langle\nabla u,\nabla v\rangle_{L^2(\Om)}+\tau^2\langle V(\tau x)u,v\rangle_{L^2(\Om)}-\tau \langle\Theta\gaD  u,\gaD   v\rangle_{1/2},
\no\\
\dom(\mathfrak{l}_{0,\cG} (\tau))&=H^1(\Omega)\times H^1(\Omega).\lb{dfnlGt}
\end{align}
One can associate to the form $\mathfrak{l}_{0,\cG}(\tau)$ the differential operator $L_{0,\cG}(\tau)$ given in \eqref{dfnLst00} using the following result, which can be found in, e.g.,
\cite[Theorem 2.6]{GM08}. We recall from \eqref{dfnLst00} and \eqref{dfnLst00n}  that the operator $L_{0,\cG}(\tau)$ for $\tau>0$ corresponds to the lower right end of the rectangle $\Gamma$ in Figure \ref{fig1} and that $L_{0,\cG}(0)=-\Delta_N$ is the Neumann Laplacian.
\begin{theorem}\lb{th26GM}
Assume Hypotheses \ref{h1} (i) and \ref{hypTHETA}. Then the sesquilinear form $\mathfrak{l}_{0,\cG} (\tau)$ is symmetric, bounded from below and closed in 
$L^2(\Omega)\times L^2(\Omega)$. The associated selfadjoint, bounded from below in $L^2(\Om)$ operator $L_{0,\cG}(\tau)$, satisfying the relation
\beq\lb{OPFORM}
\mathfrak{l}_{0,\cG}(\tau)(u,v)=\langle L_{0,\cG}(\tau) u, v\rangle_{L^2(\Om)}\,\text{ for all }
\, u\in\dom(L_{0,\cG}(\tau)), v\in H^1(\Om),
\enq is given by
\begin{align}
L_{0,\cG}(\tau)u(x)&=-\Delta u(x)+\tau^2V(\tau x)u(x),\, x\in\Om, u\in\dom(L_{0,\cG}(\tau)),\nonumber\\
 \dom(L_{0,\cG}(\tau))&=\big\{u\in H^1(\Omega)\,|\,\Delta u\in L^2(\Omega) \,\text{ and }\lb{6.2.1}\\
&\hskip3cm (\gaN-\tau \Theta\gaD)u=0\,\,\hbox{in}\,\,H^{-1/2}(\dOm) \big\}.\nonumber
\end{align}
\end{theorem}
Since in this section we assume only Hypothesis \ref{h1} (iii) (and not the stronger Hypothesis \ref{hypTHETA}, cf.\ Theorem \ref{GMthm}), we will not use all the properties of $\mathfrak{l}_{0,\cG}(\tau)$ listed in Theorem \ref{th26GM}. The operator $L_{0,\cG}(\tau)$ defined in \eqref{dfnLst00} or, equivalently, \eqref{6.2.1}, is selfadjoint by Hypothesis \ref{h1} (iii). For this operator and the form defined in \eqref{dfnlGt} we will use  \eqref{OPFORM}, which follows directly from 
Green's formula \eqref{wGreen}. 

Following the general discussion of holomorphic families of closed, unbounded operators in \cite[Section VII.1.2]{Kato}, we introduce our next definition.
\begin{definition}\label{cont}
A family of closed operators $\{T(\tau)\}_{\tau\in \Sigma_0}$ on a Hilbert space $\cX$ is said to be continuous on an interval $\Sigma_0\subset\bbR$ if there exists a Hilbert space $\cX'$ and continuous families of operators $\{U(\tau)\}_{\tau\in \Sigma_0}$ and $\{W(\tau)\}_{\tau\in \Sigma_0}$ in $\cB(\cX',\cX)$ such that $U(\tau)$ is a one-to-one map of $\cX'$ onto $\dom(T(\tau))$ and the identity $T(\tau)U(\tau)=W(\tau)$ holds for all $\tau\in \Sigma_0$.
\end{definition}
Before applying this definition, we recall that Hypothesis \ref{h1} (iv) yields
\begin{equation}\label{unif1}
\sup_{x\in\Om}\|V(\tau x)-V(\tau_0x)\|_{\bbR^{N\times N}}\to 0 \,\,\,\hbox{as}\,\,\, \tau\to \tau_0\,\text{ 
for any $\tau_0\in[0,1]$}.\end{equation}
\begin{lemma}
Assume Hypothesis \ref{h1} (i) and (iii). Then the family $\{L_{0,\cG}(\tau)\}_{\tau\in \Sigma_0}$ is continuous near $0$, that is, on some interval $\Sigma_0$ that contains $0$.
\end{lemma}

\begin{proof}
Letting $\tau=0$ in \eqref{dfnlGt} yields
\[(\mathfrak{l}_{0,\cG} (0)+1)(u,u)=\langle\nabla u,\nabla u\rangle_{L^2(\Om)}+\langle u,u\rangle_{L^2(\Om)}\geq \langle u,u\rangle_{L^2(\Om)}\,\text{ for  all $u\in H^1(\Omega)$}.\]

The selfadjoint operator associated with the form $\mathfrak{l}_{0,\cG} (0)+1$ by Theorem \ref{th26GM} is $L_{0,\cG}(0)+I_{L^2(\Om)}$, where $L_{0,\cG}(0)=-\Delta_N$ is the Neumann Laplacian. It is clear that Hypothesis \ref{hypTHETA} holds for the Neumann boundary condition, as $\Theta=0$ (cf.\ \eqref{dfnLst00n}).
The operator $-\Delta_N+I_{L^2(\Om)}$ is clearly invertible, and 
if $u\in H^1(\Omega)$, then \[(-\Delta_N+I_{L^2(\Om)})^{-1/2}u\in\dom(-\Delta_N+I_{L^2(\Om)})^{1/2}=H^1(\Omega).\]
For the remainder the proof we let $G$ denote the operator
\[
	G = (-\Delta_N+I_{L^2(\Om)})^{1/2}\colon H^{1}(\Omega) \rightarrow L^2(\Omega).
\]
We note, cf.\ (B.42) and (B.43) in \cite{GM08}, that
$G\in\cB(H^{1}(\Omega),L^2(\Omega))$ and 
\beq\lb{6.6.1}
\|Gu\|_{L^2(\Omega)}=\|u\|_{H^1(\Omega)}\,\text{ for all }\, u\in H^1(\Om),\enq  hence $G^{-1} = (-\Delta_N+I_{L^2(\Om)})^{-1/2}\in\cB(L^2(\Omega),H^{1}(\Omega))$.

Next, we let $u,v\in L^2(\Omega)$ be such that $\|u\|_{L^2(\Omega)},\|v\|_{L^2(\Omega)}\leq1$. 
Using the fact that $\Theta \in \cB(H^{1/2}(\partial\Omega),H^{-1/2}(\partial\Omega))$ and $\gaD\in\cB(H^{1}(\Omega),H^{1/2}(\partial\Omega))$, we find
\begin{align}\label{unif}
\big|\langle&\Theta\gaD G^{-1}u,\gaD G^{-1}v\rangle_{1/2}\big|
\leq
c \|G^{-1}u\|_{H^{1}(\Omega)}\|G^{-1}v\|_{H^{1}(\Omega)} \leq c.
\end{align}
Now using \eqref{dfnlGt} we introduce a new sesquilinear form
\begin{align}
&\widetilde{\mathfrak{l}}_{0,\cG}(\tau)(u,v)=\mathfrak{l}_{0,\cG} (\tau)\left(G^{-1}u,G^{-1}v \right),\\
&\dom(\widetilde{\mathfrak{l}}_{0,\cG}(\tau))=L^2(\Omega)\times L^2(\Omega).
\end{align}
From \eqref{6.6.1} and \eqref{unif} it is easy to see that $\widetilde{\mathfrak{l}}_{0,\cG}(\tau)$ is bounded on  $L^2(\Omega)\times L^2(\Omega)$. Let $\widetilde{L}_{0,\cG}(\tau)\in\cB(L^2(\Omega))$ be the selfadjoint operator associated with $\widetilde{\mathfrak{l}}_{0,\cG}(\tau)$ by the First Representation Theorem \cite[Theorem VI.2.1]{Kato}. Then 
\begin{align}\lb{LLtil}
\langle \widetilde{L}_{0,\cG}&(\tau)u,v\rangle_{L^2(\Om)}=\widetilde{\mathfrak{l}}_{0,\cG}(\tau)(u,v)\\&=\mathfrak{l}_{0,\cG}(\tau)\left(G^{-1}u,G^{-1}v\right) \,\text{ for all} \,u,v\in L^2(\Omega).\no
\end{align}
Taking into account \eqref{unif} and \eqref{unif1}, we conclude that
\begin{equation}
\langle\widetilde{L}_{0,\cG}(\tau)u,v\rangle_{L^2(\Om)}\rightarrow\langle\widetilde{L}_{0,\cG}(\tau_0)u,v\rangle_{L^2(\Om)}\,\,\hbox{as}\,\,\tau\to \tau_0
\end{equation}
uniformly with respect to $u$ and $v$ satisfying $\|u\|_{L^2(\Omega)},\|v\|_{L^2(\Omega)}\leq1$. Hence 
\begin{equation}
\|\widetilde{L}_{0,\cG}(\tau)-\widetilde{L}_{0,\cG}(\tau_0)\|_{\cB(L^2(\Omega))}\rightarrow0\,\,\hbox{as}\,\,\tau\to \tau_0,
\end{equation}
which implies $\widetilde{L}_{0,\cG}(\tau)\in\cB(L^2(\Omega))$ is a continuous family on $[0,1]$.
Replacing $u$ in \eqref{LLtil} by $Gu$ (and similarly for $v$), we conclude that
\[
\mathfrak{l}_{0,\cG}(\tau)(u,v)=\left<\widetilde{L}_{0,\cG}(\tau)Gu,Gv\right>_{L^2(\Om)}\] for any $u,v\in H^{1}(\Omega)$.
Therefore, cf.\ \cite[VII-(4.4), (4.5)]{Kato}, for all $u\in\dom(L_{0,\cG})$
\beq\lb{simLL}
L_{0,\cG}(\tau)u=G\widetilde{L}_{0,\cG}(\tau)Gu,\enq
when $G$ is viewed as an unbounded, selfadjoint operator on $L^2(\Om)$.
It follows that $L_{0,\cG}(\tau)+I_{L^2(\Omega)}=G\widetilde{L}_{0,\cG}(\tau) G+I_{L^2(\Omega)}$, so we have the identity
\[L_{0,\cG}(\tau)+I_{L^2(\Omega)}=G\Big(\widetilde{L}_{0,\cG}(\tau)+G^{-2}\Big)G.\] Since
$L_{0,\cG}(0)+I_{L^2(\Omega)}=-\Delta_N+I_{L^2(\Om)} = G^2$, using \eqref{simLL} with $\tau=0$ we conclude that $\widetilde{L}_{0,\cG}(0)+G^{-2}=I_{L^2(\Om)}$. Since $\widetilde{L}_{0,\cG}(\tau)\in\cB(L^2(\Omega))$ is a continuous operator family on $[0,1]$, the operator 
$\widetilde{L}_{0,\cG}(\tau)+G^{-2}$ 
is boundedly  invertible for $\tau$ near $0$.  Using \eqref{simLL} again, we conclude that for any $\tau$ near $0$ one has
\begin{equation*}
(L_{0,\cG}(\tau)+I_{L^2(\Omega)})^{-1}=G^{-1}\Big(\widetilde{L}_{0,\cG}(\tau)+G^{-2}\Big)^{-1}G^{-1}.
\end{equation*}

Introducing the continuous operator families $U(\tau)=(L_{0,\cG}(\tau)+I_{L^2(\Omega)})^{-1}$ and $W(\tau)=I_{L^2(\Om)}-U(\tau)$, it is now clear that 
\begin{equation}\label{UV}
L_{0,\cG}(\tau)U(\tau)=W(\tau)\,\text{ for $\tau$ near $0$}.
\end{equation} Hence, according to Definition \ref{cont} applied to $T(\tau)=L_{0,\cG}(\tau)+I_{L^2(\Om)}$, the family $\{L_{0,\cG}(\tau)\}$ is continuous near $0$. 
\end{proof}
We denote by
\[R(\zeta,\tau)=\big(L_{0,\cG}(\tau)-\zeta I_{L^2(\Om)}\big)^{-1},\quad \zeta\in\bbC\setminus\Sp(L_{0,\cG}(\tau)),\, \tau\in[0,1],\]
the resolvent operator for $L_{0,\cG}(\tau)$ in $L^2(\Om)$, and recall that $L_{0,\cG}(0)=-\Delta_N$, the Neumann Laplacian.
\begin{lemma}\label{rescont}
Let $\zeta\in\bbC\setminus\Sp(-\Delta_N)$. Then $\zeta\in\bbC\setminus\Sp(L_{0,\cG}(\tau))$ for $\tau$ near $0$. Moreover, the function 
$\tau \mapsto R(\zeta,\tau)\in\cB(L^2(\Om))$ is continuous for $\tau$ near $0$, uniformly for $\zeta$ in compact subsets of $\bbC\setminus\Sp(L_{0,\cG}(\tau))$.
\end{lemma}
\begin{proof}
Let $\zeta\in\bbC\setminus\Sp(-\Delta_N)$. It follows from \eqref{UV} that
\begin{equation}\lb{eq1.26}
(L_{0,\cG}(\tau)-\zeta)U(\tau)=W(\tau)-\zeta U(\tau)\, \text{ for $\tau$ near $0$}.
\end{equation}
The operator  $W(0)-\zeta U(0)=(L_{0,\cG}(0)-\zeta)(L_{0,\cG}(0)+I_{L^2(\Om)})^{-1}$ is a bijection of $L^2(\Omega)$ onto $L^2(\Omega)$. Therefore, $W(\tau)-\zeta U(\tau)$
is boundedly  invertible for $\tau$ near $0$ since the function $\tau \mapsto W(\tau)-\zeta U(\tau)$ is continuous in $\cB(L^2(\Om))$, uniformly for $\zeta$ in compact subsets of $\bbC$. This implies that $\zeta\in\bbC\setminus\Sp(L_{0,\cG}(\tau))$ for $\tau$ near $0$, since using \eqref{eq1.26} it is easy to check that 
\begin{equation}
(L_{0,\cG}(\tau)-\zeta)^{-1}=U(\tau)(W(\tau)-\zeta U(\tau))^{-1}\, \text{ for $\tau$ near $0$}.
\end{equation}
Hence, the function $\tau \mapsto R(\zeta,\tau)$ is continuous for $\tau$ near $0$ in the operator norm, uniformly in $\zeta$.
\end{proof}
Our next lemma gives an asymptotic result for the difference of the resolvents of the operators $L_{0,\cG}(\tau)$ and $L_{0,\cG}(0)=-\Delta_N$ as $\tau\to0$ which involves the value $V(0)$ of the potential at zero. We recall the inclusion \[[\gaD   R(\zeta, 0)^*]^*\in\cB(H^{-1/2}(\partial\Omega),H^{1}(\Omega))\,\text{ for }\zeta\in\bbC\setminus\Sp(-\Delta_N)\] (see e.g., \cite[Equation  (4.13)]{GM08})  which, in particular, yields 
\begin{equation}\label{adj}
\langle [\gaD   R(\zeta, 0)^*]^*g,v)\rangle_{L^2(\Om)}=\langle g,\gaD   R(\zeta, 0)^*v\rangle_{1/2}
\end{equation}
for any $g\in H^{-1/2}(\partial\Omega)$ and $v\in L^2(\Omega)$.
\begin{lemma}
If $\zeta\in\bbC\setminus\Sp(-\Delta_N)$ and $\|u\|_{L^2(\Om)}\le1$, then 
\begin{equation}\label{resolvent}
\begin{split}
R(\zeta, \tau)u&-R(\zeta, 0)u\\&
=\tau[\gaD   R(\zeta, 0)^*]^*\Theta\gaD R(\zeta, 0)u-\tau^2R(\zeta, 0)V(0)R(\zeta, 0)u\\&\quad+\tau^2[\gaD   R(\zeta, 0)^*]^*\Theta\gaD[\gaD   R(\zeta, 0)^*]^*\Theta\gaD  R(\zeta, 0)u+r(\tau),
\end{split}
\end{equation}
where $\|r(\tau)\|_{L^2(\Om)}=\mathrm{o}(\tau^{2})$ as $\tau\to0$, uniformly for $\zeta$ in compact subsets of $\bbC\setminus\Sp(-\Delta_N)$ and $\|u\|_{L^2(\Om)}\le1$.
\end{lemma}
\begin{proof}
We recall that 
  $\zeta\in\bbC\setminus\Sp(L_{0,\cG}(\tau))$ for $\tau$ near $0$ by Lemma \ref{rescont}, since
  $\zeta\in\bbC\setminus\Sp(-\Delta_N)$, and define $w=R(\zeta, \tau)u-R(\zeta, 0)u$ for $\|u\|_{L^2(\Om)}\le1$. Using \eqref{OPFORM}, the definition of $w$ and the definition of the form $\mathfrak{l}_{0,\cG}(u,v)$, we obtain
\begin{align}\label{res}
\begin{split}
&\langle w,v\rangle_{L^2(\Om)}=({l}_{0,\cG}(0)-\zeta)(R(\zeta, 0)w,v)=({l}_{0,\cG}(0)-\zeta)(w,R(\zeta, 0)^*v)\\
&=({l}_{0,\cG}(0)-\zeta)(R(\zeta, \tau)u,R(\zeta, 0)^*v)-({l}_{0,\cG}(0)-\zeta)(R(\zeta, 0)u,R(\zeta, 0)^*v)\\
&=({l}_{0,\cG}(\tau)-\zeta)(R(\zeta, \tau)u,R(\zeta, 0)^*v)-\tau^2\langle V(\tau x)R(\zeta, \tau)u,R(\zeta, 0)^*v\rangle_{L^2(\Om)}\\&
 \quad+\tau \langle\Theta\gaD  R(\zeta, \tau)u,\gaD   R(\zeta, 0)^*v\rangle_{1/2}-({l}_{0,\cG}(0)-\zeta)(R(\zeta, 0)u,R(\zeta, 0)^*v)\\
&=\langle u,R(\zeta, 0)^*v\rangle_{L^2(\Om)}-\tau^2\langle V(\tau x)R(\zeta, \tau)u,R(\zeta, 0)^*v\rangle_{L^2(\Om)}\\&
 \quad+\tau \langle\Theta\gaD  R(\zeta, \tau)u,\gaD   R(\zeta, 0)^*v\rangle_{1/2}-\langle u,R(\zeta, 0)^*v\rangle_{L^2(\Om)}\\
 &=-\tau^2\langle V(\tau x)R(\zeta, \tau)u,R(\zeta, 0)^*v\rangle_{L^2(\Om)}\\&
  \hskip3cm+\tau \langle\Theta\gaD  R(\zeta, \tau)u,\gaD   R(\zeta, 0)^*v\rangle_{1/2}.
\end{split}
\end{align}
Using \eqref{res} and \eqref{adj} we arrive at
\begin{align}
\begin{split}
&\langle w,v\rangle_{L^2(\Om)}=-\tau^2\langle R(\zeta, 0)V(\tau x)R(\zeta, \tau)u,v\rangle_{L^2(\Om)}\\&
  \hskip3cm+\tau \langle [\gaD   R(\zeta, 0)^*]^*\Theta\gaD  R(\zeta, \tau)u,v\rangle_{L^2(\Om)},
\end{split}
\end{align}
hence
\begin{align}\label{formula}
\begin{split}
&R(\zeta, \tau)u=R(\zeta, 0)u-\tau^2R(\zeta, 0)V(\tau x)R(\zeta, \tau)u
 \\&\hskip3cm +\tau[\gaD   R(\zeta, 0)^*]^*\Theta\gaD  R(\zeta, \tau)u.
\end{split}
\end{align}
Replacing $ R(\zeta, \tau)u$ in the right-hand  side of \eqref{formula} again by \eqref{formula} yields
\begin{align}
R&(\zeta, \tau)u=R(\zeta, 0)u
-\tau^2R(\zeta, 0)V(\tau x)\Big(R(\zeta, 0)u-\tau^2R(\zeta, 0)V(\tau x)R(\zeta, \tau)u\no\\&\hskip3cm
  +\tau[\gaD   R(\zeta, 0)^*]^*\Theta\gaD  R(\zeta, \tau)u\Big)\no\\
  &+\tau[\gaD   R(\zeta, 0)^*]^*\Theta\gaD\Big(R(\zeta, 0)u-\tau^2R(\zeta, 0)V(\tau x)R(\zeta, \tau)u
   \label{fformula}\\&\hskip3cm +\tau[\gaD   R(\zeta, 0)^*]^*\Theta\gaD  R(\zeta, \tau)u\Big)\no\\
            &=R(\zeta, 0)u+\tau[\gaD   R(\zeta, 0)^*]^*\Theta\gaD R(\zeta, 0)u-\tau^2R(\zeta, 0)V(0)R(\zeta, 0)u\no\\&\qquad+\tau^2[\gaD   R(\zeta, 0)^*]^*\Theta\gaD[\gaD   R(\zeta, 0)^*]^*\Theta\gaD  R(\zeta, 0)u+r(\tau),\no
\end{align}
where
\begin{equation}
\begin{split}
r(\tau)=-&\tau^2R(\zeta, 0)\big(V(\tau x)-V(0)\big)R(\zeta, 0)u\\&+\tau^2[\gaD   R(\zeta, 0)^*]^*\Theta\gaD[\gaD   R(\zeta, 0)^*]^*\Theta\gaD
\big(R(\zeta, \tau)u-R(\zeta, 0)u\big)\\&\qquad-\tau ^3R(\zeta, 0)V(\tau x)[\gaD   R(\zeta, 0)^*]^*\Theta\gaD  R(\zeta, \tau)u\\&\qquad\qquad-\tau ^3[\gaD   R(\zeta, 0)^*]^*\Theta\gaD R(\zeta, 0)V(\tau x)R(\zeta, \tau)u\\&
\qquad\qquad\qquad+\tau ^4R(\zeta, 0)V(\tau x)R(\zeta, 0)V(\tau x)R(\zeta, \tau)u.
\end{split}
\end{equation}
Finally, we remark that $\|w\|_{L^2(\Om)}\to0$ and $\|R(\zeta,\tau)\|_{\cB(L^2(\Om))}$ is bounded as $\tau\to0$ by Lemma \ref{rescont} and thus, using \eqref{unif1} for $\tau_0=0$, we conclude that $\|r(\tau)\|_{L^2(\Om)}=\mathrm{o}(\tau^{2})$ as $\tau\to0$, uniformly for $\zeta$ in compact subsets of $\bbC\setminus\Sp(L_{0,\cG}(\tau))$ and $\|u\|_{L^2(\Om)}\le1$.
\end{proof}

Our next objective is to study the asymptotics of the eigenvalues of $L_{0,\cG}(\tau)$ that bifurcate from the zero eigenvalue of $L_{0,\cG}(0)=-\Delta_N$ for $\tau$ near zero. To begin, we will separate the spectrum of $L_{0,\cG}(\tau)$. First, we note that $0$ is the only nonpositive eigenvalue in $\Sp(L_{0,\cG}(0))$. Using Lemma \ref{lem:Gamma4}, let us fix $\Lambda>0$ so large that $-\Lambda/2<\lambda$ for all $\lambda\in\Sp(L_{0,\cG}(\tau))$ and $\tau\in[0,1]$.
By the standard spectral mapping theorem, see, e.g., \cite[Theorem IV.1.13]{EnNa}, we infer
\beq\lb{spmth}
\Sp\big((-\Lambda-L_{0,\cG}(\tau))^{-1}\big)\setminus\{0\}=\big\{(-\Lambda-\lambda)^{-1}:
\lambda\in\Sp(L_{0,\cG}(\tau))\big\}, \, \tau\in[0,1].
\enq
In particular, $-1/\Lambda\in\Sp\big((-\Lambda-L_{0,\cG}(0))^{-1}\big)$ since $0\in\Sp(L_{0,\cG}(0))$. Now fix a sufficiently small $\varepsilon\in(0,1/(2\Lambda))$ such that the disc of radius $2\varepsilon$ centered at the point $-1/\Lambda\in\bbC$  does not contain any other eigenvalues in $\Sp\big((-\Lambda-L_{0,\cG}(0))^{-1}\big)$ except  $-1/\Lambda$. Using Lemma \ref{rescont} we know that  $(-\Lambda-L_{0,\cG}(\tau))^{-1}\to(-\Lambda-L_{0,\cG}(0))^{-1}$ in $\cB(L^2(\Om))$ as $\tau\to0$. By the upper semicontinuity of the spectra of bounded operators, see, e.g., \cite[Theorem IV.3.1]{Kato}, there exists a $\delta=\delta(\varepsilon)$ such that if $\tau\in[0,\delta]$, then 
\beq\lb{spincl}
\Sp\big((-\Lambda-L_{0,\cG}(\tau))^{-1}\big)\subset\{\mu:\dist\big(\mu,\Sp\big((-\Lambda-L_{0,\cG}(0))^{-1}\big)\big)<\varepsilon\big\}.
\enq

In the remaining part of this section we assume that $\tau \le\delta$. Let $\{\nu_\ell(\tau)\}_{\ell=1}^{n(\tau)}\subset\Sp\big((-\Lambda-L_{0,\cG}(\tau))^{-1}\big)$ denote the eigenvalues of $(-\Lambda-L_{0,\cG}(\tau))^{-1}$ which are located inside of the disc of radius $\varepsilon$ centered at the point $-1/\Lambda\in\bbC$ and let $\lambda_\ell(\tau)=-\Lambda-1/\nu_\ell(\tau)$ be the respective eigenvalues of $L_{0,\cG}(\tau)$. Let $\gamma$ be a small circle centered at zero which encloses the eigenvalues $\lambda_\ell(\tau)$ for all $\ell=1,\dots,n(\tau)$ and $\tau\in[0,\delta]$ and separates them from the rest of the spectrum of $L_{0,\cG}(\tau)$. By choosing $\varepsilon$ sufficiently small we will assume that $\gamma$ also separates $0\in\Sp(L_{0,\cG}(0))$ from the rest of the spectrum of $L_{0,\cG}(0)$. 

We let $P$ denote the orthogonal Riesz projection for $L_{0,\cG}(0)$ corresponding to the eigenvalue $0\in\Sp(L_{0,\cG}(0))$ and let $\ran({P})=\ker(L_{0,\cG}(0))$ be the $N$-dimensional subspace in $L^2(\Om; \bbC^N)$ spanned by the constant vector-valued functions $\{\e_j\}_{j=1}^N$, that is, the standard unit vectors in $\bbC^N$.
 We denote by $\cS \subset H^{1/2}(\dOm)$ the corresponding subspace of constant vector-valued functions on $\dOm$. Also, we
let $P(\tau)$ denote the Riesz spectral protection for $L_{0,\cG}(\tau)$ corresponding to the eigenvalues $\{\lambda_\ell(\tau)\}\subset\Sp(L_{0,\cG}(\tau))$, that is,
\beq\lb{dfnRPr}
P=\frac{1}{2\pi i}\int_\gamma(\zeta-L_{0,\cG}(0))^{-1}\,d\zeta,\,
P(\tau)=\frac{1}{2\pi i}\int_\gamma(\zeta-L_{0,\cG}(\tau))^{-1}\,d\zeta.
\enq
Our objective is to establish an asymptotic formula for the eigenvalues $\lambda_\ell(\tau)$ as $\tau\to0$ similar to \cite[Theorem II.5.11]{Kato}, which is valid for families of bounded operators on finite-dimensional spaces. We stress that one can not directly use a related result \cite[Theorem VIII.2.9]{Kato} for families of unbounded operators, as the $\tau$-dependence of $L_{0,\cG}(\tau)$ in our case  is more complicated than allowed in the latter theorem. We are thus forced to mimic the main strategy of \cite{Kato} in order to extend the relevant results to the family $\{L_{0,\cG}(\tau)\}_{\tau\in[0,\delta]}$. 
\begin{remark}\lb{rem:LPt} The main part of the proof of Theorem \ref{tim:Nbased} consists of counting the negative eigenvalues of the operator $L_{0,\cG}(\tau)$ for $\tau$ near zero. We claim that this number is equal to the number of negative eigenvalues of the operator $L_{0,\cG}(\tau)P(\tau)$ for $\tau$ near zero, that is, of the restriction of $L_{0,\cG}(\tau)$ to the finite-dimensional subspace $\ran(P(\tau))$. Indeed, by the spectral mapping theorem \eqref{spmth}, $\lambda<0$ is in $\Sp(L_{0,\cG}(\tau))$ if and only if $(-\Lambda-\lambda)^{-1}<-1/\Lambda$. Thus for $\tau$ near zero the negative eigenvalues of $L_{0,\cG}(\tau)$ are in one-to-one correspondence with the eigenvalues $\nu_\ell(\tau)\in\Sp\big((-\Lambda-L_{0,\cG}(\tau))^{-1}\big)$ that satisfy the inequality $\nu_\ell(\tau)<-1/\Lambda$, and therefore with the negative eigenvalues among $\lambda_\ell(\tau)\in\Sp\big(L_{0,\cG}(\tau)P(\tau)\big)$ as claimed. 
\hfill$\Diamond$\end{remark}

We recall that $L_{0,\cG}(0)P=0$ and thus the standard formula for the resolvent decomposition, see, e.g., \cite[Section III.6.5]{Kato}, yields
\begin{equation}\lb{decomp}
	R(\zeta,0)=(-\zeta)^{-1}P+\sum_{n=0}^\infty\zeta^nS^n,
\end{equation}	
where
\begin{equation}
	S=\frac{1}{2\pi i}\int_\gamma\zeta^{-1}R(\zeta,0)\,d\zeta \lb{decomp1}
\end{equation}
is the reduced resolvent for the operator $L_{0,\cG}(0)$ in $L^2(\Om)$; in particular, $PS=SP=0$. We introduce the notation $D(\tau)=P(\tau)-P$. Applying $-\frac{1}{2\pi i}\int_\gamma(\cdot)\,d\zeta$ in \eqref{resolvent} and using \eqref{dfnRPr} we conclude that $D(\tau)=
\mathrm{o}(\tau)$ as $\tau\to 0$. Therefore $I_{L^2(\Om)}-D^2(\tau)$ is strictly positive for $\tau$ near zero and, following \cite[Section I.4.6]{Kato}, we may introduce mutually inverse operators $U(\tau)$ and $U(\tau)^{-1}$ in $\cB(L^2(\Om))$ as follows:
\begin{equation}\label{dfnUUinv}
\begin{split}
U(\tau)&=(I-D^2(\tau))^{-1/2}\big((I-P(\tau))(I-P)+P(\tau)P\big),\\
U(\tau)^{-1}&=(I-D^2(\tau))^{-1/2}\big((I-P)(I-P(\tau))+PP(\tau)\big).
\end{split}
\end{equation}
As discussed in Remark \ref{DalKr}, the transformation operator $U(\tau)$ splits the projections $P$ and $P(\tau)$, that is, 
\begin{equation}\label{up}
U(\tau)P=P(\tau)U(\tau),
\end{equation}
so that $U(\tau)$ is an isomorphism of the $N$-dimensional subspace $\ran({P})$ onto the subspace $\ran({P(\tau)})$. We isolate the main technical steps in the proof of Theorem \ref{tim:Nbased} in the following lemma.
\begin{lemma}\lb{lem:simile} Let $P$ be the Riesz projection for $L_{0,\cG}(0)$ onto the subspace $\ran({P})=\ker(L_{0,\cG}(0))$ and $P(\tau)$ be the respective Riesz projection for $L_{0,\cG}(\tau)$ from \eqref{dfnRPr}, let $S$ be the reduced resolvent for $L_{0,\cG}(0)$ defined in \eqref{decomp1}, and let the transformation operators $U(\tau)$ and $U(\tau)^{-1}$ be defined in \eqref{dfnUUinv}. Then
\begin{align}\label{pulup}
\begin{split}
PU(\tau)^{-1}&L_{0,\cG}(\tau)P(\tau)U(\tau)P
=-\tau [\gaD   P]^*\Theta\gaD P+\tau^2PV(0)P\\
&\qquad-\tau^2[\gaD   P]^*\Theta\gaD[\gaD   S]^*\Theta\gaD  P+\mathrm{o}(\tau^2)
\text{ as $\tau\to0$}.
\end{split}
\end{align}
\end{lemma}
\begin{proof} We will split the proof into several steps.

{\em Step 1.}\, We first claim the following asymptotic relations for $\zeta\in\gamma$:
\begin{align}
\label{rp}
R&(\zeta, \tau)P=(-\zeta)^{-1}P+\tau (-\zeta)^{-1}[\gaD   R(\zeta, 0)^*]^*\Theta\gaD P+\mathrm{o}(\tau)_u,\\
PR&(\zeta, \tau)
  =(-\zeta)^{-1}P+\tau (-\zeta)^{-1}[\gaD  P]^*\Theta\gaD R(\zeta, 0)+\mathrm{o}(\tau)_u,\label{eq1.63}\\
PR&(\zeta, \tau)P=(-\zeta)^{-1}P+\tau (-\zeta)^{-2}[\gaD   P]^*\Theta\gaD P-\tau^2(-\zeta)^{-2}PV(0)P\nonumber\\&\quad+\tau^2(-\zeta)^{-2}[\gaD   P]^*\Theta\gaD[\gaD   R(\zeta, 0)^*]^*\Theta\gaD  P+\mathrm{o}(\tau^{2})_u, \label{prp}\\
(I-P)R&(\zeta, \tau)P=\tau  (-\zeta)^{-1}(I-P)[\gaD   R(\zeta, 0)^*]^*\Theta\gaD P+\mathrm{o}(\tau)_u.\label{7.44.2}
\end{align}
Here and below we write $\mathrm{o}(\tau^\alpha)_u$ to indicate a term which is $\mathrm{o}(\tau^\alpha)$ as $\tau\to0$ uniformly for $\zeta\in\gamma$.
To prove \eqref{rp} we note that $R(\zeta,0)P=(-\zeta)^{-1}P$ by \eqref{decomp} and use \eqref{resolvent}. Similarly, $P[\gaD R(\zeta,0)^*]^*=[\gaD(PR(\zeta,0))^*]^*=(-\zeta)^{-1}[\gaD P]^*$ and \eqref{resolvent} yield \eqref{eq1.63} and
\eqref{prp}. Also, \eqref{7.44.2} follows immediately from \eqref{rp}.

{\em Step 2.}\, We claim the following asymptotic relations for the Riesz projections:
\begin{align}\label{ptp}
P(\tau)P&=P+\tau[\gaD S]^*\Theta\gaD P+\mathrm{o}(\tau)_u,\\
\label{ppt}
PP(\tau)&=P+\tau[\gaD  P]^*\Theta\gaD S+\mathrm{o}(\tau)_u,\\
\label{ppp}
PP(\tau)P&=P-\tau^2[\gaD   P]^*\Theta\gaD[\gaD   (S^2)]^*\Theta\gaD P+\mathrm{o}(\tau^{2})_u.
\end{align}
These follow from \eqref{dfnRPr} and \eqref{decomp1} by applying integration  $-\frac{1}{2\pi i}\int_\gamma(\cdot)\,d\zeta$ to \eqref{rp}, \eqref{eq1.63} and \eqref{prp} respectively.

{\em Step 3.}\, We next claim the following asymptotic relations for the transformation operators defined in 
\eqref{dfnUUinv}:
\begin{align}
U(\tau)&=I+\tau ([\gaD S]^*\Theta\gaD P-[\gaD  P]^*\Theta\gaD S)+\mathrm{o}(\tau)_u,\lb{eq1.72}\\
U(\tau)^{-1}&=I+\tau ([\gaD  P]^*\Theta\gaD S-[\gaD S]^*\Theta\gaD P)+\mathrm{o}(\tau)_u,\label{ut-1}\\
PU(\tau)P&=P-\frac{1}{2}\tau^2[\gaD   P]^*\Theta\gaD[\gaD   (S^2)]^*\Theta\gaD P+\mathrm{o}(\tau^2)_u,\label{pup}\\
\label{u-1}
PU(\tau)^{-1}P&=P-\frac{1}{2}\tau^2[\gaD   P]^*\Theta\gaD[\gaD   (S^2)]^*\Theta\gaD P+\mathrm{o}(\tau^2)_u,\\
PU(\tau)^{-1}(I-P)&=\tau[\gaD  P]^*\Theta\gaD S+\mathrm{o}(\tau)_u.\label{7.44.1}
\end{align}
Indeed, recalling that $D(\tau) = P(\tau) - P$ and using \eqref{ptp} and \eqref{ppt} yields
\begin{align}\label{qq}
D^2(\tau)&=(P(\tau)-P)(P(\tau)-P)=P(\tau)+P-P(\tau)P-PP(\tau)\\
&=P(\tau)-P+(P-P(\tau)P)+(P-PP(\tau))=D(\tau)-\tau P^{(1)}+\mathrm{o}(\tau)_u,
\no
\end{align}
where we have defined $P^{(1)}=[\gaD S]^*\Theta\gaD P+[\gaD  P]^*\Theta\gaD S$. Hence,
\begin{align}\label{qqq}
(I-D(\tau))(D(\tau)-\tau P^{(1)})=\tau  D(\tau)P^{(1)}+\mathrm{o}(\tau)_u=\mathrm{o}(\tau)_u,
\end{align}
and therefore $D(\tau)=\tau  P^{(1)}+(I-D(\tau))^{-1}\mathrm{o}(\tau)_u$, yielding
\begin{align}\label{qqqq}
D(\tau)=\tau  P^{(1)}+\mathrm{o}(\tau)_u.
\end{align}
Since $D(\tau)=O(\tau)_u$, formula \eqref{dfnUUinv} yields
$U(\tau)=I-P-P(\tau)+2P(\tau)P+O(\tau^2)_u$.
Applying \eqref{ptp} and then \eqref{qqqq}, we obtain
\eqref{eq1.72}. Similarly,
$U(\tau)^{-1}=I-P-P(\tau)+2PP(\tau)+O(\tau^2)_u$ yields \eqref{ut-1}.
Formula \eqref{pup} follows from the calculation
\begin{equation*}
\begin{split}
PU(\tau)P&=P(I-D^2(\tau))^{-1/2}P(\tau)P
=PP(\tau)P+\frac{1}{2}PD(\tau)^2P+O(\tau^3)_u\\
&=P-\tau^2[\gaD   P]^*\Theta\gaD[\gaD   (S^2)]^*\Theta\gaD P+\frac{1}{2}\tau^2P(P^{(1)})^2P+\mathrm{o}(\tau^2)_u\\
&=P-\tau^2[\gaD   P]^*\Theta\gaD[\gaD   (S^2)]^*\Theta\gaD P\\&\qquad\qquad+\frac{1}{2}\tau^2([\gaD  P]^*\Theta\gaD S[\gaD S]^*\Theta\gaD P)+\mathrm{o}(\tau^2)_u\\
&=P-\frac{1}{2}\tau^2[\gaD   P]^*\Theta\gaD[\gaD   (S^2)]^*\Theta\gaD P+\mathrm{o}(\tau^2)_u,
\end{split}
\end{equation*}
and a similar argument yields \eqref{u-1}. Also, \eqref{7.44.1} follows using \eqref{ut-1}.

{\em Step 4.}\, We now claim the following asymptotic relations for the resolvent:
\begin{equation}\lb{finstep}
\begin{split}
PU(\tau)&^{-1}R(\zeta, \tau)U(\tau)P=(-\zeta)^{-1}P+(-\zeta)^{-2}\tau[\gaD   P]^*\Theta\gaD P\\&-(-\zeta)^{-2}\tau^2PV(0)P-(-\zeta)^{-1}\tau^2[\gaD   P]^*\Theta\gaD[\gaD   (S^2)]^*\Theta\gaD P\\
&\quad+(-\zeta)^{-2}\tau^2[\gaD   P]^*\Theta\gaD[\gaD   [-(-\zeta)S^2\\&\qquad+R(\zeta, 0)(I+2(-\zeta) S+(-\zeta)^2 S^2)]^*]^*\Theta\gaD  P+\mathrm{o}(\tau^2)_u.
\end{split}
\end{equation}
Indeed, let us consider the operator
\[ PU(\tau)^{-1}R(\zeta, \tau)U(\tau)P=A_1+A_2+A_3+A_4, \]
where we denote
\begin{align*}
\begin{split}
&A_1=PU(\tau)^{-1}PR(\zeta, \tau)PU(\tau)P,\\&
\quad A_2=PU(\tau)^{-1}(I-P)R(\zeta, \tau)PU(\tau)P,\\
&\qquad A_3=PU(\tau)^{-1}PR(\zeta, \tau)(I-P)U(\tau)P,
\\&\quad\qquad
 A_4=PU(\tau)^{-1}(I-P)R(\zeta, \tau)(I-P)U(\tau)P.
\end{split}
\end{align*}
Using \eqref{u-1}, \eqref{prp} and \eqref{pup},  we obtain
\begin{align*}
\begin{split}
A_1=&(P-\frac{1}{2}\tau^2[\gaD   P]^*\Theta\gaD[\gaD   (S^2)]^*\Theta\gaD P+\mathrm{o}(\tau^2)_u)\\
&\times((-\zeta)^{-1}P+(-\zeta)^{-2}\tau[\gaD   P]^*\Theta\gaD P-(-\zeta)^{-2}\tau^2PV(0)P\\&+(-\zeta)^{-2}\tau^2[\gaD   P]^*\Theta\gaD[\gaD   R(\zeta, 0)^*]^*\Theta\gaD  P+\mathrm{o}(\tau^{2})_u)\\
&\times(P-\frac{1}{2}\tau^2[\gaD   P]^*\Theta\gaD[\gaD   (S^2)]^*\Theta\gaD P+\mathrm{o}(\tau^2)_u)\\
=&(-\zeta)^{-1}P+(-\zeta)^{-2}\tau[\gaD   P]^*\Theta\gaD P-(-\zeta)^{-2}\tau^2PV(0)P\\&+(-\zeta)^{-2}\tau^2[\gaD   P]^*\Theta\gaD[\gaD   R(\zeta, 0)^*]^*\Theta\gaD  P\\&-(-\zeta)^{-1}\tau^2[\gaD   P]^*\Theta\gaD[\gaD   (S^2)]^*\Theta\gaD P+\mathrm{o}(\tau^{2})_u.
\end{split}
\end{align*}
Also, it follows from \eqref{7.44.1} and \eqref{7.44.2} that 
\begin{align*}
\begin{split}
A_2=&(\tau[\gaD  P]^*\Theta\gaD S+\mathrm{o}(\tau)_s)
(\tau(-\zeta)^{-1}(I-P)[\gaD   R(\zeta, 0)^*]^*\Theta\gaD P+\mathrm{o}(\tau)_u)\\
&\times(P-\frac{1}{2}\tau^2[\gaD   P]^*\Theta\gaD[\gaD   (S^2)]^*\Theta\gaD P+\mathrm{o}(\tau^2)_u)\\
=&(-\zeta)^{-1}\tau^2[\gaD  P]^*\Theta\gaD S[\gaD   R(\zeta, 0)^*]^*\Theta\gaD P+\mathrm{o}(\tau^2)_u.
\end{split}
\end{align*}
Similarly,
\begin{align*}
\begin{split}
A_3=&PU(\tau)^{-1}PR(\zeta, \tau)(I-P)U(\tau)P\\
=&(P-\frac{1}{2}\tau^2[\gaD   P]^*\Theta\gaD[\gaD   (S^2)]^*\Theta\gaD P+\mathrm{o}(\tau^2)_u)\\
&\times(\tau(-\zeta)^{-1}[\gaD  P]^*\Theta\gaD R(\zeta, 0)+\mathrm{o}(\tau)_s)(\tau[\gaD S]^*\Theta\gaD P+\mathrm{o}(\tau)_s)\\
=&(-\zeta)^{-1}\tau^2[\gaD  P]^*\Theta\gaD R(\zeta, 0)[\gaD S]^*\Theta\gaD P+\mathrm{o}(\tau^2)_u,\\
A_4=&PU(\tau)^{-1}(I-P)R(\zeta, \tau)(I-P)U(\tau)P\\
=&\tau^2[\gaD  P]^*\Theta\gaD SR(\zeta, 0)[\gaD S]^*\Theta\gaD P+\mathrm{o}(\tau^2)_u.
\end{split}
\end{align*}
Collecting all these terms, we obtain \eqref{finstep}:
\begin{align*}
\begin{split}
PU(\tau)^{-1}R&(\zeta, \tau)U(\tau)P\\
=&(-\zeta)^{-1}P+(-\zeta)^{-2}\tau[\gaD   P]^*\Theta\gaD P-(-\zeta)^{-2}\tau^2PV(0)P\\&+(-\zeta)^{-2}\tau^2[\gaD   P]^*\Theta\gaD[\gaD   R(\zeta, 0)^*]^*\Theta\gaD  P\\&-(-\zeta)^{-1}\tau^2[\gaD   P]^*\Theta\gaD[\gaD   (S^2)]^*\Theta\gaD P+\mathrm{o}(\tau^{2})\\
&+(-\zeta)^{-1}\tau^2[\gaD  P]^*\Theta\gaD S[\gaD   R(\zeta, 0)^*]^*\Theta\gaD P+\mathrm{o}(\tau^2)_u\\
&+(-\zeta)^{-1}\tau^2[\gaD  P]^*\Theta\gaD R(\zeta, 0)[\gaD S]^*\Theta\gaD P+\mathrm{o}(\tau^2)_u\\
&+\tau^2[\gaD  P]^*\Theta\gaD SR(\zeta, 0)[\gaD S]^*\Theta\gaD P+\mathrm{o}(\tau^2)_u\\
=&(-\zeta)^{-1}P+(-\zeta)^{-2}\tau[\gaD   P]^*\Theta\gaD P-(-\zeta)^{-2}\tau^2PV(0)P\\&-(-\zeta)^{-1}\tau^2[\gaD   P]^*\Theta\gaD[\gaD   (S^2)]^*\Theta\gaD P\\
&+(-\zeta)^{-2}\tau^2[\gaD   P]^*\Theta\gaD[\gaD   [-(-\zeta)S^2\\&+R(\zeta, 0)(I+2(-\zeta) S+(-\zeta)^2 S^2)]^*]^*\Theta\gaD  P+\mathrm{o}(\tau^2)_u.
\end{split}
\end{align*}
{\em Step 5.}\, We are ready to finish the proof of the lemma. Using the standard relation from  \cite[Equation (III.6.24)]{Kato} we have
\[L_{0,\cG}(\tau)P(\tau)=-\frac{1}{2\pi i}\int_\gamma \zeta R(\zeta,\tau)\,d\zeta,\]
 and applying integration $-\frac{1}{2\pi i}\int_\gamma\zeta(\cdot)\,d\zeta$ in \eqref{finstep} yields \eqref{pulup}.
\end{proof}
\subsection{The Neumann-based case}
 We are ready to present the proof of Theorem \ref{tim:Nbased}, relating the Morse index of the operator $L_{\cG}$ with a Neumann-based subspace $\cG$, the Maslov index of the path $\Upsilon|_{\Sigma_3}$, the Morse index of a matrix associated with the boundary operator $\Theta$, and the Morse index of the potential matrix $V(0)$. 
 
 \begin{proof}[Proof of Theorem \ref{tim:Nbased}]  We begin exactly as in the proof of Theorem \ref{tim:Dbased}:
 since the boundary $\Gamma$ of the square $[-\Lambda,0]\times[\tau,1]$ can be contracted to a point, the Maslov index of $\Upsilon|_\Sigma$ is equal to zero.  Using Lemmas \ref{lem:Gamma4} and \ref{lambdaMon}, we arrive at the identity 
 $\mor(L_{\cG})=-\mas(\Upsilon|_{\Sigma_2})+\mor(L_{0,\cG}(\tau))$.
 So, it remains to show that under the assumptions in the theorem 
 one has the identity \[\mor(L_{0,\cG}(\tau))=\mor(-B)+\mor(Q_0V(0)Q_0)\,\text{ for $\tau$ near zero.}\] By Remark \ref{rem:LPt}, it suffices to count the negative eigenvalues of the finite-dimensional operator $L_{0,\cG}(\tau)P(\tau)$. By Lemma \ref{lem:simile}, it is enough to obtain an asymptotic formula for the eigenvalues of the operator $T(\tau)=
PU(\tau)^{-1}L_{0,\cG}(\tau)P(\tau)U(\tau)P$, where
\[
	T(\tau)=T+\tau T^{(1)}+\tau^2T^{(2)}+\mathrm{o}(\tau^2)
	\text{ as $\tau\to0$}
\]
and we denote
\begin{align}
T&=0,\quad T^{(1)}=-[\gaD   P]^*\Theta\gaD P,\quad 
T^{(2)}=T^{(2)}_1+T^{(2)}_2,\nonumber\\
T^{(2)}_1&=PV(0)P,\quad T^{(2)}_2=-[\gaD   P]^*\Theta\gaD
[\gaD   S]^*\Theta\gaD.\lb{dfnT22}\end{align}
These operators act on the $N$-dimensional space $\ran(P)=\ker(L_{0,\cG}(0))$ consisting of constant vector-valued functions on $\Om$. We will apply a well known finite-dimensional perturbation result \cite[Theorem II.5.1]{Kato} to the family $\{T(\tau)\}$ for $\tau$ near zero. For this we will need some more notations and preliminaries.

Let $\lambda^{(1)}_j$ denote the eigenvalues of the 
operator $T^{(1)}$, let $m_j$ denote their multiplicities, and let $P^{(1)}_j$ denote the respective orthogonal Riesz spectral projections, $j=1,\dots,m$. We recall the notation $\mathfrak{b}(p,q)=\langle\Theta p,q\rangle_{1/2}$ for the form on the $N$-dimensional subspace $\cS$ in $H^{1/2}(\dOm)$ consisting of the constant (that is, $x$-independent) vector-valued functions $p,q$ on $\dOm$. We also recall that $B$ is the operator on $\cS$ associated with the form $\mathfrak{b}$ so that $\langle Bp,q\rangle_\cS=\mathfrak{b}(p,q)$. We note that the Dirichlet trace $\gaD$ maps $\ran(P)$ onto $\cS$ isomorphically, and use the same notation $p,q$ for the elements of $\ran(P)$. The quadratic form on $\ran(P)$ associated with $T^{(1)}$ is given by 
\begin{align*}
\langle T^{(1)}p,q\rangle_{\ran(P)}&=-\langle [\gaD   P]^*\Theta\gaD Pp,q\rangle_{\ran(P)}\\&=_{H^{-1/2}}\langle\Theta \gaD Pp,\gaD Pq\rangle_{1/2}=-\langle\Theta p,q\rangle_{1/2}
=-\mathfrak{b}(p,q).
\end{align*}
Thus, the number of positive eigenvalues among $\lambda^{(1)}_j$ is given by $n_-(\mathfrak{b})$, the number of negative eigenvalues among $\lambda^{(1)}_j$ is given by $n_+(\mathfrak{b})$, and number of zero eigenvalues among $\lambda^{(1)}_j$ is given by $n_0(\mathfrak{b})$. Let us suppose that $\{\lambda^{(1)}_j\}_{j=0}^{n_0(\mathfrak{b})}$ are the zero eigenvalues. For the corresponding spectral projections $P^{(1)}_j$ one then has the identity $\ker(B)=\gaD(\bigoplus_{j=1}^{n_0(\mathfrak{b})}\ran(P^{(1)}_j))$.
For the operator $T^{(2)}_2$ from \eqref{dfnT22} and the spectral projections corresponding to the zero eigenvalues one has
\beq\lb{zeroev}
T^{(2)}_2P^{(1)}_j=0,\quad j=1,\dots,n_0(\mathfrak{b}),
\enq
because $\Theta\gaD Pp=\Theta p=0$ provided $p\in\bigoplus_{j=1}^{n_0(\mathfrak{b})} \ran(P^{(1)}_j)$.

Following \cite[Section II.5]{Kato}, we let $\lambda^{(2)}_{jk}$, $k=1,\dots,m_j$, $j=1,\dots,m$, denote the eigenvalues of the operator 
$P^{(1)}_jT^{(2)}P^{(1)}_j$ in $\ran(P^{(1)}_j)$ (recall that in our case the unperturbed operator is just $T=0$ and thus its reduced resolvent is zero and $P^{(1)}_j\widetilde{T}^{(2)}P^{(1)}_j=P^{(1)}_jT^{(2)}P^{(1)}_j$ using the notations from \cite[Section II.5]{Kato}). Taking into account \eqref{zeroev}, we observe that the eigenvalues 
$\lambda^{(2)}_{jk}$ for $j=1,\dots,n_0(\mathfrak{b})$ corresponding to the zero eigenvalues of $T^{(1)}$ are in fact eigenvalues of the operator $P^{(1)}_jT^{(2)}_1P^{(1)}_j$. By the assumption in the theorem we know that the form
\[\mathfrak{v}(p,q)=\langle V(0)p,q\rangle_\cS
\,\text{ on }\, \ran(Q_0)= \ker(B)=\gaD
\Big(\bigoplus_{j=1}^{n_0(\mathfrak{b})}\ran(P^{(1)}_j) \Big)\]
is nondegenerate. Thus, the number of negative, respectively, positive
eigenvalues among $\lambda^{(2)}_{jk}$, $k=1,\dots,m_j$, $j=1,\dots,n_0(\mathfrak{b})$, is equal to $n_-(\mathfrak{v})$, respectively, $n_+(\mathfrak{v})$.

By \cite[Theorem II.5.11]{Kato} the eigenvalues $\lambda_{jk}(\tau)$ of the operator $T(\tau)$ are given by the formula
\beq\lb{eq1.84}
\lambda_{jk}(\tau)=\tau  \lambda^{(1)}_j+\tau^2\lambda^{(2)}_{jk}+\mathrm{o}(\tau^2)
\,\text{ as $\tau\to0$}, \, k=1,\dots,m_j,\,j=1,\dots,m.
\enq
This implies that the number of the negative eigenvalues of the operator  $L_{0,\cG}(\tau)P(\tau)$ for $\tau$ near zero is equal to $n_+(\mathfrak{b})+n_-(\mathfrak{v})$, and thus completes the proof of the theorem.
 \end{proof}


\end{document}